\documentclass[reqno]{amsart}

\usepackage{amsmath,amsfonts,amssymb,amscd,verbatim,delarray,fullpage,float}

\usepackage{tikz}
\usepackage{tikz-cd}
\usetikzlibrary{matrix,arrows}
\usepackage{stmaryrd}
\usepackage{amsthm}
\usepackage{url}
\usepackage[polutonikogreek,english]{babel}

\DeclareSymbolFont{bbold}{U}{bbold}{m}{n}
\DeclareSymbolFontAlphabet{\mathbbold}{bbold}

\DeclareMathOperator{\pr}{pr}
\DeclareMathOperator{\ad}{adm.}
\DeclareMathOperator{\SF}{sf}
\DeclareMathOperator{\cyc}{cyc}
\DeclareMathOperator{\can}{can}
\DeclareMathOperator{\disc}{disc}

\DeclareMathOperator{\level}{level}

\DeclareMathOperator{\conj}{Conj}

\DeclareMathOperator{\ns}{ns}
\DeclareMathOperator{\s}{s}
\DeclareMathOperator{\aut}{Aut}

\DeclareMathOperator{\frob}{Frob}
\DeclareMathOperator{\Char}{char}
\DeclareMathOperator{\tors}{tors}
\DeclareMathOperator{\GL}{GL}
\DeclareMathOperator{\SL}{SL}
\DeclareMathOperator{\PSL}{PSL}
\DeclareMathOperator{\PGL}{PGL}
 
 \DeclareMathOperator{\non}{non-}

\DeclareMathOperator{\tr}{tr}

\chardef\bslash=`\\ 

\begin{document}

\restylefloat{table}

\newtheorem{Theorem}{Theorem}[section]

\newtheorem{example}[Theorem]{Example}
\newtheorem{cor}[Theorem]{Corollary}
\newtheorem{goal}[Theorem]{Goal}

\newtheorem{Conjecture}[Theorem]{Conjecture}
\newtheorem{guess}[Theorem]{Guess}

\newtheorem{exercise}[Theorem]{Exercise}
\newtheorem{Question}[Theorem]{Question}
\newtheorem{lemma}[Theorem]{Lemma}
\newtheorem{property}[Theorem]{Property}
\newtheorem{proposition}[Theorem]{Proposition}
\newtheorem{ax}[Theorem]{Axiom}
\newtheorem{claim}[Theorem]{Claim}

\newtheorem{nTheorem}{Surjectivity Theorem}

\theoremstyle{definition}
\newtheorem{Definition}[Theorem]{Definition}
\newtheorem{problem}[Theorem]{Problem}
\newtheorem{question}[Theorem]{Question}
\newtheorem{Example}[Theorem]{Example}

\newtheorem{remark}[Theorem]{Remark}
\newtheorem{diagram}{Diagram}
\newtheorem{Remark}[Theorem]{Remark}
\newcommand{\diagref}[1]{diagram~\ref{#1}}
\newcommand{\thmref}[1]{Theorem~\ref{#1}}
\newcommand{\secref}[1]{Section~\ref{#1}}
\newcommand{\subsecref}[1]{Subsection~\ref{#1}}
\newcommand{\lemref}[1]{Lemma~\ref{#1}}
\newcommand{\corref}[1]{Corollary~\ref{#1}}
\newcommand{\exampref}[1]{Example~\ref{#1}}
\newcommand{\remarkref}[1]{Remark~\ref{#1}}
\newcommand{\corlref}[1]{Corollary~\ref{#1}}
\newcommand{\claimref}[1]{Claim~\ref{#1}}
\newcommand{\defnref}[1]{Definition~\ref{#1}}
\newcommand{\propref}[1]{Proposition~\ref{#1}}
\newcommand{\prref}[1]{Property~\ref{#1}}
\newcommand{\itemref}[1]{(\ref{#1})}
\newcommand{\ul}[1]{\underline{#1}}


\newcommand{\CE}{\mathcal{E}}
\newcommand{\CG}{\mathcal{G}}\newcommand{\CV}{\mathcal{V}}
\newcommand{\CL}{\mathcal{L}}
\newcommand{\CM}{\mathcal{M}}
\newcommand{\A}{\mathcal{A}}
\newcommand{\CO}{\mathcal{O}}
\newcommand{\B}{\mathcal{B}}
\newcommand{\CS}{\mathcal{S}}
\newcommand{\CX}{\mathcal{X}}
\newcommand{\CY}{\mathcal{Y}}
\newcommand{\CT}{\mathcal{T}}
\newcommand{\CW}{\mathcal{W}}
\newcommand{\CJ}{\mathcal{J}}

\newcommand\myeq{\mathrel{\stackrel{\makebox[0pt]{\mbox{\normalfont\tiny def}}}
{\Longleftrightarrow}}}
\newcommand{\st}{\sigma}
\renewcommand{\k}{\varkappa}
\newcommand{\Frac}{\mbox{Frac}}
\newcommand{\XC}{\mathcal{X}}
\newcommand{\wt}{\widetilde}
\newcommand{\wh}{\widehat}
\newcommand{\mk}{\medskip}
\renewcommand{\sectionmark}[1]{}
\renewcommand{\Im}{\operatorname{Im}}
\renewcommand{\Re}{\operatorname{Re}}
\newcommand{\la}{\langle}
\newcommand{\ra}{\rangle}
\newcommand{\LND}{\mbox{LND}}
\newcommand{\Pic}{\mbox{Pic}}
\newcommand{\lnd}{\mbox{lnd}}
\newcommand{\GLND}{\mbox{GLND}}\newcommand{\glnd}{\mbox{glnd}}
\newcommand{\Der}{\mbox{DER}}\newcommand{\DER}{\mbox{DER}}
\renewcommand{\th}{\theta}
\newcommand{\ve}{\varepsilon}
\newcommand{\1}{^{-1}}
\newcommand{\iy}{\infty}
\newcommand{\iintl}{\iint\limits}
\newcommand{\capl}{\operatornamewithlimits{\bigcap}\limits}
\newcommand{\cupl}{\operatornamewithlimits{\bigcup}\limits}
\newcommand{\suml}{\sum\limits}
\newcommand{\ord}{\operatorname{ord}}
\newcommand{\gal}{\operatorname{Gal}}
\newcommand{\bk}{\bigskip}
\newcommand{\fc}{\frac}
\newcommand{\g}{\gamma}
\newcommand{\be}{\beta}
\newcommand{\dl}{\delta}
\newcommand{\Dl}{\Delta}
\newcommand{\lm}{\lambda}
\newcommand{\Lm}{\Lambda}
\newcommand{\om}{\omega}
\newcommand{\ov}{\overline}
\newcommand{\vp}{\varphi}
\newcommand{\kap}{\varkappa}

\newcommand{\Vp}{\Phi}
\newcommand{\Varphi}{\Phi}
\newcommand{\BC}{\mathbb{C}}
\newcommand{\C}{\mathbb{C}}\newcommand{\BP}{\mathbb{P}}
\newcommand{\BQ}{\mathbb {Q}}
\newcommand{\BM}{\mathbb{M}}
\newcommand{\mbh}{\mathbb{H}}
\newcommand{\BR}{\mathbb{R}}\newcommand{\BN}{\mathbb{N}}
\newcommand{\BZ}{\mathbb{Z}}\newcommand{\BF}{\mathbb{F}}
\newcommand{\BA}{\mathbb {A}}
\renewcommand{\Im}{\operatorname{Im}}
\newcommand{\idd}{\operatorname{id}}
\newcommand{\ep}{\epsilon}
\newcommand{\tp}{\tilde\partial}
\newcommand{\doe}{\overset{\text{def}}{=}}
\newcommand{\supp} {\operatorname{supp}}
\newcommand{\loc} {\operatorname{loc}}
\newcommand{\de}{\partial}
\newcommand{\z}{\zeta}
\renewcommand{\a}{\alpha}
\newcommand{\G}{\Gamma}
\newcommand{\der}{\mbox{DER}}

\newcommand{\Spec}{\operatorname{Spec}}
\newcommand{\Sym}{\operatorname{Sym}}
\newcommand{\Aut}{\operatorname{Aut}}

\newcommand{\Idd}{\operatorname{Id}}

\newcommand{\tG}{\widetilde G}
\newcommand{\F}{\mathbb{F}}
\newcommand{\Q}{\mathbb{Q}}
\newcommand{\Z}{\mathbb{Z}}
\newcommand{\XG}{\text{N}_{s}(5)'}
\newcommand{\tB}{\text{B}}
\newcommand{\Gal}{\text{Gal}}
\newcommand{\cX}{\mathcal{X}}
\newcommand{\Inn}{\text{Inn}}
\newcommand{\bP}{\mathbf{P}}
\newcommand{\FX}{\mathfrac {X}}
\newcommand{\FV}{\mathfrac {V}}
\newcommand{\SX}{\mathcal {X}}
\newcommand{\SV}{\mathcal {V}}
\newcommand{\SO}{\mathcal {O}}
\newcommand{\SD}{\mathcal {D}}
\newcommand{\Sr}{\rho}
\newcommand{\SR}{\mathcal {R}}
\newcommand{\cl}{\mathcal{C}}
\newcommand{\ok}{\mathcal{O}_K}
\newcommand{\ab}{\mathcal{AB}}

\setcounter{equation}{0} \setcounter{section}{0}

\newcommand{\ds}{\displaystyle}
\newcommand{\gl}{\lambda}
\newcommand{\gL}{\Lambda}
\newcommand{\gge}{\epsilon}
\newcommand{\gG}{\Gamma}
\newcommand{\ga}{\alpha}
\newcommand{\gb}{\beta}
\newcommand{\gd}{\delta}
\newcommand{\gD}{\Delta}
\newcommand{\gs}{\sigma}
\newcommand{\mbq}{\mathbb{Q}}
\newcommand{\mbr}{\mathbb{R}}
\newcommand{\mbz}{\mathbb{Z}}
\newcommand{\mbc}{\mathbb{C}}
\newcommand{\mbn}{\mathbb{N}}
\newcommand{\mbp}{\mathbb{P}}
\newcommand{\mbf}{\mathbb{F}}
\newcommand{\mbe}{\mathbb{E}}
\newcommand{\lcm}{\text{lcm}\,}
\newcommand{\mf}[1]{\mathfrak{#1}}
\newcommand{\ol}[1]{\overline{#1}}
\newcommand{\mc}[1]{\mathcal{#1}}
\newcommand{\mb}[1]{\mathbb{#1}}
\newcommand{\nequiv}{\equiv\hspace{-.07in}/\;}
\newcommand{\bnequiv}{\equiv\hspace{-.13in}/\;}

\renewcommand{\thefootnote}{\fnsymbol{footnote}} 
\renewcommand{\thefootnote}{\arabic{footnote}}

\renewcommand{\thefootnote}{\fnsymbol{footnote}} 
\renewcommand{\thefootnote}{\arabic{footnote}} 

\title{Elliptic curves with missing Frobenius traces}
\author{Nathan Jones and Kevin Vissuet}


\date{}

\begin{abstract}
Let $E$ be an elliptic curve defined over $\mbq$.  In 1976, Lang and Trotter conjectured an asymptotic formula for the number $\pi_{E,r}(X)$ of primes $p \leq X$ of good reduction for which the Frobenius trace at $p$ associated to $E$ is equal to a given fixed integer $r$.  We investigate elliptic curves $E$  over $\mbq$ that have a missing Frobenius trace, i.e. for which the counting function $\pi_{E,r}(X)$ remains bounded as $X \rightarrow \infty$, for some $r \in \mbz$.  In particular, we classify all elliptic curves $E$ over $\mbq(t)$ that have a missing Frobenius trace.
\end{abstract}

\maketitle

\section{Introduction and statement of results} \label{introduction}

Let $E$ be an elliptic curve over $\mbq$ of conductor $N_E$.  For a prime $p$ not dividing $N_E$, we consider the Frobenius trace $a_p(E) \in \mbz$ associated to $p$, which satisfies the equation
\[
\#E(\mbf_p) = p + 1 - a_p(E).
\]
The following conjecture, formulated by S. Lang and H. Trotter in 1976, articulates one distributional aspect of the infinite sequence $\left( a_p(E) : p \nmid N_E \right)$.  Specifically, it states a precise asymptotic formula for the counting function
\begin{equation} \label{defofpisubErofX}
\pi_{E,r}(X) := \# \{ p \leq X : p \nmid N_E, \, a_p(E) = r \}.
\end{equation}
\begin{Conjecture} \label{LTConjecture}
Let $E$ be an elliptic curve over $\mbq$ without complex multiplication, let $r \in \mbz$ and define the quantity $\pi_{E,r}(X)$ by \eqref{defofpisubErofX}.  There exists a constant $C_{E,r} \geq 0$ so that, as $X \rightarrow \infty$, we have
\begin{equation} \label{LTasymptoticfromconj}
\pi_{E,r}(X) \sim C_{E,r} \frac{\sqrt{X}}{\log X}.
\end{equation}
\end{Conjecture}
In case $C_{E,r} = 0$, we interpret the asymptotic of Conjecture \ref{LTConjecture} as
\begin{equation*} 
\begin{matrix} \pi_{E,r}(X) \sim 0 \\ \text{ as } X\rightarrow \infty \end{matrix} \; \myeq \; \lim_{X \rightarrow \infty} \pi_{E,r}(X) < \infty.
\end{equation*}

Conjecture \ref{LTConjecture} was developed using a probabilistic model based upon the Sato-Tate conjecture and the Chebotarev density theorem for division fields of $E$.  In spite of the Sato-Tate conjecture having been proved (see \cite{clozeletal} and \cite{taylor}), Conjecture \ref{LTConjecture} remains open, even assuming the Generalized Riemann Hypothesis (GRH).  In contrast, analogues of Conjecture \ref{LTConjecture} have been proven in the function field case (see \cite{katz}, which also gives the statement of Conjecture \ref{LTConjecture} over a general number field).  The best known upper bounds for $\pi_{E,r}(X)$ are
\[
\pi_{E,0}(X) \ll 
\begin{cases}
\frac{X^{3/4}}{(\log X)^{1/2}} & \text{ assuming GRH (see \cite{zywinascholar}, which builds upon \cite{serre1} and \cite{murtymurtysaradha}),} \\
X^{3/4} & \text{ unconditionally (see \cite{elkies1}),}
\end{cases}
\]
and, for $r \neq 0$,
\[
\pi_{E,r}(X) \ll 
\begin{cases}
\frac{X^{4/5}}{(\log X)^{3/5}} & \text{ assuming GRH, (see \cite{zywinascholar}, \cite{serre1} and \cite{murtymurtysaradha})} \\
\frac{X (\log \log X)^2}{(\log X)^2} & \text{ unconditionally (see \cite{thornterzaman}).}
\end{cases}
\]
\sloppy The only known non-trivial lower bound is due to Elkies \cite{elkies2} in the case $r=0$, who proved that $\ds \lim_{X \rightarrow \infty} \pi_{E,0}(X) = \infty$.  In fact, his work gives rise to the lower bounds
\[
\pi_{E,0}(X) \geq
\begin{cases}
\log \log X & \text{ assuming the GRH} \\
\frac{ \log \log \log X }{ (\log \log \log \log X )^{1 + \gd}} & \text{ $\forall \gd > 0$ and $x \gg_\gd 1$, unconditionally}
\end{cases}
\]
(see \cite{fouvrymurty} for the unconditional lower bound).  Conjecture \ref{LTConjecture} has also been proven to hold ``on average over $E$,'' i.e. if we average the counting function $\pi_{E,r}(X)$ over the set of all elliptic curves of height bounded by a function that grows appropriately with $X$, we obtain an asymptotic formula of the same order of magnitude as in \eqref{LTasymptoticfromconj} (but with a slightly different constant $C_r$ that is independent of $E$).  See \cite{fouvrymurty} for the case $r=0$, \cite{davidpappalardi} for $r \neq 0$, and \cite{baier} for the same result averaged over a smaller set of elliptic curves; see \cite{jones1} and also \cite{davidkoukoulopolissmith} for an analysis relating the average constant $C_r$ with the conjectural constants $C_{E,r}$.  Finally, see \cite{kevinjames} for a brief survey of some interesting variants of Conjecture \ref{LTConjecture}.

The present paper will focus on elliptic curves $E$ for which $C_{E,r} = 0$ for some $r \in \mbz$.  An explicit formula (see \eqref{explicitformofCEr} below) reveals that $C_{E,r} = 0$ can only happen when there is a \emph{congruence obstruction}, in which case it follows from the Chebotarev density theorem that
\[
C_{E,r} = 0 \; \Longrightarrow \; \{ p \text{ prime } : p \nmid N_E, \, a_p(E) = r \} \subseteq \{ p \text{ prime } : p \mid m \},
\]
for an appropriately chosen $m \in \mbn$.  In particular, if $C_{E,r} = 0$, we \emph{provably} have  $\displaystyle \lim_{X \rightarrow \infty} \pi_{E,r}(X) < \infty$. When this is the case, we will call $r$ a \textbf{\emph{missing Frobenius trace for $E$}}.

In this paper, we consider elliptic curves $E$ over $\mbq$ that have a missing Frobenius trace, the broad goal being to classify all such elliptic curves.  Here are a few examples.
\begin{example} \label{level3example}
Consider the elliptic curve $E_3$ over $\mbq$ defined by the Weierstrass equation
\[
E_3 : \; y^2 + xy + y = x^3 - x^2 - 56x + 163.
\]
The finite sequence $\left( a_p(E_{3}) \mod 3 :  p \leq 150 \text{ and } p \nmid N_{E_{3}} \right)$ is equal to
\[
(0, 2, 0, 2, 0, 2, 0, 0, 2, 2, 0, 2, 0, 0, 0, 2, 2, 0, 2, 2, 0, 0, 2, 0, 2, 0, 
2, 0, 2, 0, 0, 2, 0).
\]
We see that the residue class $1 \mod 3$ is missing from this list.  This is caused by the presence of a rational $3$-torsion point $P := (7,5) \in E_3[3]$.  Examining the effect of $P$ on the action of $\gal(\ol{\mbq}/\mbq)$ on $E_3[3]$, it follows that every $r \in \mbz$ satisfying $r \equiv 1 \mod 3$ is a missing Frobenius trace for $E_3$.
\end{example}

\begin{example} \label{level8example}
Consider the elliptic curve $E_{6}$ over $\mbq$ defined by the Weierstrass equation
\[
E_{6} : \; y^2 = x^3 - 15876x - 777924.
\]
The finite sequence $\left( a_p(E_{6}) \mod 6 :  p \leq 150 \text{ and } p \nmid N_{E_{6}} \right)$ is equal to
\[
\begin{split}
&( 4, 4, 5, 2, 5, 4, 0, 5, 0, 0, 5, 0, 2, 2, 5, 1, 2, 5, 5, 4, 0, 1, 0, 1, 0, 1, 
0, 5, 4, 0, 0, 4 ).
\end{split}
\]
We see that the residue class $3 \mod 6$ is missing from this list.  As we shall see, due to the nature of the action of $\gal(\ol{\mbq}/\mbq)$ on the $6$-torsion of $E_{6}$, every $r \in \mbz$ satisfying $r \equiv 3 \mod 6$ is a missing Frobenius trace for $E_6$.  Furthermore, $m=6$ is the smallest level for which $E_{6}$ has a missing trace modulo $m$.
\end{example}

\begin{example} \label{level28example}
Consider the elliptic curve $E_{28}$ over $\mbq$ defined by the Weierstrass equation
\[
E_{28} : \; y^2 = x^3 - 7138223372x + 232131092574192.
\]
The finite sequence $\left( a_p(E_{28}) \mod 28 :  p \leq 580 \text{ and } p \nmid N_{E_{28}} \right)$ is equal to
\[
\begin{split}
&( 0, 1, 2, 26, 25, 22, 24, 19, 10, 9, 26, 22, 2, 24, 1, 10, 6, 10, 10, 16, 21, 
21, 25, 22, 18, 23, 6, 20, 0, 19, 16, 11, 4, 17, \\
&6, 16, 21, 16, 5, 24, 19, 15, 
10, 26, 0, 14, 1, 3, 6, 14, 14, 21, 4, 18, 14, 3, 27, 14, 5, 14, 18, 21, 27, 20,
27, 16, 9, 25, \\
&24, 0, 11, 22, 6, 3, 13, 10, 25, 2, 19, 18, 21, 20, 4, 6, 3, 2, 
6, 20, 4, 12, 26, 18, 26, 21, 14, 8, 11, 26, 23, 4, 3, 16, 18).
\end{split}
\]
We see that the residue class $7 \mod 28$ is missing from this list.  As we shall see, due to the nature of the action of $\gal(\ol{\mbq}/\mbq)$ on the $28$-torsion of $E_{28}$, every $r \in \mbz$ satisfying $r \equiv 7 \mod 28$ is a missing Frobenius trace for $E_{28}$.  Furthermore, $m=28$ is the smallest level for which $E_{28}$ has a missing trace modulo $m$.
\end{example}

Of the above three examples, the elliptic curve $E_3$ has a congruence obstruction at the prime level $3$, whereas $E_6$ and $E_{28}$ have obstructions at composite levels.  Examples like $E_3$ are handled in \cite{davidkisilevskypappalardi}, which classifies elliptic curves $E$ over $\mbq$ satisfying $C_{E,r} = 0$ and admitting a rational isogeny of prime degree.  On the other hand, for $E_6$ and $E_{28}$, the congruence obstruction is caused by nontrivial \emph{entanglements} (e.g. $\mbq(E_6[2]) \cap \mbq(E_6[3]) \neq \mbq$); the main contribution of the present paper lies in classifying the remaining (genus $0$) cases where $C_{E,r} = 0$ due to \emph{composite} level congruence obstructions, and in clarifying the role of entanglements in those congruence obstructions.

Towards the goal of describing explicitly the constant $C_{E,r}$, we now consider the continuous Galois representations
\[
\begin{split}
\rho_{E,m} : \; &G_\mbq \longrightarrow \GL_2(\mbz/m\mbz), \\
\rho_E : \; &G_\mbq \longrightarrow \GL_2(\hat{\mbz}),
\end{split}
\]
where $\rho_{E,m}$ is defined by letting $G_\mbq := \gal(\ol{\mbq}/\mbq)$ act on $E[m]$ and fixing a $\mbz/m\mbz$-basis thereof, and $\rho_E$ is likewise defined by letting $G_\mbq$ act on the entire torsion subgroup $\ds E_{\tors} := \bigcup_{m=1}^\infty E[m]$ of $E$ and choosing a $\mbz/m\mbz$-basis of each $E[m]$ in a compatible manner.  Here,
\[
\hat{\mbz} := \lim_{ \leftarrow } \mbz/m\mbz \simeq \prod_{\ell \text{ prime}} \mbz_{\ell},
\]
is the inverse limit of the projective system $\{ \mbz/m\mbz : m \in \mbn \}$, ordered according to divisibility and with the canonical projection maps.  We may likewise view $\rho_E$ as being the inverse limit of the system of representations $\rho_{E,m}$, with $m \in \mbn$.  A famous theorem due to Serre \cite{serre} states that, if $E$ has no complex multiplication, then $\rho_E(G_\mbq) \subseteq \GL_2(\hat{\mbz})$ is an \emph{open} subgroup, or, equivalently, that the index of $\rho_E(G_{\mbq})$ in $\GL_2(\hat{\mbz})$ is finite.  Consequently, there is a positive integer $m$ for which
\begin{equation*} 
\ker \left( \GL_2(\hat{\mbz}) \rightarrow \GL_2(\mbz/m\mbz) \right) \subseteq \rho_E(G_\mbq).
\end{equation*}
We define $m_E \in \mbn$ to be the smallest positive integer $m$ for which this holds.  

As described in detail in \cite{langtrotter}, the constant $C_{E,r}$ appearing in Conjecture \ref{LTConjecture} is given by
\begin{equation} \label{explicitformofCEr}
C_{E,r} = \frac{2}{\pi} \cdot \frac{m_E | \rho_{E,m_E}(G_\mbq)_r |}{| \rho_{E,m_E}(G_\mbq) |} \prod_{{\begin{substack} {\ell \text{ prime} \\ \ell \nmid m_E} \end{substack}}} \frac{\ell | \GL_2(\mbz/\ell\mbz)_r  |}{| \GL_2(\mbz/\ell\mbz) |},
\end{equation}
where, for any subgroup $H \subseteq \GL_2(\mbz/m\mbz)$, we are employing the notation
\[
H_r := \{ g \in H : \tr g \equiv r \mod m \}.
\]
Furthermore, it can be verified by direct computation that the infinite product over primes $\ell \nmid m_E$ in \eqref{explicitformofCEr} is \emph{convergent}, and a straightforward computation shows that each $\ell$-th factor $ \frac{\ell | \GL_2(\mbz/\ell\mbz)_r  |}{| \GL_2(\mbz/\ell\mbz) |}$ is nonzero for any $r \in \mbz$.  It follows that, for any elliptic curve $E$ over $\mbq$ without complex multiplication, we have
\[
C_{E,r} = 0 \; \Longleftrightarrow \; \exists m \mid m_E \text{ for which } \rho_{E,m}(G_\mbq)_r = \emptyset.
\]
To find elliptic curves $E$ with missing Frobenius traces (i.e. which satisfy $C_{E,r} = 0$ for some $r \in \mbz$), we are thus led to associate such elliptic curves $E$ with points on a modular curve of level $m$.  Specifically, fix a subgroup $G \subseteq \GL_2(\mbz/m\mbz)$ satisfying 
\begin{equation} \label{conditiononG}
\exists r \in\mbz \; \text{ for which } \; G_r = \emptyset.
\end{equation}
For such a group $G$, let $\tilde{G} := \langle G, -I \rangle$, and consider the modular curve $X_{\tilde{G}}$, whose non-cuspidal, non-CM rational points correspond to $j$-invariants of elliptic curves $E$ with $\rho_{E,m}(G_\mbq) \subseteq \tilde{G}$, up to conjugation inside $\GL_2(\mbz/m\mbz)$ (for more details, see \cite{delignerapoport}).  Our main theorem 
focuses on the case where $X_{\tilde{G}}$ has genus zero (for a recent classification of all genus zero modular curves $X_G$ over $\mbq$ for which $X_G(\mbq) \neq \emptyset$, see \cite{rakvi}).  Since our goal is to classify all elliptic curves $E$ such that $C_{E,r} = 0$ for some $r \in \mbz$, we may as well consider only \emph{maximal} subgroups $G \subseteq \GL_2(\mbz/m\mbz)$ among those satisfying \eqref{conditiononG}.  Furthermore, because we will be varying the level $m$, we will phrase our definitions in terms of open subgroups $G \subseteq \GL_2(\hat{\mbz})$.
For any open subgroup $G \subseteq \GL_2(\hat{\mbz})$, we denote by $m_G$ its \emph{level}, i.e. the smallest $m \in \mbn$ for which $\ker\left(\GL_2(\hat{\mbz}) \rightarrow \GL_2(\mbz/m\mbz) \right) \subseteq G$, and for any $m \in \mbn$ we define 
\[
G(m) := G \bmod m \subseteq \GL_2(\mbz/m\mbz).  
\]
We extend our notation for the associated modular curve by setting $\tilde{G} := \langle G, -I \rangle$ and setting the notation
\[
X_{\tilde{G}} := X_{\tilde{G}(m_{\tilde{G}})}.
\]
Furthermore, we denote by $j_{\tilde{G}} : X_{\tilde{G}} \longrightarrow X(1)$ the map which associates to any point in $X_{\tilde{G}}$ the underlying elliptic curve $E$.  
Finally, since we are only interested in subgroups $G$ up to conjugation inside $\GL_2(\hat{\mbz})$, we define the following notation, for subgroups $G_1, G_2, G \subseteq \GL_2(\hat{\mbz})$ and any integer $r$:
\begin{equation*} 
\begin{split}
G_1 \doteq G_2 \; &\myeq \; \exists g \in \GL_2(\hat{\mbz}) \text{ with } G_1 = g G_2 g^{-1}, \\
G_1 \, \dot\subseteq \, G_2 \; &\myeq \; \exists g \in \GL_2(\hat{\mbz}) \text{ with } G_1 \subseteq g G_2 g^{-1}, \\
G_r &:= \{ g \in G : \tr g \equiv r \bmod m_G \}.
\end{split}
\end{equation*}
We consider the following collections of open subgroups $G \subseteq \GL_2(\hat{\mbz})$:
\begin{equation} \label{listofdefsofopensubgroups}
\begin{split}
\mf{G} &:= \{ G \subseteq \GL_2(\hat{\mbz}) : \, G \text{ is open and } \det G = \hat{\mbz}^\times \}, \\
\mf{G}(g) &:= \{ G \in \mf{G} : \,  X_{\tilde{G}} \text{ has genus } g \}, \\
\mf{G}_{MT} &:= \{ G \in \mf{G}  :  \exists r \in \mbz \text{ with } G_r = \emptyset \}, \\
\mf{G}_{MT}^{\max} &:= \{ G \in \mf{G}_{MT} : \, G \text{ is maximal with respect to } \dot\subseteq \}, \\
\mf{G}_{MT}(g) &:= \mf{G}_{MT} \cap \mf{G}(g),  \quad\quad\quad \mf{G}_{MT}^{\max}(g) := \mf{G}_{MT}^{\max} \cap \mf{G}(g).\\
\end{split}
\end{equation}
As a consequence of the Weil pairing, for any elliptic curve $E$ over $\mbq$, we have $\det \rho_{E}(G_\mbq) = \hat{\mbz}^\times$ (see Lemma \ref{weilpairinglemma} below); this is the reason for the condition $\det G = \hat{\mbz}^\times$ in the definition of $\mf{G}$ in \eqref{listofdefsofopensubgroups}.  Furthermore, for any $m \in \mbn$ and prime $p$, we have
\[
p \nmid m N_E \; \Longrightarrow \; \rho_{E,m} \text{ is unramified at $p$ } \; \text{ and } \; a_p(E) \equiv \tr \rho_{E,m} (\frob_p) \bmod m,
\]
where $\frob_p \in G_\mbq$ denotes any Frobenius automorphism over $p$ (this follows, for instance, from \cite[Proposition 2.3, p. 134]{silverman}).  
Consequently, for any elliptic curve $E$ over $\mbq$, Conjecture \ref{LTConjecture} implies that
\begin{equation} \label{goalbygenus}
\exists r \in \mbz \text{ with } \lim_{X \rightarrow \infty} \pi_{E,r}(X) < \infty \; \Longleftrightarrow \; \exists G \in \mf{G}_{MT}^{\max} \text{ with } \rho_{E}(G_\mbq) \, \dot\subseteq \, G.
\end{equation}
(The implication ``$\Longleftarrow$'' is unconditional, whereas ``$\Longrightarrow$'' depends on Conjecture \ref{LTConjecture}.)  Thus, the goal of classifying elliptic curves $E$ over $\mbq$ satisfying the left-hand condition in \eqref{goalbygenus} leads to our consideration of the rational points of the modular curves $X_{\tilde{G}}$, for each $G \in \mf{G}_{MT}^{\max}(g)$, for each fixed $g \geq 0$.
We remark that, in case $G = \tilde{G}$ (i.e. in case $-I \in G$) and assuming $j_E \notin \{ 0, 1728 \}$, the property that $\rho_{E}(G_\mbq) \subseteq G$ does not vary as we twist $E$, i.e. we have
\[
-I \in G \; \Longrightarrow \; \left( \forall d \in \mbq^\times/(\mbq^\times)^2, \; \rho_{E}(G_\mbq) \, \dot\subseteq \, G \Leftrightarrow \rho_{E^{(d)}}(G_\mbq) \, \dot\subseteq \, G \right),
\]
where $E^{(d)}$ denotes the twist of $E$.  In particular, when $-I \in G$ and assuming $j_E \notin \{ 0, 1728 \}$, the property that $\rho_{E}(G_\mbq) \, \dot\subseteq \, G$ only depends on the $j$-invariant of $E$.  By contrast, in case $-I \notin G$, the property $\rho_{E^{(d)}}(G_\mbq) \, \dot\subseteq \, G$ depends, in general, on twist parameter $d$.  Thus, classifying elliptic curves $E$ with $\rho_{E}(G_\mbq) \, \dot\subseteq \, G$, amounts to 
\begin{enumerate}
\item describing explicitly the map $j_{\tilde{G}} : X_{\tilde{G}} \longrightarrow X(1)$,
\item in case $-I \notin G$, describing the particular twists $\{ E^{(d)} \}_{d \in \mbq^\times/(\mbq^\times)^2}$ that satisfy $\rho_{E^{(d)}}(G_\mbq) \, \dot\subseteq \, G$.
\end{enumerate}

Our main result classifies the set of elliptic curves $E$ over $\mbq$ for which $\rho_E(G_\mbq) \, \dot\subseteq \, G$ for some $G \in \mf{G}_{MT}^{\max}(0)$, in cases according to whether or not $-I \in G$, as described above.  In particular, it extends each of Examples \ref{level3example}, \ref{level8example}, and \ref{level28example} to the following one-parameter families.  We define the rational functions in $\mbq(t,D)$:
\begin{equation} \label{shortlistofjsandds}
\begin{array}{ll}
j_{3,1}(t) := \ds 27\frac{(t+1)(t+9)^3}{t^3} & d_{3,1,1}(t,D) :=  \ds \frac{6(t+1)(t+9)}{t^2 - 18t - 27}, \\
& \\
j_{6,1}(t) := \ds 2^{10}3^3t^3(1-4t^3) & d_{6,1,1}(t,D) := \ds D, \\
& \\
j_{28,1}(t) := \ds - \frac{(49t^4 - 13t^2 + 1)(2401t^4 - 245t^2 + 1)^3}{t^2} & \\
& \\
d_{28,1,1}(t,D) :=  \ds \frac{-7t(49t^4 - 13t^2 + 1)(2401t^4 - 245t^2 + 1)}{823543t^8 - 235298t^6 + 21609t^4 - 490t^2 - 1}. & 
\end{array} 
\end{equation}
Furthermore, for $(m,i) \in \{ (3,1), (6,1), (28,1) \}$, we set the Weierstrass coefficients $a_{4;m,i}(t), a_{6;m,i}(t) \in \mbq(t)$ by
\begin{equation} \label{defofa4anda6}
a_{4;m,i}(t) := \frac{108j_{m,i}(t)}{1728 - j_{m,i}(t)}, \quad\quad a_{6;m,i}(t) := \frac{432j_{m,i}(t)}{1728-j_{m,i}(t)}.
\end{equation}
For $(m,i,k) \in \{ (3,1,1), (8,1,1), (28,1,1) \}$ we have already declared the twist parameters $d_{m,i,k}(t,D) \in \mbq(t,D)$ in \eqref{shortlistofjsandds}, and we define the elliptic curves $\mc{E}_{m,i,k}$ over $\mbq(t,D)$ by
\begin{equation} \label{ellipticsurface}
\mc{E}_{m,i,k} : \; d_{m,i,k}(t,D) y^2 = x^3 + a_{4;m,i}(t) x + a_{6;m,i}(t).
\end{equation}
For $t_0,D_0 \in \mbq$, we denote by $\mc{E}_{m,i,k}(t_0,D_0)$ the elliptic curve over $\mbq$ obtained by specializing $\mc{E}_{m,i,k}$ at $t = t_0$ and $D = D_0$.  The elliptic curves $E_3$, $E_6$ and $E_{28}$ of Examples \ref{level3example} - \ref{level28example} satisfy
\[
E_3 = \mc{E}_{3,1,1}(1,1), \quad\quad E_6 = \mc{E}_{6,1,1}(1,1), \quad\quad E_{28} = \mc{E}_{28,1,1}(1,1).
\]

In Tables \ref{masterlistofjinvariants} and \ref{masterlistoftwistparameters}, which appear in Section \ref{tablesection}, we associate $j$-invariants $j_{m,i}(t) \in \mbq(t)$ and twist parameters $d_{m,i,k}(t,D) \in \mbq(t,D)$ to all of the $3$-tuples\footnote{In each $3$-tuple $(m,i,k)$, the first entry $m$ names the $\GL_2$-level of the corresponding group; for a fixed $m$, the index $i$ changes exactly if the $j$-invariant changes, and the last index $k$ changes as that twist class changes for a fixed $j$-invariant.}
\begin{equation} \label{masterlistofindices}
(m,i,k) \in \left\{ \begin{matrix} (2,1,1), (3,1,1), (3,1,2), (4,1,1), (5,1,1), (5,2,1), (5,2,1), (5,2,1), (6,1,1), \\ (6,2,1), (6,3,1), (6,3,2), (7,1,1), (7,1,2), (7,2,1), (7,2,2), (7,3,1), (7,3,2), \\ (8,1,1), (9,1,1), (9,2,1), (9,3,1), (9,4,1), (10,1,1), (10,1,2), (10,2,1), \\ (10,2,2), (10,3,1), (12,1,1), (12,1,2), (12,2,1), (12,3,1), (12,4,1), (14,1,1), \\ (14,2,1), (14,2,2), (14,3,1), (14,3,2), (14,4,1), (14,4,2), (14,5,1), (14,6,1), \\ (14,6,2), (14,7,1), (14,7,2), (28,1,1), (28,2,1), (28,2,2), (28,3,1), (28,3,2) \end{matrix} \right\}.
\end{equation}
Each $j$-invariant $j_{m,i}(t)$ in our list will correspond to the natural map $X_{\tilde{G}} \longrightarrow X(1)$ associated to a group $G \in \mf{G}_{MT}^{\max}(0)$ and a choice of parameter $t \in \mbq(X_{\tilde{G}})$, and we will have $d_{m,i,k}(t,D) = D$ just in case $-I \in G$.  When  $-I \notin G$, we will have $d_{m,i,k}(t,D) \in \mbq(t) \subseteq \mbq(t,D)$, and we may also denote it simply by $d_{m,i,k}(t)$ in this case.  For each $j$-invariant $j_{m,i}(t)$ corresponding to such a tuple $(m,i,k)$ in \eqref{masterlistofindices}, we again define the Weierstrass coefficients $a_{4;m,i}(t), a_{6;m,i}(t) \in \mbq(t)$ by \eqref{defofa4anda6} and consider the associated elliptic curve $\mc{E}_{m,i,k}$ over $\mbq(t,D)$ defined by \eqref{ellipticsurface}; for $t_0, D_0 \in \mbq$ we denote by $\mc{E}_{m,i,k}(t_0,D_0)$ the specialization of $\mc{E}_{m,i,k}(t,D)$ to $t = t_0$ and $d = D_0$.  For all pairs $(t_0,D_0)$ in a Zariski open subset of $\mathbb{A}_2(\mbq)$, the specialized curve $\mc{E}_{m,i,k}(t_0,D_0)$ is an elliptic curve over $\mbq$.  In case $-I \notin G$, since the corresponding elliptic curve $\mc{E}_{m,i,k}$ as in \eqref{ellipticsurface} is defined over $\mbq(t)$, we may also denote simply by $\mc{E}_{m,i.k}(t_0)$ its specialization to $t = t_0$, which is an elliptic curve over $\mbq$ for all but finitely many $t_0 \in \mbq$.

\begin{Theorem} \label{maintheorem}
Let $E$ be an elliptic curve over $\mbq$ with $j$-invariant $j_E$ satisfying $j_E \notin \{ 0, 1728 \}$.  We have that $\exists G \in \mf{G}_{MT}^{\max}(0)$ with $\rho_E(G_\mbq) \, \dot\subseteq \, G$ if and only if there are $t_0,D_0 \in \mbq$ and a $3$-tuple $(m,i,k)$ in the set \eqref{masterlistofindices} so that $E$ is isomorphic over $\mbq$ to the elliptic curve
\[
\mc{E}_{m,i,k}(t_0,D_0) : \; d_{m,i,k}(t_0,D_0) y^2 = x^3 + a_{4;m,i}(t_0)x + a_{6;m,i}(t_0),
\]
where the $j$-invariant $j_{m,i}(t) \in \mbq(t)$ and twist parameter $d_{m,i,k}(t,D) \in \mbq(t,D)$ are as listed in Tables \ref{masterlistofjinvariants} and \ref{masterlistoftwistparameters} of Section \ref{tablesection} and the coefficients $a_{4;m,i}(t), a_{6,m,i}(t) \in \mbq(t)$ are defined by \eqref{defofa4anda6}.
\end{Theorem}
The proof of Theorem \ref{maintheorem} falls into two steps, the first one bounding the levels associated to each of the groups $G \in \mf{G}_{MT}^{\max}(0)$.  In addition to \eqref{listofdefsofopensubgroups}, we make the definitions
\begin{equation} \label{defsofopensubgroupswithlevel}
\begin{split}
\mf{G}(g,m) &:= \{ G \in \mf{G}(g) : \level_{\GL_2}(G) = m \} \\
\mf{G}_{MT}(g,m) &:= \mf{G}_{MT} \cap \mf{G}(g,m) \\
\mf{G}_{MT}^{\max}(g,m) &:= \mf{G}_{MT}^{\max} \cap \mf{G}(g,m).
\end{split}
\end{equation}
We will establish the following theorem.
 \begin{Theorem} \label{boundingthelevelsthm}
Let the set $\mf{G}_{MT}^{\max}(g)$ of open subgroups of $\GL_2(\hat{\mbz})$ be as defined in \eqref{listofdefsofopensubgroups}.  We then have
\begin{equation} \label{genuszerocurvesbylevel}
\mf{G}_{MT}^{\max}(0) = \bigcup_{m \in \left\{ 2, 3, 4, 5, 6, 7, 8, \atop 9, 10, 12, 14, 28 \right\}} \mf{G}_{MT}^{\max}(0,m),
\end{equation}
where the set $\mf{G}_{MT}^{\max}(g,m)$ is as in \eqref{defsofopensubgroupswithlevel}.
\end{Theorem}
Theorem \ref{boundingthelevelsthm} is proved as follows.  An equivalent formulation of a conjecture of  Rademacher states that
\[
\{ \text{open subgroups } S \subseteq \SL_2(\hat{\mbz}) : -I \in S \text{ and } X_S \text{ has genus } 0 \} / \doteq
\]
is a finite set.  
This conjecture was proven by Denin (see \cite{denin1}, \cite{denin2} and \cite{denin3}).  More generally, in \cite{thompson} and \cite{zograf}, the same is shown with $0$ replaced by a general $g \in \mbn \cup \{ 0 \}$. 
In addition, there are several papers on the \emph{effective} resolution of Rademacher's conjecture.  In particular, Cummins and Pauli \cite{cumminspauli} have produced the complete list of the elements of 
\[
\{ \text{open subgroups } S \subseteq \SL_2(\hat{\mbz}) : -I \in S \text{ and } X_S \text{ has genus } g \} / \doteq
\]
for $g \leq 24$, and our proof of Theorem \ref{boundingthelevelsthm} makes use of the tables therein.
We extend the notion of $\GL_2$-level of an open subgroup $G \subseteq \GL_2(\hat{\mbz})$ by defining
\begin{equation} \label{defofGL2levelandSL2level}
\begin{split}
\level_{\GL_2}(G) &:= \min \left\{ m \in \mbn : \ker \left( \GL_2(\hat{\mbz}) \rightarrow \GL_2(\mbz/m\mbz) \right) \subseteq G \right\} \\
\level_{\SL_2}(G) &:= \min \left\{ m \in \mbn : \ker \left( \SL_2(\hat{\mbz}) \rightarrow \SL_2(\mbz/m\mbz) \right) \subseteq G \cap \SL_2(\hat{\mbz}) \right\}.
\end{split}
\end{equation}
It is straightforward to see that $\level_{\SL_2}(G)$ divides $\level_{\GL_2}(G)$, and in general they can be different.  Using the main result of \cite{cumminspauli}, we will first show that
\begin{equation*} 
G \in \mf{G}(0) \; \Longrightarrow \; \level_{\SL_2}(G) \in \left\{ \begin{matrix} 1, 2, 3, 4, 5, 6, 7, 8, 9, 10, 11, 12, 13, 14, 15, 16, 18, 20, 21, 22, 24, \\ 25, 26, 27, 28, 30, 32, 36, 40, 42, 48, 50, 52, 54, 56, 60, 64, 72, 96 \end{matrix} \right\}.
\end{equation*}
Next, for each $G \in \mf{G}_{MT}^{\max}(0)$, we exhibit a positive integer $d_G$ for which $\level_{\GL_2}(G)$ divides $d_G \cdot \level_{\SL_2}(G)$, and this, together with a MAGMA computation, yields Theorem \ref{boundingthelevelsthm}.

To establish Theorem \ref{maintheorem}, we will utilize results of \cite{sutherlandzywina} and \cite{zywina}, which describe explicitly all prime power level modular curves with infinitely many rational points.  For the prime power levels (other than the $m=8$) occurring on the right-hand side of \eqref{genuszerocurvesbylevel} we use those results directly; for each group $G$ of level $m$ that is not a prime power, the associated missing trace is caused by an \emph{entanglement}, i.e. a non-trivial intersection
\[
\mbq(E[m_1]) \cap \mbq(E[m_2]) \neq \mbq \quad\quad\quad \left(m = m_1m_2, \; \gcd(m_1,m_2) = 1 \right)
\]
implicit in the group $G$ (for $m=8$, the missing trace is caused by a ``vertical entanglement'' and also requires additional work).  In the cases involving entanglement, we undertake a finer analysis, identifying precisely the underlying subfields and determining the subfamily for which those subfields agree.  We remark that entanglements of division fields comprise an area of recent interest and activity.  For example, the results in \cite{braujones} classify elliptic curves $E$ over $\mbq$ for which $\left[ \mbq(E[2]) : \mbq \right] = 6$ and $\mbq(E[2]) \subseteq \mbq(E[3])$ (see also \cite{morrow}); this was generalized in \cite{jonesmcmurdy}, which classifies all elliptic curves $E$ defined over any number field $K$ that fit into a genus zero family and for which there exist coprime $(m,n) \in \mbn^2$ such that $K(E[m]) \cap K(E[n])$ is non-abelian over $K$.  In \cite{danielsmorrow} the authors classify all possible elliptic curves $E$ over $\mbq$ that fit into an infinite family of genus zero or one and for which there exist primes $p$ and $q$ such that $\mbq(E[p]) \cap \mbq(E[q]) \neq \mbq$, and the results in \cite{danielslozanorobledo} classify all elliptic curves $E$ over $\mbq$ and pairs $(m,n) \in \mbn^2$ for which $\mbq(E[n]) = \mbq(E[m])$.  Finally, \cite{campagnapengo} studies the phenomenon of entanglement for CM elliptic curves.

The paper is organized as follows.  In Section \ref{proofofboundingthelevelsthmsection} we prove Theorem \ref{boundingthelevelsthm}.  In Section \ref{proofofmaintheoremsection} we prove Theorem \ref{maintheorem} and in Section \ref{tablesection} we summarize the results in three tables.  Finally, in Section \ref{concludingremarkssection} we discuss future directions.

\subsection{Remarks on computation} \label{Section:ComputationRemark}

All of the computations necessary to justify the theorems in this paper were performed using the computational software package MAGMA \cite{MAGMA}. The code used to perform these computations can be found at the following link:
\[
\text{{\tt{https://github.com/ncjones-uic/MissingFrobeniusTraces}}}
\]

\subsection{Acknowledgements}

The authors would like to thank Chantal David, Zeev Rudnick and Andrew Sutherland for comments on a previous version of the paper.

\section{Notation and Group-theoretic preliminaries}

In this section, we gather notation and preliminary lemmas.  Given an open subgroup $G \subseteq \GL_2(\hat{\mbz})$, throughout the paper we will often denote by $m_G$ the $\GL_2$-level of $G$ and by $m_S$ the $\SL_2$-level of $G$, as
defined as in \eqref{defofGL2levelandSL2level}.
For any open subgroup $S \subseteq \SL_2(\hat{\mbz})$, we may also denote by $m_S$ its $\SL_2$-level.  Also, for any such subgroups we maintain the notation from the introduction:
\begin{equation} \label{defoofGtildeandStilde}
\begin{split}
\tilde{G} &:= \langle -I, G \rangle, \\
\tilde{S} &:= \langle -I, S \rangle.
\end{split}
\end{equation}
\begin{proposition} \label{levelincreaseboundprop}
Let $G \subseteq \GL_2(\hat{\mbz})$ be an open subgroup.  We then have
\begin{equation*} 
\frac{\level_{\SL_2}(G)}{\level_{\SL_2}(\tilde{G})} \in \{ 1, 2 \},
\end{equation*}
where $\tilde{G}$ is as in \eqref{defoofGtildeandStilde}.
\end{proposition}
\begin{proof}
See \cite[Lemma 3.1]{jonesmcmurdy}.
\end{proof}

In the next lemma, for an open subgroup $G \subseteq \GL_2(\hat{\mbz})$ of $\GL_2$-level $m_G$ and $\SL_2$-level $m_S$, and for an arbitrary positive integer $m$ with $m_S \mid m \mid m_G$, we let $\pi_{\GL_2}$ and $\pi_{\mb{G}_m}$ denote the canonical projection maps
\begin{equation} \label{defofpi}
\begin{split}
&\pi_{\GL_2} :  \GL_2(\mbz/m_G\mbz) \longrightarrow \GL_2(\mbz/m\mbz), \\
&\pi_{\mb{G}_m} : (\mbz/m_G\mbz)^\times \longrightarrow (\mbz/m\mbz)^\times.
\end{split}
\end{equation}
\begin{lemma} \label{verticalSL2liftlemma}
Let $G \subseteq \GL_2(\hat{\mbz})$ be an open subgroup satisfying $m_G \mid m_S^\infty$ and let $m$ be any positive integer satisfying $m_S \mid m$ and $m \mid m_G$.  Then there exists a unique group homomorphism
\[
\gd : G(m) \longrightarrow (\mbz/m_G\mbz)^\times
\]
satisfying $\pi_{\mb{G}_m} \circ \gd = \det$ (where $\pi_{\mb{G}_m}$ is the canonical projection as in \eqref{defofpi}), and such that
\[
G(m_G) = \left\{ g \in \pi_{\GL_2}^{-1}\left( G(m) \right) : \gd \left( \pi_{\GL_2}(g) \right) = \det g \right\}.
\]
If $\det G = \hat{\mbz}^\times$, then $\gd$ is surjective and $\gd\left( G(m) \cap \SL_2(\mbz/m\mbz) \right) = \ker \pi_{\mb{G}_m}$.
\end{lemma}
\begin{proof}
Let $\pi_G : G(m_G) \longrightarrow G(m)$ denote the restriction to $G(m_G)$ of $\pi_{\GL_2}$.  We will first establish that
\begin{equation} \label{kerpiGequalsSL2kernel}
\ker \pi_G = \ker \left( \SL_2(\mbz/m_G\mbz) \rightarrow \SL_2(\mbz/m\mbz) \right).
\end{equation}
First, by definition of $m_S$ and since $m_S \mid m$, we have
\[
\ker \pi_G \supseteq \ker \left(  \SL_2(\mbz/m_G\mbz) \rightarrow \SL_2(\mbz/m\mbz) \right),
\]
and we will now argue by induction that these kernels have the same size.  Let $p$ be any prime dividing $m_G/m$ and factor $\pi_G$ as
\[
\begin{tikzcd}
G(m_G) \rar{\pi_p} \arrow[black, bend left]{rr}{\pi_G} & G(m_G/p) \rar{\pi_{m/p}} & G(m).
\end{tikzcd}
\]
By induction, we have that
\[
\left| \ker \pi_{m/p} \right| = \left| \ker \left(  \SL_2(\mbz/(m_G/p)\mbz) \rightarrow \SL_2(\mbz/m\mbz) \right) \right|.
\]
Since $m_G$ divides $m_S^\infty$, we see that $p$ divides $m_G/p$, and so
\[
\begin{split}
&\ker \left( \GL_2(\mbz/m_G\mbz) \rightarrow \GL_2(\mbz/(m_G/p)\mbz) \right), \\
&\ker \left( \SL_2(\mbz/m_G\mbz) \rightarrow \SL_2(\mbz/(m_G/p)\mbz) \right)
\end{split}
\]
are abelian groups of orders $p^4$ and $p^3$, respectively.  Since $m_G/p$ is not the $\GL_2$-level of $G$, we have
\[
\ker \left( \SL_2(\mbz/m_G\mbz) \rightarrow \SL_2(\mbz/(m_G/p)\mbz) \right) \subseteq \ker \pi_p \subsetneq \ker \left( \GL_2(\mbz/m_G\mbz) \rightarrow \GL_2(\mbz/(m_G/p)\mbz) \right).
\]
It follows that $\ker \left( \SL_2(\mbz/m_G\mbz) \rightarrow \SL_2(\mbz/(m_G/p)\mbz) \right) = \ker \pi_p$, so 
\[
\left| \ker \pi_G \right| = \left| \pi_p^{-1} \left( \ker \pi_{m/p} \right) \right| = \left| \ker \left( \SL_2(\mbz/m_G\mbz) \rightarrow \SL_2(\mbz/m\mbz) \right) \right|, 
\]
and \eqref{kerpiGequalsSL2kernel} is thus verified.

We now define the map $\gd : G(m) \longrightarrow (\mbz/m_G\mbz)^\times$ as follows.  For $g \in G(m)$, fix any element $\tilde{g} \in G(m_G)$ satisfying $\pi_G(\tilde{g}) = g$ and set $\gd(g) := \det \tilde{g}$.  By virtue of \eqref{kerpiGequalsSL2kernel}, we see that $\gd(g)$ is independent of the choice of lift $\tilde{g}$ and thus $\gd$ is a well-defined group homomorphism; the condition $\pi_{\mb{G}_m} \circ \gd = \det$ is immediately verified, as is
\begin{equation*} 
\gd(G(m)) = \det G(m_G).
\end{equation*}
In particular, if $\det G(m_G) = (\mbz/m_G\mbz)^\times$, then $\gd$ is surjective.  Furthermore, it follows from $\pi_{\mb{G}_m} \circ \gd = \det$ that
\[
\gd\left( G(m) \cap \SL_2(\mbz/m\mbz) \right) = \ker \pi_{\mb{G}_m} \cap \gd \left( G(m) \right) .
\]
Thus, if $\gd$ is surjective then $\gd\left( G(m) \cap \SL_2(\mbz/m\mbz) \right) = \ker \pi_{\mb{G}_m}$.  Finally, 
we clearly have
\[
G(m_G) \subseteq \left\{ g \in \pi_{\GL_2}^{-1}\left( G(m) \right) : \gd\left( \pi_{\GL_2}(g) \right) = \det g \right\},
\]
and, from \eqref{kerpiGequalsSL2kernel}, the two groups are seen to have equal size, and are thus equal.
\end{proof}

\section{Bounding the $\GL_2$-level of missing trace groups of genus zero} \label{proofofboundingthelevelsthmsection}

In this section we prove Theorem \ref{boundingthelevelsthm}.  To begin with, the main results in \cite{cumminspauli} immediately imply that
\begin{equation*} 
\left\{ \level_{\SL_2}(\tilde{G}) : G \in \mf{G}(0) \right\} = \left\{ \begin{matrix} 1, 2, 3, 4, 5, 6, 7, 8, 9, 10, 11, 12, 13, 14, 15, 16, \\ 18, 20, 21, 24, 25, 26, 27, 28, 30, 32, 36, 48 \end{matrix} \right\}.
\end{equation*}
By Proposition \ref{levelincreaseboundprop} we then have
\begin{equation} \label{boundforSL2levels}
G \in \mf{G}(0) \; \Longrightarrow \; \level_{\SL_2}(G) \in \left\{ \begin{matrix} 1, 2, 3, 4, 5, 6, 7, 8, 9, 10, 11, 12, 13, 14, 15, 16, 18, 20, 21, 22, 24, \\ 25, 26, 27, 28, 30, 32, 36, 40, 42, 48, 50, 52, 54, 56, 60, 64, 72, 96 \end{matrix} \right\}.
\end{equation}
Our next lemma is crucial in what follows and is useful in describing the Galois group of the compositum of two fields that have a non-trivial intersection over the base field.
\begin{lemma} \label{goursatlemma} (Goursat's Lemma)
\noindent
Let $G_1$, $G_2$ be groups and for $i \in \{1, 2 \}$ denote by $\pr_i : G_1 \times G_2 \longrightarrow G_i$ the projection map onto the $i$-th factor. Let $G \subseteq G_1 \times G_2$ be a subgroup and 
assume that 
$$
\pr_1(G) = G_1, \; \pr_2(G) = G_2.
$$
Then 
there exists a group $\Gamma$ together with a pair of surjective homomorphisms 
\[
\begin{split}
\psi_1 : G_1 &\longrightarrow \Gamma \\
\psi_2 : G_2 &\longrightarrow \Gamma
\end{split}
\]
so that 
\[
G = G_1 \times_\psi G_2 := \{ (g_1,g_2) \in G_1 \times G_2 : \psi_1(g_1) = \psi_2(g_2) \}.
\]
\end{lemma}
\begin{proof}
See \cite[Lemma (5.2.1)]{ribet}.
\end{proof}
We will now establish the following proposition, which bounds $m_G$ for $G \in \mf{G}_{MT}^{\max}(0)$ in terms of the $\SL_2$-level of $G$.  For an open subgroup $G \subseteq \GL_2(\hat{\mbz})$, we set 
\begin{equation} \label{defofksubG}
d_G := \gcd \left( m_S^ \infty, \left| \frac{G(m_S) \cap \SL_2(\mbz/m_S\mbz)}{[G(m_S),G(m_S)]} \right| \right) \qquad\quad \left( m_S := \level_{\SL_2}(G) \right),
\end{equation}
i.e. $d_G$ is the largest factor of $\left| \frac{G(m_S) \cap \SL_2(\mbz/m_S\mbz)}{[G(m_S),G(m_S)]} \right|$ supported on primes dividing $m_S$.
\begin{proposition} \label{twotothekproposition}
Let $G \in \mf{G}_{MT}^{\max}$ be a maximal missing trace group of $\GL_2$-level $m_G$ and $\SL_2$-level $m_S$.  Then $m_G$ divides $d_G m_S$, where $d_G$ is defined by \eqref{defofksubG}.  In particular, $m_G$ is supported on primes dividing $m_S$.
\end{proposition}
\begin{proof}
Without loss of generality, we may assume that $m_G > m_S$; we let $p$ be any prime for which $v_p(m_G) > v_p\left( m_S \right)$ and set $m_G' := m_G/p^{v_p(m_G) - v_p(m_S)}$.  Note that, for any prime $\ell$, we have
\begin{equation} \label{vpofmsubGequalsvpofmsubS}
v_{\ell}(m_G') = 
\begin{cases}
v_{\ell}(m_G) & \text{ if } \ell \neq p \\
v_{\ell}(m_S) & \text{ if } \ell = p. 
\end{cases} 
\end{equation}
In particular, since $m_S$ divides $m_G'$, we then have
\begin{equation} \label{kernelofSL2iscontainedinG}
\ker \left( \SL_2(\mbz/m_G\mbz) \rightarrow \SL_2(\mbz/m_G'\mbz) \right) \subseteq G(m_G).
\end{equation}
Furthermore, since $G$ is maximal among missing trace groups, it follows that
\begin{equation} \label{tracemodmsubGprimeissurjective}
\tr\left( G(m_G') \right) = \mbz/m_G'\mbz.
\end{equation}
We first claim that 
\begin{equation} \label{pdividesmsubGprime}
p \mid m_G'.
\end{equation}
To see this, suppose for the sake of contradiction that $p \nmid m_G'$, and define $\ga := v_p(m_G) - v_p(m_S)$.  Then, since $m_S$ divides $m_G'$, under the Chinese Remainder Isomorphism $\mbz/m_G\mbz \simeq \mbz/m_G'\mbz \times \mbz/p^{\ga}\mbz$, condition \eqref{kernelofSL2iscontainedinG} reads 
\begin{equation} \label{SL2crossoneiscontained}
\{ 1 \} \times \SL_2(\mbz/p^{\ga}\mbz) \subseteq G(m_G).
\end{equation}
By surjectivity of $\det : G(p^{\ga}) \twoheadrightarrow (\mbz/p^{\ga}\mbz)^\times$ and Lemma \ref{goursatlemma}, we then have
\begin{equation} \label{GofmsubGasfiberedproduct}
G(m_G) \simeq G(m_G') \times_{\psi} \GL_2(\mbz/p^{\ga}\mbz),
\end{equation}
where $\psi_{m_G'} : G(m_G') \longrightarrow \Gamma$ and $\psi_p : \GL_2(\mbz/p^{\ga}\mbz) \longrightarrow \Gamma$ denote the surjective group homomorphisms onto the common quotient group $\Gamma$ implicit in the fibered product.
It follows from \eqref{SL2crossoneiscontained} that $\SL_2(\mbz/p^{\ga}\mbz) \subseteq \ker \psi_p$, and thus, for every $\gamma \in \Gamma$, there exists $d \in (\mbz/p^{\ga}\mbz)^\times$ for which
\[
\{ g \in \GL_2(\mbz/p\mbz) : \det g = d \} \subseteq \psi_p^{-1}(\gamma).
\]
Since $\tr\left( \{ g \in \GL_2(\mbz/p\mbz) : \det g = d \}  \right) = \mbz/p^{\ga}\mbz$, it is then easy to deduce from \eqref{GofmsubGasfiberedproduct} and \eqref{tracemodmsubGprimeissurjective} that $\tr\left( G(m_G) \right) = \mbz/m_G\mbz$, a contradiction.  Therefore we have established \eqref{pdividesmsubGprime}, and, by \eqref{vpofmsubGequalsvpofmsubS}, $p \mid m_S$.  Since the prime $p$ was arbitrary, it follows that
\begin{equation} \label{msubGdividesmsubStotheinfty}
m_G \mid m_S^\infty.
\end{equation}

We now apply Lemma \ref{verticalSL2liftlemma}, which asserts that there is a surjective group homomorphism $\gd : G(m_S) \longrightarrow (\mbz/m_G\mbz)^\times$ satisfying $\pi \circ \gd = \det : G(m_S) \longrightarrow (\mbz/m_S\mbz)^\times$, and for which
\[
G(m_G) = \left\{ g \in \pi_{\GL_2}^{-1}\left( G(m_S) \right) : \gd\left( \pi(g) \right) = \det g \right\}.
\]
By \eqref{msubGdividesmsubStotheinfty} and Lemma \ref{verticalSL2liftlemma}, we have that
\[
m_G/m_S = \left| \ker \left( (\mbz/m_G\mbz)^\times \rightarrow (\mbz/m_S\mbz)^\times \right) \right| = \left| \gd\left( G(m_S) \cap \SL_2(\mbz/m_S\mbz) \right) \right|,
\]
which in turn divides
\[
\left| \frac{G(m_S) \cap \SL_2(\mbz/m_S\mbz)}{[G(m_S),G(m_S)]} \right|.
\]
By \eqref{msubGdividesmsubStotheinfty}, $m_G/m_S$ also divides $m_S^\infty$, and Proposition \ref{twotothekproposition} follows.
\end{proof}
Theorem \ref{boundingthelevelsthm} follows from \eqref{boundforSL2levels} and Proposition \ref{twotothekproposition}, together with a computer computation.  The latter was carried out using the computational package MAGMA \cite{MAGMA}.  More precisely, for any fixed $g \in \mbz_{\geq 0}$, Proposition \ref{twotothekproposition} bounds the $\GL_2$-level of any $G \in \mf{G}_{MT}^{\max}(g)$ in terms of data available at the $\SL_2$-level of $G$.  Suppose that $G \in \mf{G}_{MT}^{\max}$ satisfies $m_G := \level_{\GL_2}(G) > \level_{\SL_2}(G) =: m_S$.  Defining the open subgroup $\hat{G} \subseteq \GL_2(\hat{\mbz})$ by $\hat{G} := \pi^{-1}\left( G(m_S) \right)$, we have that 
\begin{equation} \label{relationshipbetweenGandGhat}
m_{\hat{G}} := \level_{\GL_2}(\hat{G}) = \level_{\SL_2}(G) \; \text{ and } \; G(m_{\hat{G}}) = \hat{G}(m_{\hat{G}}). 
\end{equation}
Note also that the quantity $d_G$ defined by \eqref{defofksubG} depends only on the group $\hat{G}$, and we may thus denote it by $d_{\hat{G}}$:
\begin{equation*} 
d_G = \gcd \left( m_{\hat{G}}^ \infty, \left| \frac{\hat{G}(m_{\hat{G}}) \cap \SL_2(\mbz/m_{\hat{G}}\mbz)}{[\hat{G}(m_{\hat{G}}),\hat{G}(m_{\hat{G}})]} \right| \right) =: d_{\hat{G}}.
\end{equation*}
To make our computation more efficient, it will be useful to observe other restrictions on the group $\hat{G}$ imposed by the existence of the subgroup $G \in \mf{G}_{MT}^{\max}$.
Momentarily forgetting about $G$ and viewing $\hat{G}$ as a group on its own, our next proposition establishes a necessary condition on $\hat{G}$ that will allow the possibility of a subgroup $G \in \mf{G}_{MT}^{\max}$ with $G \subsetneq \hat{G}$ and satisfying \eqref{relationshipbetweenGandGhat}.
\begin{Definition} \label{defofdadmissible}
Let $\hat{G} \subseteq \GL_2(\hat{\mbz})$ be an open subgroup with $\GL_2$-level $m_{\hat{G}}$ and let $d \in \mbz_{\geq 2}$ be a divisor of $m_{\hat{G}}^\infty$.  We say that $\hat{G}$ is {\textbf{\emph{$d$-admissible}}} if there exists $t \in \mbz/m_{\hat{G}}\mbz$ such that
\[
\forall g \in \hat{G}(m_{\hat{G}}) \, \text{ with } \, \tr g \equiv t \bmod m_{\hat{G}} \, \text{ and } \, \forall p \mid d, \, \text{ we have } \, g \equiv \gl_p I \bmod p \, \text{ for some } \, \gl _p\in (\mbz/p\mbz)^\times.
\]
\end{Definition}
Note that a group $G$ is $d$-admissible if and only if it is $d_{\SF}$-admissible, where $d_{\SF}$ is the square-free part of $d$.  Thus, when discussing $d$-admissibility of a group, we may as well assume that $d$ is square-free.  Furthermore, note that, if $G$ is $d$-admissible, then $G$ is $d'$-admissible for every divisor $d'$ of $d$.  
\begin{proposition} \label{admissiblegroupsprop}
Suppose that $G \in \mf{G}_{MT}^{\max}$ satisfies $m_G := \level_{\GL_2}(G) > \level_{\SL_2}(G) =: m_S$ and define the open subgroup $\hat{G} \subseteq \GL_2(\hat{\mbz})$ by $\hat{G} := \pi^{-1}\left( G(m_S) \right)$.  Setting $d := m_G/m_S$, we have that $\hat{G}$ is $d$-admissible.  Furthermore, $\tr \left( \hat{G}(m_{\hat{G}}) \right) = \mbz/m_{\hat{G}}\mbz$, where $m_{\hat{G}}$ denotes the $\GL_2$-level of $\hat{G}$.
\end{proposition}
\begin{proof}
Since $G \in \mf{G}_{MT}^{\max}$, there exists $t_0 \in \mbz/m_G\mbz$ such that
\begin{equation} \label{tobecontradicted}
\forall g \in G(m_G), \; \tr g \not \equiv t_0 \bmod m_G.
\end{equation}
Choose any such $t_0$, let $p$ be any prime dividing $m_G/m_S$ and denote by $\ol{t}_0 \in \mbz/(m_G/p)\mbz$ the reduction of $t_0$ modulo $m_G/p$.  Write $m_G =: p^n m_G'$, with $n \geq 2$ and $p \nmid m_G'$ (since $m_G \mid m_S^\infty$, we cannot have $n = 1$).  We then have
\begin{equation} \label{CRTisom}
\mbz/m_G\mbz \simeq \mbz/p^n\mbz \times \mbz/m_G'\mbz, \qquad \mbz/(m_G/p)\mbz \simeq \mbz/p^{n-1}\mbz \times \mbz/m_G'\mbz,
\end{equation}
under the isomorphisms of the Chinese remainder theorem.  We note that, by maximality of $G$ among missing trace groups, we must have that $\tr\left( G(m_G/p) \right) = \mbz/(m_G/p)\mbz$.  Fix any $\ol{g} \in G(m_G/p)$ satisfying $\tr \ol{g} \equiv \ol{t}_0 \bmod m_G/p$.  We claim that $\ol{g} \equiv \gl I \bmod p$ for some $\gl \in (\mbz/p\mbz)^\times$.  To see this, let $\pi_G : G(m_G) \longrightarrow G(m_G/p)$ denote the restriction of the canonical projection map and consider the fiber
\[
\pi_G^{-1}(\ol{g}) \subseteq G(m_G).
\] 
Observe that, since $m_S$ divides $m_G/p$, by the same reasoning as in \eqref{kerpiGequalsSL2kernel}, we have 
\[
\ker \pi_G = \ker\left( \SL_2(\mbz/m_G\mbz) \to \SL_2(\mbz/(m_G/p)\mbz) \right).  
\]
Thus, denoting by $\ol{g} \mapsto (\ol{g}_p, \ol{g}_{m_G'})$ (resp. $g \mapsto (g_p, \ol{g}_{m_G'})$) the associated element under the embedding $G(m_G/p) \hookrightarrow G(p^{n-1}) \times G(m_G')$ (resp. under $G(m_G) \hookrightarrow G(p^{n}) \times G(m_G')$)  induced by \eqref{CRTisom}, we have
\[
\pi_G^{-1}(\ol{g}) = \left\{ \left( g_p \left( I + p^{n-1} \begin{pmatrix} x & y \\ z & -x \end{pmatrix} \right) \bmod p^n, \ol{g}_{m_G'} \right) \in G(p^n) \times G(m_G') \right\},
\]
where $g_p \in G(p^n)$ is any fixed lift of $\ol{g}_p \in G(p^{n-1})$.  Writing $\ds g_p = \begin{pmatrix} a & b \\ c & d \end{pmatrix}$, we see that
\[
\tr \left( g_p \left( I + p^{n-1} \begin{pmatrix} x & y \\ z & -x \end{pmatrix} \right) \right) = a + d + p^{n-1} \left( (a-d)x + bz + cy \right),
\]
and it follows that, if $g_p$ is not congruent to a scalar matrix modulo $p$, then 
\[
\tr\left( \pi_G^{-1}(\ol{g}) \right) = \{ t \in \mbz/m_G\mbz : t \equiv \ol{t}_0 \bmod m_G/p \}, 
\]
contradicting \eqref{tobecontradicted}.  Therefore $g_p$, whence $\ol{g}$, must be congruent to a scalar matrix modulo $p$.  Since $p$ and $\ol{g}$ were arbitrary, it follows by taking $t$ in Definition \ref{defofdadmissible} to be $t_0 \bmod m_S$ that $\hat{G}$ is $m_G/m_S$-admissible, as asserted.
\end{proof}

\subsection{An algorithm that finds all maximal missing trace groups associated to a fixed genus}

Based on the previous analysis (especially \eqref{boundforSL2levels} and Propositions \ref{twotothekproposition} and \ref{admissiblegroupsprop}) we obtain the following algorithm that explicitly computes the set $\mf{G}_{MT}^{\max}(0)$.  We note that if one replaces \eqref{boundforSL2levels} with the set $\{ \level_{\SL_2}(G) : G \in \mf{G}(g) \}$, the same algorithm will explicitly compute the set  $\mf{G}_{MT}^{\max}(g)$.  The levels $m \in  \{ 48, 96 \}$ are computationally intensive, and we include a separate remark about how to handle them.

\begin{center}
For $m \in \mc{CP}_0 := \left\{ \begin{matrix} 1, 2, 3, 4, 5, 6, 7, 8, 9, 10, 11, 12, 13, 14, 15, \\ 16, 18, 20, 21, 24, 25, 26, 27, 28, 30, 32, 36, 48 \end{matrix} \right\}$, do the following:
\end{center}
\begin{enumerate}
\item Exhibit all the groups $G \in \mf{G}(0,m)$ in a list $\mc{L}_m$.
\item Form the sublist $\mc{L}_m^{MT}$ of all ``new missing trace groups'' in $\mc{L}_m$, i.e. $\mc{L}_m^{MT}$ consists of all $G \in \mc{L}_m$ which satisfy $\tr\left( G(m) \right) \subsetneq \mbz/m\mbz$ and, for each prime $p$ dividing $m$, $\tr\left( G(m/p) \right) = \mbz/(m/p)\mbz$.
\item Form the sublist $\mc{L}_{m,\max}^{MT}$ of all $G$ appearing in $\mc{L}_m^{MT}$ which are maximal (relative to $\dot\subseteq$) among all the groups appearing in $\mc{L}_m^{MT}$.
\end{enumerate}
Each list $\mc{L}_{m,\max}^{MT}$ constructed via (1) through (3) exhibits the groups $G \in \mf{G}_{MT}^{\max}(0,m)$ for 
$m \in \mc{CP}_0$. 
It remains to exhibit all groups $G \in \mf{G}_{MT}^{\max}(0,m)$ for $m \notin \mc{CP}_0$.  We proceed by cases, according to whether or not each of the conditions $\level_{\SL_2}(G) \in \mc{CP}_0$ and $\level_{\SL_2}(G) < \level_{\GL_2}(G)$ is satisfied.  (We just handled the case $\level_{\SL_2}(G) \in \mc{CP}_0$ and $\level_{\SL_2}(G) = \level_{\GL_2}(G)$, so there are 3 more cases to consider.)  To deal with the case $\level_{\SL_2}(G) \in \mc{CP}_0$ and $\level_{\SL_2}(G) < \level_{\GL_2}(G)$, we perform step ($1$) above and then proceed as follows:
\begin{enumerate}
\item[($2'$)] Form the sublist $\mc{L}_m^{\non MT}$ of $\mc{L}_m$ consisting of all ``non-missing trace groups'' (i.e. groups $\hat{G}$ for which $\tr\left( \hat{G}(m) \right) = \mbz/m\mbz$).
\item[($3'$)] For each $\hat{G}$ in the list $\mc{L}_m^{\non MT}$, determine all divisors $d$ of $d_{\hat{G}}$ for which $\hat{G}$ is $d$-admissible and store the pair $\langle \hat{G}, d \rangle$ in a new list $\mc{L}_{m,\ad}^{\non MT}$ just in case $d > 1$.
\item[($4'$)] For each $\langle \hat{G}, d \rangle \in \mc{L}_{m,\ad}^{\non MT}$, search for new missing trace subgroups $G \subseteq \pi^{-1}\left( \hat{G}(m) \right)$ of genus zero and $\GL_2$-level $dm$ satisfying $G(m) = \hat{G}(m)$.  Form a new list $\mc{H}_m$ consisting of any such groups $G$ that turn up.
\item[($5'$)] Form the sublist $\mc{H}_m^{\max}$ of $G$ appearing in $\mc{H}_m$ that are maximal relative to $\dot\subseteq$ among all the groups in $\mc{H}_m$.
\end{enumerate}
To handle the case $\level_{\SL_2}(G) \notin \mc{CP}_0$ and $\level_{\SL_2}(G) = \level_{\GL_2}(G)$, we perform steps ($1$) and ($2'$) above, and then proceed as follows:
\begin{enumerate}
\item[($3''$)] Form the sublist $\mc{L}_{m,-I \in}^{\non MT}$ of $\mc{L}_m^{\non MT}$ consisting of those groups $\hat{G}$ for which $-I \in \hat{G}$.
\item[($4''$)] For each $\hat{G}$ in the list $\mc{L}_{m,-I \in}^{\non MT}$, find all index two subgroups $G \in \pi^{-1}\left( G(m) \right)$ with $-I \notin G$, $\det G = \hat{\mbz}^\times$, $\level_{\GL_2}(G) = 2m$ and $G(m) = \hat{G}(m)$.  Collect these into a list $\mc{L}_{2m}$.
\item[($5''$)] Form the sublist $\mc{L}_{2m}^{MT}$ of all new missing trace groups in $\mc{L}_{2m}$, i.e. $\mc{L}_{2m}^{MT}$ consists of all $G$ in $\mc{L}_{2m}$ which satisfy $\tr\left( G(2m) \right) \subsetneq \mbz/2m\mbz$ and, for each prime $p$ dividing $2m$, $\tr\left( G(2m/p) \right) = \mbz/(2m/p)\mbz$.
\item[($6''$)] Form the sublist $\mc{L}_{2m,\max}^{MT}$ of all $G$ appearing in $\mc{L}_{2m}^{MT}$ which are maximal (relative to $\dot\subseteq$) among all the groups appearing in $\mc{L}_{2m}^{MT}$.
\end{enumerate}
 For the final the case $\level_{\SL_2}(G) \notin \mc{CP}_0$ and $\level_{\SL_2}(G) < \level_{\GL_2}(G)$, we perform steps ($1$), ($2'$), ($3''$) and ($4''$) above, and then proceed as follows:
\begin{enumerate}
\item[($5'''$)] Form the sublist $\mc{L}_{2m}^{\non MT}$ of $\mc{L}_{2m}$ consisting of all non-missing trace groups (i.e. groups $\hat{G}$ for which $\tr\left( \hat{G}(2m) \right) = \mbz/2m\mbz$).
\item[($6'''$)] For each $\hat{G}$ in the list $\mc{L}_{2m}^{\non MT}$, determine all divisors $d$ of $d_{\hat{G}}$ for which $\hat{G}$ is $d$-admissible and store the pair $\langle \hat{G}, d \rangle$ in a new list $\mc{L}_{2m,\ad}^{\non MT}$ just in case $d > 1$.
\item[($7'''$)] For each $\langle \hat{G}, d \rangle \in \mc{L}_{m,\ad}^{\non MT}$, search for new missing trace subgroups $G \subseteq \pi^{-1}\left( \hat{G}(2m) \right)$ of genus zero and $\GL_2$-level $d \cdot 2m$ satisfying $G(2m) = \hat{G}(2m)$.  Form a new list $\mc{H}_{2m}$ consisting of any such groups $G$ that turn up.
\item[($8'''$)] Form the sublist $\mc{H}_{2m}^{\max}$ of $G$ appearing in $\mc{H}_{2m}$ that are maximal relative to $\dot\subseteq$ among all the groups in $\mc{H}_{2m}$.
\end{enumerate}
We remark that, for the level $m = 48$, the large number of subgroups of $\SL_2(\mbz/48\mbz)$ makes it computationally intensive (memory-wise) to determine the genus zero subgroups via a direct search.  To work around this problem, we instead started with the single unique subgroup $S(48) \subseteq \SL_2(\mbz/48\mbz)$ of $\SL_2$-level $48$ and genus zero and constructed all other subgroups $G(48) \subseteq \GL_2(\mbz/48\mbz)$ by initializing $G_0 := S(48)$ and recursively defining $G_{n+1} := \langle g, G_n \rangle$ with $g \in \GL_2(\mbz/48\mbz)$ any matrix for which $G_{n+1} \cap \SL_2(\mbz/48\mbz) = S(48)$.  This process generates all subgroups $G(48) \subseteq \GL_2(\mbz/48\mbz)$ of $\SL_2$-level $48$ and genus zero.

Running the above algorithm, we obtain $\mf{G}_{MT}^{\max}(0)$ as an explicit list of $52$ groups which are presented in Section \ref{proofofmaintheoremsection}, grouped in subsections according to $\GL_2$-level.  For exactly two of these groups $G$ (of $\GL_2$-levels $8$ and $9$), the corresponding modular curves $X_G$ satisfy $X_G(\mbq) = \emptyset$; the $50$ remaining groups give rise to Weierstrass models with $j$-invariants and twist parameters as described in Tables \ref{masterlistofjinvariants} and \ref{masterlistoftwistparameters} of Section \ref{tablesection}.

\section{Preliminaries on division fields of elliptic curves} \label{preliminariesondivisionfieldssection}

In this section, we gather various preliminary results to be used in the following section to develop explicit models for the modular curves associated to the groups occurring in Theorem \ref{boundingthelevelsthm}.  We recall the general set-up:  $G \subseteq \GL_2(\hat{\mbz})$ is an open subgroup, $\tilde{G} := \langle G, -I \rangle$ and $X_{\tilde{G}}$ is the modular curve associated to $\tilde{G}$.  
Denoting by $m$ the $\GL_2$-level of $\tilde{G}$ and by $j_{\tilde{G}} : X_{\tilde{G}} \longrightarrow X(1) \simeq \mbp^1$ the forgetful map, we have that, for any $j \in \mbq - \{ 0, 1728 \}$, 
\begin{equation} \label{interpofrationalpointsonXsubG}
j \in j_{\tilde{G}}\left( X_{\tilde{G}}(\mbq) \right) \; \Longleftrightarrow \; \exists E / \mbq \; \text{ with } \;\rho_E(G_\mbq) \, \dot\subseteq \, \tilde{G} \; \text{ and } \; j_E = j.
\end{equation}
We will use repeatedly the following consequence of the Weil pairing on an elliptic curve.  Let $K$ be a field, let $E$ an elliptic curve defined over $K$, let $m \in \mbn$ be a positive integer co-prime with $\Char K$ and let
\[
\rho_{E,m} : G_K \longrightarrow \aut(E[m]) \simeq  \GL_2(\mbz/m\mbz)
\]
be the Galois representation defined by letting $G_K$ act on the $m$-torsion of $E$.  On the other hand, let
\[
\chi_m : G_K \longrightarrow \aut(\mu_m) \simeq \left( \mbz/m\mbz \right)^\times
\]
be the mod $m$ cyclotomic character, defined by letting $G_K$ act on $\mu_m$.
\begin{lemma} \label{weilpairinglemma}
We have $\det \circ \rho_{E,m} = \chi_m$.
\end{lemma}
\begin{proof}
This follows from properties of the Weil pairing; see \cite[Ch III, \S 8]{silverman}.
\end{proof}
We remark that, whenever the level $m$ of a maximal missing trace group may be written in the form $m = m_1 m_2$ with $m_1, m_2 > 1$ and $\gcd(m_1,m_2) = 1$, the group $G(m)$ decomposes via the Chinese Remainder Theorem as a fibered product
\begin{equation} \label{fiberedproductdecompofG}
G(m) = G(m_1m_2) \simeq G(m_1) \times_{\psi} G(m_2) := \{ (g_1, g_2) \in G(m_1) \times G(m_2) : \; \psi_1(g_1) = \psi_2(g_2) \},
\end{equation}
where $\psi_i : G(m_i) \longrightarrow H$ are each surjective homomorphisms onto a common non-trivial quotient group $H$ (if $H$ were trivial, then $G(m)$ would decompose as a full cartesian product, implying that either $G(m_1)$ or $G(m_2)$ would have a missing trace, contradicting maximality of $G$).  

There are various methods for obtaining an explicit model for such a modular curve $X_{\tilde{G}(m)}$.  For example, one can use Siegel functions as was done in \cite{danielsmorrow} and \cite{sutherlandzywina}, but we will proceed differently.  By Theorem \ref{boundingthelevelsthm}, the levels $m_1$ and $m_2$ in \eqref{fiberedproductdecompofG} are always equal to prime powers in the cases we consider.  We will work directly with existing models for $X_{\tilde{G}(m_1)}$ and $X_{\tilde{G}(m_2)}$ coming from \cite{sutherlandzywina}, constructing the desired cover $X_{\tilde{G}(m)}$ by 
explicitly computing appropriate subfields of $L_1 \subseteq \mbq(E[m_1])$ and $L_2 \subseteq \mbq(E[m_2])$ and setting them equal to one another.  This highlights the role played by entanglements in these examples.

The following lemma will be useful for determining when $\rho_{E,m_1m_2}(G_\mbq) \, \dot\subseteq \, G(m_1) \times_\psi G(m_2)$.  In general, let $\mc{G}_1$ and $\mc{G}_2$ be finite groups and let 
\[
G_i, G_i' \subseteq \mc{G}_i \quad\quad \left( i \in \{1, 2 \} \right)
\]
be subgroups.  Let
\[
\psi_i : G_i \longrightarrow H, \quad \psi_i' : G_i' \longrightarrow H' \quad\quad \left( i \in \{1, 2 \} \right)
\]
be surjective homomorphisms onto the finite groups $H$ and $H'$, respectively, and denote by $G_1 \times_\psi G_2$ and $G_1' \times_{\psi'} G_2'$ the corresponding fibered products.  In particular, these fibered products are \emph{honest} in the sense that, for each $i \in \{1, 2 \}$, the canonical projections $G_1 \times_\psi G_2 \longrightarrow G_i$ and $G_1' \times_{\psi'} G_2' \longrightarrow G_i'$ are each surjective.
\begin{lemma} \label{whenisgroupinsidefiberedproduct}
With the notation just outlined, we have
\[
G_1' \times_{\psi'} G_2' \subseteq G_1 \times_{\psi} G_2
\]
if and only if
\begin{enumerate}
\item $\forall i \in \{ 1, 2 \}, \, G_i' \subseteq G_i$,
\item there exists a group homomorphism $\varpi : H' \longrightarrow H$ such that $\forall i \in \{1, 2 \}, \; \varpi \circ \psi_i' = \psi_i \vert_{G_i'}$.
\end{enumerate}
\end{lemma}
\begin{proof}
For the direction ``$\Rightarrow$'', condition (1) is immediate.  To verify condition (2), we first note that $\ker \psi_1' \times \ker \psi_2' \subseteq G_1 \times_\psi G_2$, which implies that
\[
\ker \psi_i' \subseteq \ker \psi_i \quad\quad \left( i \in \{1, 2 \} \right).
\]
We thus have a well-defined group homomorphism $\varpi : H' \longrightarrow H$, given by $\varpi \left( \psi_1'(g_1') \right) := \psi_1(g_1')$.
We observe that $\varpi \circ \psi_i' = \psi_i \vert_{G_i'}$ by definition when $i = 1$ and by $G_1' \times_{\psi'} G_2' \subseteq G_1 \times_{\psi} G_2$ when $i = 2$.  This establishes (2).

For the converse, assume that (1) and (2) hold and let $(g_1', g_2') \in G_1' \times_{\psi'} G_2'$.  By (1), $g_1' \in G_1$ and $g_2' \in G_2$.  Furthermore, $\psi_1'(g_1') = \psi_2'(g_2')$, and thus
\[
\psi_1(g_1') = \varpi \left( \psi_1'(g_1') \right) = \varpi \left( \psi_2'(g_2') \right) = \psi_2(g_2').
\]
Thus, $(g_1',g_2') \in G_1 \times_\psi G_2$, and we have proved the converse, establishing the lemma.
\end{proof}

By \eqref{interpofrationalpointsonXsubG}, we need only consider such groups up to conjugation inside $\GL_2(\mbz/m\mbz)$, i.e. up to the relation $\doteq$.  It is straightforward to see that, for any inner automorphisms $\eta_1, \eta_2 : H \longrightarrow H$, we have
\[
G(m_1) \times_{(\eta_1\psi_1, \eta_2 \psi_2)} G(m_2) \doteq G(m_1) \times_{(\psi_1, \psi_2)} G(m_2).
\]
However, the same is not always true for outer automorphisms $\eta_1, \eta_2 \in \aut(H)$; given such a pair $(\eta_1,\eta_2) \in \aut(H)^2$, we clearly have
\[
G(m_1) \times_{(\eta_1 \psi_1, \eta_2 \psi_2)} G(m_2) = G(m_1) \times_{(\psi_1, \eta_1^{-1} \eta_2 \psi_2)} G(m_2).
\]
Thus, it suffices to consider postcomposing the pair $(\psi_1,\psi_2)$ with pairs of automorphsims of the form $(1,\eta_2)$, or of the form $(\eta_1,1)$.  
\begin{Definition} \label{defofGL2induced}
Given the notation above, we say that an automorphism $\eta_i \in \aut(H)$ \emph{\textbf{is $\GL_2(\mbz/m_i\mbz)$-induced}} if there exists $g_i \in \GL_2(\mbz/m_i\mbz)$ satisfying $g_i G(m_i) g_i^{-1} = G(m_i)$ and for which the diagram
\[
\begin{tikzcd}
G(m_i) \rar{\conj_{g_i}} \dar{\psi_i} & G(m_i) \dar{\psi_i} \\
H \rar{\eta_i} & H
\end{tikzcd}
\]
commutes.  
\end{Definition}
The following useful lemma is straightforward to prove.
\begin{lemma} \label{GL2inducedlemma}
With notation as just outlined and with $\eta_1, \eta_2 \in \aut(H)$, we have
\[
\begin{split}
&G(m_1) \times_{(\psi_1,\eta_2\psi_2)} G(m_2) \doteq G(m_1) \times_{(\psi_1,\psi_2)} G(m_2) \; \Longleftrightarrow \; \eta_2 \in \aut(H) \text{ is $\GL_2(\mbz/m_2\mbz)$-induced,} \\
&G(m_1) \times_{(\eta_1\psi_1,\psi_2)} G(m_2) \doteq G(m_1) \times_{(\psi_1,\psi_2)} G(m_2) \; \Longleftrightarrow \; \eta_1 \in \aut(H) \text{ is $\GL_2(\mbz/m_1\mbz)$-induced.} 
\end{split}
\]
\end{lemma}
We will now describe an interpretation of $\rho_{E,m_1m_2}(G_\mbq) \, \dot\subseteq \, G(m_1) \times_\psi G(m_2)$ in terms of entanglements.  It will be convenient to decompose the representation $\rho_{E}$ as
\begin{equation*}
\begin{tikzcd}
G_\mbq \rar{\tilde{\rho}_{E}} \arrow[black, bend left]{rr}{\rho_E} & \aut(E_{\tors}) \rar{\iota_{\mc{B}}} & \GL_2(\hat{\mbz}),
\end{tikzcd}
\end{equation*}
where $\mc{B} = \{ \mc{B}(m) := ({\bf{b}}_{1,m}, {\bf{b}}_{2,m}) : m \in \mbn \}$ is a collection of ordered $\mbz/m\mbz$-bases of $E[m] \subseteq E_{\tors}$, one for each $m \in \mbn$, chosen compatibly.  This decomposition has a corresponding a finite level analogue 
\begin{equation}  \label{defofrhotilde}
\rho_{E,m} = \iota_{\mc{B}(m)} \circ \tilde{\rho}_{E,m}
\end{equation}
for any $m \in \mbn$.  We will denote simply by $\iota$ either of the the isomorphisms $\aut(E[m]) \to \GL_2(\mbz/m\mbz)$ or $\aut(E_{\tors}) \to \GL_2(\hat{\mbz})$ induced by such a collection $\mc{B}$, suppressing the dependence on $\mc{B}$.  Thus, for any $m \in \mbn$, we have
\[
\rho_{E,m}(G_\mbq) \, \dot \subseteq \, G(m) \; \Longleftrightarrow \; \exists \iota : \tilde{\rho}_{E,m}(G_\mbq) \hookrightarrow G(m).
\]

Unraveling \eqref{interpofrationalpointsonXsubG}, we find that
\[
\exists \iota : \tilde{\rho}_{E,m_1m_2}(G_\mbq) \hookrightarrow G(m_1) \times_\psi G(m_2) \; \Longrightarrow \; \begin{matrix} \exists \iota_1 : \tilde{\rho}_{E,m_1}(G_\mbq) \hookrightarrow G(m_1), \; \exists \iota_2 : \tilde{\rho}_{E,m_2}(G_\mbq) \hookrightarrow G(m_2), \\ \text{and } \; \mbq(E[m_1])^{\iota_1^{-1}(\ker \psi_1)} = \mbq(E[m_2])^{\iota_2^{-1}(\ker \psi_2)},\end{matrix}
\]
where we are understanding the subfield $\mbq(E[m_i])^{\iota_i^{-1}(\ker \psi_i)} \subseteq \mbq(E[m_i])$ via the natural isomorphism $\gal\left( \mbq(E[m_i])/\mbq \right) \simeq \tilde{\rho}_{E,m_i}(G_\mbq)$.
The following corollary states conditions under which the converse holds. 
\begin{cor} \label{justentanglementequalitycorollary}
Let $\aut_{\GL_2(\mbz/m_i\mbz)}(H) \subseteq \aut(H)$ denote the subgroup of $\GL_2(\mbz/m_i\mbz)$-induced automorphisms, and suppose that either $\aut_{\GL_2(\mbz/m_1\mbz)}(H) = \aut(H)$ or $\aut_{\GL_2(\mbz/m_2\mbz)}(H) = \aut(H)$.  We then have
\[
\rho_{E,m_1m_2}(G_\mbq) \, \dot\subseteq \, G(m_1) \times_\psi G(m_2) \; \Longleftrightarrow \; \begin{matrix} \exists \iota_1 : \tilde{\rho}_{E,m_1}(G_\mbq) \hookrightarrow G(m_1), \; \exists \iota_2 : \tilde{\rho}_{E,m_2}(G_\mbq) \hookrightarrow G(m_2), \\ \text{and } \quad \mbq(E[m_1])^{\iota_1^{-1}(\ker \psi_1)} = \mbq(E[m_2])^{\iota_2^{-1}(\ker \psi_2)}. \end{matrix}
\]
\end{cor}
We would like to apply Corollary \ref{justentanglementequalitycorollary} to groups $G \in \mf{G}_{MT}^{\max}(0)$.  Our computation shows that, for every such group $G$ whose $\GL_2$-level $m$ is divisible by at least two primes, and for any $m_1, m_2 > 1$ with $m = m_1m_2$ and $\gcd(m_1,m_2) = 1$, writing $G(m) \simeq G(m_1) \times_{\psi} G(m_2)$ and denoting by $H := \psi_i(G(m_i))$ the common quotient group implicit in the fibered product, we have
\begin{equation} \label{formsofH}
H \in \{ \mbz/2\mbz, \mbz/3\mbz, \mbz/6\mbz, S_3 \}.
\end{equation}
When $H \in \{ \mbz/2\mbz, S_3 \}$, all automorphisms of $H$ are inner, whence $\GL_2(\mbz/m_i\mbz)$-inner (no matter what the level $m_i$ is.  For the case $H \in \{ \mbz/3\mbz, \mbz/6\mbz \}$,  we will now look in more detail at the particulars.  

The common quotient $H = \mbz/3\mbz$ arises as an entanglement between the groups $G(2)$ and $G(7)$, and in all such cases, we have
\begin{equation} \label{descriptionofpsisub2in14case}
G(2) = \left\langle \begin{pmatrix} 1 & 1 \\ 1 & 0 \end{pmatrix} \right\rangle \subseteq \GL_2(\mbz/2\mbz), \quad\quad \psi_2 :  \left\langle \begin{pmatrix} 1 & 1 \\ 1 & 0 \end{pmatrix} \right\rangle \simeq \mbz/3\mbz.
\end{equation}
Thus, the image $H = \mbz/3\mbz$ is isomorphic to $G(2)$, an index two (and hence normal) subgroup of $\GL_2(\mbz/2\mbz)$.  Since
\[
\begin{pmatrix} 0 & 1 \\ 1 & 0 \end{pmatrix} \begin{pmatrix} 1 & 1 \\ 1 & 0 \end{pmatrix} \begin{pmatrix} 0 & 1 \\ 1 & 0 \end{pmatrix}^{-1} = \begin{pmatrix} 0 & 1 \\ 1 & 1 \end{pmatrix},
\]
it follows that the conjugation action of $\GL_2(\mbz/2\mbz)$ on $G(2)$ gives rise to all automorphisms of $H$, i.e. that $\aut(H) = \aut_{\GL_2(\mbz/2\mbz)}(H)$ in this case.

The common quotient $H \simeq \mbz/6\mbz$ arises as an entanglement between $G(4)$ and $G(7)$, and in all such cases, we have
\[
G(4) = \pi_{\GL_2}^{-1}\left( \left\langle \begin{pmatrix} 1 & 1 \\ 1 & 0 \end{pmatrix} \right\rangle \right) \subseteq \GL_2(\mbz/4\mbz), \quad\quad \ker \psi_4 = \ker \pi_{\GL_2} \cap \SL_2(\mbz/4\mbz),
\]
where $\pi_{\GL_2} : \GL_2(\mbz/4\mbz) \to \GL_2(\mbz/2\mbz)$ denotes the canonical projection map.  Therefore the map $\psi_4$ decomposes as
\[
\begin{tikzcd}
\pi_{\GL_2}^{-1}\left( \left\langle \begin{pmatrix} 1 & 1 \\ 1 & 0 \end{pmatrix} \right\rangle \right) \rar{\psi_2 \times \det} \arrow[black, bend left]{rr}{\psi_4} & \mbz/3\mbz \times \mbz/2\mbz \rar{\simeq} & \mbz/6\mbz,
\end{tikzcd}
\]
where $\psi_2$ denotes the function $g \mapsto \psi_2(g \mod 2)$, with $\psi_2$ as in \eqref{descriptionofpsisub2in14case}.  Thus, it follows from the discussion in the previous paragraph that $\aut_{\GL_2(\mbz/4\mbz)}(H) = \aut(H)$ in this case as well.  Taking these observations together with the computation that establishes \eqref{formsofH}, we thus have
\begin{cor} \label{keycorollaryforinterpretationofentanglements}
For each group $G \in \mf{G}_{MT}^{\max}(0)$ with the property that $m := \level_{\GL_2}(G)$ is divisible by at least two primes, choose any $m_1, m_2 \in \mbn$ with $\gcd(m_1,m_2) = 1$ and $m_1, m_2 > 1$ and write $G(m) \simeq G(m_1) \times_{\psi} G(m_2)$.  For any elliptic curve $E$ over $\mbq$, we have
\[
\rho_{E,m_1m_2}(G_\mbq) \, \dot\subseteq \, G(m_1) \times_\psi G(m_2) \; \Longleftrightarrow \; \begin{matrix} \exists \iota_1 : \tilde{\rho}_{E,m_1}(G_\mbq) \hookrightarrow G(m_1), \; \exists \iota_2 : \tilde{\rho}_{E,m_2}(G_\mbq) \hookrightarrow G(m_2), \\ \text{and } \quad \mbq(E[m_1])^{\iota_1^{-1}(\ker \psi_1)} = \mbq(E[m_2])^{\iota_2^{-1}(\ker \psi_2)},
\end{matrix}
\]
where the basis-induced embeddings $\iota_1$ and $\iota_2$ and the representations $\tilde{\rho}_{E,m_i}$ are as in \eqref{defofrhotilde}, and the subfield $\mbq(E[m_i])^{\iota_i^{-1}(\ker \psi_i)} \subseteq \mbq(E[m_i])$ is understood via the natural isomorphism $\tilde{\rho}_{E,m_i}(G_\mbq) \simeq \gal\left( \mbq(E[m_i])/\mbq \right)$.
\end{cor}

\subsection{Cyclic cubic fields}

Since we will be dealing with cyclic cubic extensions of $\mbq(t,D)$, we now state two lemmas about such field extensions that will be used in what follows.
Our first lemma allows us to exhibits explicit polynomials for generators of each of the cyclic cubic  subfields of a given $\mbz/3\mbz \times \mbz/3\mbz$-extension.  In general, let $K$ be any field, let
\begin{equation} \label{defoffsubSandfsubT}
\begin{split}
f_S(x) &= x^3 - S_1x^2 + S_2 x - S_3, \\
f_T(x) &= x^3 - T_1x^2 + T_2 x - T_3
\end{split}
\end{equation}
be two monic irreducible polynomials with coefficients in $K$, and denote by $K_{f_S}$ (resp. by $K_{f_T}$) the splitting polynomial of $f_S$ (resp. of $f_T$), viewed as subfields of a fixed algebraic closure $\ol{K}$ of $K$.  Assume that 
\begin{equation} \label{KfsubSdifferentfromKfsubT}
K_{f_S} \neq K_{f_T}
\end{equation}
and that the discriminants $\gD_S := \disc(f_S)$ and $\gD_T := \disc(f_T)$ are each in $\left( K^\times \right)^2$, or equivalently that $\gal(K_{f_S}/K) \simeq \gal(K_{f_T}/K)$ is a cyclic group of order $3$.  The assumption \eqref{KfsubSdifferentfromKfsubT} then implies that the composite field $K_{f_Sf_T} = K_{f_S}K_{f_T}$ satisfies
\begin{equation} \label{galoisgroupisZmod3timesZmod3}
\gal(K_{f_Sf_T}/K) \simeq \mbz/3\mbz \times \mbz/3\mbz,
\end{equation}
and this field contains $4$ cyclic cubic subfields.  Fix square roots $\sqrt{\gD_S}, \sqrt{\gD_T} \in K$ and define the coefficients $R_1, R_2, R_3, R_3' \in K$ by
\begin{equation} \label{defofRsubis}
\begin{split}
R_1 := &S_1 T_1, \\
R_2 := &S_1^2 T_2 + T_1^2 S_2 - 3S_2T_2, \\
R_3 := &S_1^3 T_3 + T_1^3 S_3 - 3S_1 S_2 T_3 - 3 T_1 T_2 S_3 + 9S_3 T_3 \\
&+ \left( S_1S_2 - 3S_3 + \sqrt{\gD_S} \right) \left( T_1 T_2 - 3T_3 + \sqrt{\gD_T} \right) / 4 \\
&+ \left( S_1S_2 - 3S_3 - \sqrt{\gD_S} \right) \left( T_1 T_2 - 3T_3 - \sqrt{\gD_T} \right) / 4, \\
R_3' := &S_1^3 T_3 + T_1^3 S_3 - 3S_1 S_2 T_3 - 3 T_1 T_2 S_3 + 9S_3 T_3 \\
&+ \left( S_1S_2 - 3S_3 + \sqrt{\gD_S} \right) \left( T_1 T_2 - 3T_3 - \sqrt{\gD_T} \right) / 4 \\
&+ \left( S_1S_2 - 3S_3 - \sqrt{\gD_S} \right) \left( T_1 T_2 - 3T_3 + \sqrt{\gD_T} \right) / 4.
\end{split}
\end{equation}
Define the cubic polynomials $f_R(x), f_{R'}(x) \in K[x]$ by 
\begin{equation} \label{defoffsubRandfsubRprime}
\begin{split}
f_R(x) &:= x^3 - R_1x^2 + R_2x - R_3, \\
f_{R'}(x) &:= x^3 - R_1x^2 + R_2x - R_3'.
\end{split}
\end{equation}
\begin{lemma} \label{gettingattheothercubicfieldslemma}
Let $K$ be a field, let $f_S(x), f_T(x) \in K[x]$ be irreducible monic cubic polynomials as in \eqref{defoffsubSandfsubT}, and assume the setup and notation laid out above (in particular, assume that the splitting field $K_{f_Sf_T}$ of $f_S(x)f_T(x)$ satisfies \eqref{galoisgroupisZmod3timesZmod3}).  Then the four cyclic cubic subfields of $K_{f_Sf_T}$ are the splitting fields of the polynomials $f_S(x)$, $f_T(x)$, $f_R(x)$ and $f_{R'}(x)$, where $f_R(x)$ and $f_{R'}(x)$ are defined by \eqref{defoffsubRandfsubRprime} and \eqref{defofRsubis}.
\end{lemma}
\begin{proof}
An exercise in symmetric polynomials.
\end{proof}

Given a field $K$ and elements $S_1, S_2, S_3, T_1, T_2, T_3 \in K$, we may consider the following system of equations in the variables $a$, $b$ and $c$.
\begin{equation} \label{abcequations}
\begin{split}
T_1 = &a \left( S_1^2 - 2S_2 \right) + b S_1 + 3c, \\
T_2 = &a^2 S_2^2 - 2a^2 S_1 S_3 + ab S_1 S_2 - 3abS_3 + 2acS_1^2 - 4acS_2 + b^2 S_2 + 2bc S_1 + 3c^2, \\
T_3 = &a^3 S_3^2 + a^2b S_2 S_3 - 2a^2c S_1 S_3 + a^2c S_2^2 + ab^2 S_1 S_3 + abc S_1 S_2 - 3abc S_3 \\
&+ ac^2 S_1^2 - 2ac^2 S_2 + b^3 S_3 + b^2 c S_2 + bc^2 S_1 + c^3.
\end{split}
\end{equation}

\begin{lemma} \label{settingcubicfieldsequaltoeachotherlemma}
Let $K$ be a field, let $f_S(x)$, $f_T(x) \in K[x]$ be irreducible monic cubic polynomials as in \eqref{defoffsubSandfsubT} and denote by $K_{f_S}$ (resp. by $K_{f_T}$) the splitting field of $f_S(x)$ (resp. of $f_T(x)$), viewed as subfields of a fixed algebraic closure $\ol{K}$ of $K$.  Assume that the discriminant $\gD_S$ of $f_S(x)$ satisfies $\gD_S \in \left( K^\times \right)^2$, so that $\gal(K_{f_S}/K)$ is a cyclic group of order $3$. We then have that $K_{f_S} = K_{f_T}$ if and only if the system of equations \eqref{abcequations} has a solution $(a,b,c) \in K^3$.
\end{lemma}
\begin{proof}
Let $\ga \in \ol{K}$ denote a root of $f_S(x)$.  Then $K_{f_S} = K(\ga)$, and so we may write an arbitrary element $\gb \in K_{f_S}$ in the form
\begin{equation} \label{betaasalgebraicexpressioninalpha}
\gb = a \ga^2 + b\ga + c \quad\quad \left( a, b, c \in K \right).
\end{equation}
The elementary symmetric polynomials $E_i(a \ga^2 + b\ga + c)$ of such an element are then readily computed to be
\begin{equation} \label{symmetricpolynomialequalities}
\begin{split}
E_1(a \ga^2 + b\ga + c) = &a \left( S_1^2 - 2S_2 \right) + b S_1 + 3c, \\
E_2(a \ga^2 + b\ga + c) = &a^2 S_2^2 - 2a^2 S_1 S_3 + ab S_1 S_2 - 3abS_3 + 2acS_1^2 - 4acS_2 + b^2 S_2 + 2bc S_1 + 3c^2, \\
E_3(a \ga^2 + b\ga + c) = &a^3 S_3^2 + a^2b S_2 S_3 - 2a^2c S_1 S_3 + a^2c S_2^2 + ab^2 S_1 S_3 + abc S_1 S_2 - 3abc S_3 \\
&+ ac^2 S_1^2 - 2ac^2 S_2 + b^3 S_3 + b^2 c S_2 + bc^2 S_1 + c^3.
\end{split}
\end{equation}
Thus, if \eqref{abcequations} have a solution $(a,b,c) \in K^3$, we see that $f_T(x)$ has a root in $K_{f_S}$, thus $f_T(x)$ splits completely over $K_{f_S}$, and so $K_{f_T} = K_{f_S}$.  Conversely, if $K_{f_T} = K_{f_S}$, then let $\gb \in K_{f_T}$ be a root of $f_T(x)$.  Writing $\gb$ in the form \eqref{betaasalgebraicexpressioninalpha}, we see that $(a,b,c) \in K^3$ is then a solution to \eqref{abcequations}.
\end{proof}

\subsection{Division fields of elliptic curves and their subfields} \label{divisionfieldssubsection}

In this section, we exhibit explicitly various subfields of the $m$th division field $\mbq(t,D)\left( \mc{E}[m] \right)$ for various levels $m$ and elliptic curves $\mc{E}$ over $\mbq(t,D)$.  The Borel subgroup
\[
B(\ell) := \left\{ \begin{pmatrix} * & * \\ 0 & * \end{pmatrix} \right\} \subseteq \GL_2(\mbz/\ell\mbz)
\]
plays a key role, as do the two multiplicative homomorphisms $\psi_{\ell,1}, \psi_{\ell,2} : B(\ell) \longrightarrow (\mbz/\ell\mbz)^\times$ defined by
\begin{equation} \label{generaldefinitionofpsisubp}
\psi_{\ell,1}\left( \begin{pmatrix} a & b \\ 0 & d \end{pmatrix} \right) := a, \quad\quad \psi_{\ell,2}\left( \begin{pmatrix} a & b \\ 0 & d \end{pmatrix} \right) := d.
\end{equation}
For any elliptic curve $E$ over $\mbq$, we have
\begin{equation} \label{interpretationofwhenGaloismapsintoBorel}
\rho_{E,\ell}(G_\mbq) \subseteq B(\ell) \; \Longleftrightarrow \; \text{ $\exists$ a $G_\mbq$-stable cyclic subgroup } \langle P \rangle \subseteq E[\ell].
\end{equation}
When this is the case, we denote by $E_{\langle P \rangle}' := E / \langle P \rangle$ the quotient curve, which is isogenous over $\mbq$ to $E$.  We have
$
\rho_{E_{\langle P \rangle}',\ell}(G_\mbq) \subseteq B(\ell),
$
or in other words, 
\begin{equation} \label{existenceofPprime}
\text{ $\exists$ a $G_\mbq$-stable cyclic subgroup } \langle P' \rangle \subseteq E_{\langle P \rangle}'[\ell]
\end{equation}
(this is the kernel of the dual isogeny $E_{\langle P \rangle}' \rightarrow E$).  In these terms, the Galois representations $\psi_{\ell,1}$ and $\psi_{\ell,2}$ above are simply defined by restricting the action of $G_\mbq$ respectively to $\langle P \rangle$ and to $\langle P' \rangle$, i.e. we have
\begin{equation} \label{psisubellsactingonPandPprime}
\gs : P \mapsto \left[ \psi_{\ell,1}(\gs) \right] P, \quad\quad\quad \gs : P' \mapsto \left[ \psi_{\ell,2}(\gs) \right] P' \quad\quad \left( \gs \in G_\mbq \right).
\end{equation}
Finally, we note that $\psi_{\ell,i}^{(\ell-1)/2}\left( g \right) \in \{ \pm 1 \} \subseteq (\mbz/\ell\mbz)^\times$, and that this value agrees with the Legendre symbol evaluated at $\psi_{\ell,i}\left( g \right)$, i.e. we have
\[
\psi_{\ell,1}^{(\ell-1)/2}\left( \begin{pmatrix} a & b \\ 0 & d \end{pmatrix} \right) \equiv \left( \frac{a}{\ell} \right) \mod \ell, \quad\quad \psi_{\ell,2}^{(\ell-1)/2}\left( \begin{pmatrix} a & b \\ 0 & d \end{pmatrix} \right) \equiv \left( \frac{d}{\ell} \right) \mod \ell.
\]

\subsubsection{The level $m = 2$}

The group $\GL_2(\mbz/2\mbz)$ is a non-abelian group of order $6$, and since there is only one such group up to isomorphism, we see that it is isomorphic to the symmetric group of order $6$:  
\begin{equation} \label{GL2Zmod2ZisisomorphictoD3}
\GL_2(\mbz/2\mbz) \simeq S_3.
\end{equation}
As such, there is a unique proper non-trivial normal subgroup 
\[
\left\langle \begin{pmatrix} 1 & 1 \\ 1 & 0 \end{pmatrix} \right\rangle \subseteq \GL_2(\mbz/2\mbz)
\]
which has index two (corresponding under \eqref{GL2Zmod2ZisisomorphictoD3} to the alternating subgroup $A_3$).  This index two subgroup happens to be the commutator subgroup $\SL_2(\mbz/2\mbz)'$ of $\SL_2(\mbz/2\mbz)$; as we shall see, both this fact and the next classical, well-known lemma generalize to levels $3$ and $4$.
\begin{lemma} \label{level2kummersubextensionlemma}
Let $E$ be an elliptic curve over $\mbq$ and let $\gD_E$ denote the discriminant of any\footnote{Note that the field $\mbq(\sqrt{\gD_E})$ (resp. the fields $\mbq(\gD_E^{1/3})$ and $\mbq(\gD_E^{1/4})$ appearing in Lemmas \ref{level3kummersubextensionlemma} and \ref{level4kummersubextensionlemma}) does not depend on the choice of Weierstrass model for $E$, even though the discriminant $\gD_E$ does.} Weierstrass model of $E$.  We have that $\mbq(\sqrt{\gD_E}) \subseteq \mbq(E[2])$.  Furthermore, this subfield corresponds via Galois theory to the subgroup $\rho_{E,2}(G_\mbq) \cap \SL_2(\mbz/2\mbz)'$, i.e. we have
\[
\mbq(E[2])^{\rho_{E,2}(G_\mbq) \cap \SL_2(\mbz/2\mbz)'} = \mbq\left( \sqrt{\gD_E} \right).
\]
\end{lemma}
\begin{proof}
See for instance \cite[pp. 218]{langtrotter}.
\end{proof}
Throughout the paper, the role played by restriction map $\gal\left( \mbq(E[2])/\mbq \right) \rightarrow \gal\left( \mbq\left( \sqrt{\gD_E} \right)/\mbq \right)$ is significant enough to warrant our giving it an explicit name.  We make the definition
\begin{equation} \label{defofve}
\begin{tikzcd}
\ve : \GL_2(\mbz/2\mbz) \rar{\simeq}& S_3 \rar{\can}& \frac{S_3}{A_3} \rar{\simeq}& \{ \pm 1 \}.
\end{tikzcd}
\end{equation} 
Thus, we have
\[
\ker \ve = \left\langle \begin{pmatrix} 1 & 1 \\ 1 & 0 \end{pmatrix} \right\rangle \quad\quad \text{ and } \quad\quad \mbq(E[2])^{\ker \ve} = \mbq\left( \sqrt{\gD_E} \right).
\]

\subsubsection{The level $m = 3$}

Using the classical theory of modular functions (see \cite{zywina}), it can be shown that there is a rational parameter $t$ on the genus zero modular curve $X_0(3)$ such that the forgetful map $X_0(3) \longrightarrow X(1)$ takes the form
\[
\mbp^1(t) \longrightarrow \mbp^1(j), \quad t \mapsto 27 \frac{(t+1)(t+9)^3}{t^3}.
\]
We define $\displaystyle j_3(t) := 27 \frac{(t+1)(t+9)^3}{t^3} \in \mbq(t)$ and the elliptic curve $\mc{E}_{3}$ over $\mbq(t,D)$ by
\begin{equation} \label{level3genericellipticcurve}
\mc{E}_3 : \; y^2 = x^3 + \frac{108D^2 j_3(t)}{1728 - j_3(t)} x + \frac{432D^3 j_3(t)}{1728 - j_3(t)}.
\end{equation}
By restricting the action of $\gal\left( \mbq(t,D)(\mc{E}_3[3])/\mbq(t,D) \right)$ to $\mc{E}_3[3]$ and fixing a $\mbz/3\mbz$-basis thereof, we obtain an isomorphism
\begin{equation} \label{genericGaloisimageiscontainedinborelmod3}
\gal\left( \mbq(t,D)(\mc{E}_3[3])/\mbq(t,D) \right) \simeq \left\{ \begin{pmatrix} * & * \\ 0 & * \end{pmatrix} \right\} \subseteq \GL_2(\mbz/3\mbz).
\end{equation}
The following lemma explicitly characterizes the subfields cut out by the characters $\psi_{3,1}$ and $\psi_{3,2}$.
\begin{lemma} \label{subfieldsoflevel3lemma}
Let $\mc{E}_3$ be the elliptic curve over $\mbq(t,D)$ defined by \eqref{level3genericellipticcurve} and define the characters 
\[
\psi_{3,1}, \psi_{3,2} : \gal\left( \mbq(t,D)(\mc{E}_3[3])/\mbq(t,D) \right) \rightarrow \{ \pm 1 \}
\]
to be the restrictions under \eqref{genericGaloisimageiscontainedinborelmod3} of the characters defined in \eqref{generaldefinitionofpsisubp}.  We then have
\begin{equation} \label{subfieldsoflevel3lemmaeqn}
\begin{split}
\mbq(t,D)(\mc{E}_3[3])^{\ker \psi_{3,1}} &= \mbq(t,D)\left( \sqrt{\frac{6D(t+1)(t+9)}{(t^2-18t-27)}} \right), \\
\mbq(t,D)(\mc{E}_3[3])^{\ker \psi_{3,2}} &= \mbq(t,D)\left( \sqrt{-\frac{2D(t+1)(t+9)}{(t^2-18t-27)}} \right).
\end{split}
\end{equation}
\end{lemma}
\begin{proof}
We compute that the $3$rd division polynomial associated to $\mc{E}_3$ has the $\mbq(t,D)$-rational factor 
\[
x - \frac{18D(t+1)(t+9)}{t^2 - 18t - 27}, 
\]
and this leads us to the point
\[
P := \left( \frac{18D(t+1)(t+9)}{t^2 - 18t - 27}, \frac{24Dt(t+9)}{t^2 - 18t - 27} \sqrt{\frac{6D(t+1)(t+9)}{(t^2-18t-27)}} \right) \in \mc{E}_3[3].
\]
Since $\displaystyle \mbq(t,D)(\mc{E}_3[3])^{\ker \psi_{3,1}} = \mbq(t,D)\left( P \right)$, this establishes the first formula in \eqref{subfieldsoflevel3lemmaeqn}; the expression for the fixed field of $\ker \psi_{3,2}$ then follows from the fact that $\displaystyle \psi_{3,1}(g) \psi_{3,2}(g) = \det g$, whose corresponding fixed field is $\mbq(t,D) \left( \sqrt{-3} \right)$.
\end{proof}

We will also make use of the following classical fact about the third division field of an elliptic curve.  Note that the commutator subgroup $\SL_2(\mbz/3\mbz)' := \left[ \SL_2(\mbz/3\mbz), \SL_2(\mbz/3\mbz) \right]$ is a normal subgroup of $\GL_2(\mbz/3\mbz)$ and the quotient group is dihedral of order 6:
\[
\frac{\GL_2(\mbz/3\mbz)}{\SL_2(\mbz/3\mbz)'} \simeq D_3.
\] 
Thus, the associated fixed field $\mbq(E[3])^{\rho_{E,3}(G_\mbq) \cap \SL_2(\mbz/3\mbz)'} \subseteq \mbq(E[3])$ is generically a $D_3$-extension of $\mbq$; the next lemma specifies generators for this subfield.
\begin{lemma} \label{level3kummersubextensionlemma}
Let $E$ be an elliptic curve over $\mbq$ and let $\gD_E$ denote the discriminant of any Weierstrass model of $E$.  We have that $\mbq\left( \mu_3, \gD_E^{1/3} \right) \subseteq \mbq(E[3])$.  Furthermore, this subfield corresponds via Galois theory to the subgroup $\rho_{E,3}(G_\mbq) \cap \SL_2(\mbz/3\mbz)'$, i.e. we have
\[
\mbq(E[3])^{\rho_{E,3}(G_\mbq) \cap \SL_2(\mbz/3\mbz)' } = \mbq \left( \mu_3, \gD_E^{1/3} \right).
\]
\end{lemma}
\begin{proof}
This is a classical result; see for instance \cite[pp. 181--183]{langtrotter} and the references therein.
\end{proof}

\subsubsection{The level $m=4$}

The following lemma details the relevant classical facts surrounding the fourth division field of an elliptic curve.  Note that the commutator subgroup $\SL_2(\mbz/4\mbz)' := \left[ \SL_2(\mbz/4\mbz), \SL_2(\mbz/4\mbz) \right]$ is a normal subgroup of $\GL_2(\mbz/4\mbz)$ and the quotient group is dihedral of order 8:
\[
\frac{\GL_2(\mbz/4\mbz)}{\SL_2(\mbz/4\mbz)'} \simeq D_4.
\] 
Thus, the associated fixed field $\mbq(E[4])^{\rho_{E,4}(G_\mbq) \cap \SL_2(\mbz/4\mbz)'} \subseteq \mbq(E[4])$ is generically a $D_4$-extension of $\mbq$; we now specify generators for this subfield.
\begin{lemma} \label{level4kummersubextensionlemma}
Let $E$ be an elliptic curve over $\mbq$ and let $\gD_E$ denote the discriminant of any Weierstrass model of $E$.  We have that $\mbq\left( \mu_4, \gD_E^{1/4} \right) \subseteq \mbq(E[4])$.  Furthermore, this subfield corresponds via Galois theory to the subgroup $\rho_{E,4}(G_\mbq) \cap \SL_2(\mbz/4\mbz)'$, i.e. we have
\[
\mbq(E[4])^{\rho_{E,4}(G_\mbq) \cap \SL_2(\mbz/4\mbz)' } = \mbq \left( \mu_4, \gD_E^{1/4} \right).
\]
\end{lemma}
\begin{proof}
See \cite[pp. 172--173]{langtrotter} and \cite[pp. 218--220]{langtrotter}.
\end{proof}

We will sometimes need to deal with this subfield in the case that $\rho_{E,4}(G_\mbq)$ is contained in a specific proper subgroup of $\GL_2(\mbz/4\mbz)$.  In particular, we will be interested in the subgroup
\begin{equation} \label{defofGL2subchi4equalsve}
\begin{split}
\GL_2(\mbz/4\mbz)_{\chi_4 = \ve} :=& \left\{ g \in \GL_2(\mbz/4\mbz) : \chi_4(\det g) = \ve(g \mod 2) \right\} \\
=& \left\langle \begin{pmatrix} 1 & 1 \\ 0 & 3 \end{pmatrix}, \begin{pmatrix} 1 & 0 \\ 1 & 3 \end{pmatrix}, \begin{pmatrix} 1 & 3 \\ 1 & 0 \end{pmatrix} \right\rangle,
\end{split}
\end{equation}
where $\chi_4 : (\mbz/4\mbz)^\times \rightarrow \{ \pm 1 \}$ is the unique nontrivial multiplicative character and $\ve : \GL_2(\mbz/2\mbz) \rightarrow \{ \pm 1 \}$ is as in \eqref{defofve}.  For any elliptic curve $E$ over $\mbq$, we have
\[
\rho_{E,4}(G_\mbq) \subseteq \GL_2(\mbz/4\mbz)_{\ve = \chi_4} \; \Longleftrightarrow \mbq\left( \sqrt{\gD_E} \right) = \mbq(i).
\]
There is a rational parameter $t$ on the genus zero modular curve $X_{\GL_2(\mbz/4\mbz)_{\chi_4 = \ve}}$ with the property that the forgetful map $X_{\GL_2(\mbz/4\mbz)_{\chi_4 = \ve}} \rightarrow X(1)$ takes the form $t \mapsto j_4(t)$, where
\[
j_4(t) := -t^2 + 1728.
\]
As detailed in \cite{sutherlandzywina}, for any elliptic curve $E$ over $\mbq$ with $j$-invariant $j_E$, we have
\begin{equation} \label{level4containmentintermsofjinvariant}
\rho_{E,4}(G_\mbq) \subseteq \GL_2(\mbz/4\mbz)_{\chi_4 = \ve} \; \Longleftrightarrow \; \exists t_0 \in \mbq \; \text{ with } \; j_E = j_4(t_0).
\end{equation}
In particular, defining the elliptic curve $\mc{E}_4$ over $\mbq(t,D)$ by
\begin{equation} \label{level4genericellipticcurve}
\mc{E}_4 : \; y^2 = x^3 + \frac{108D^2 j_{4}(t)}{1728 - j_{4}(t)} x + \frac{432D^3 j_{4}(t)}{1728 - j_{4}(t)},
\end{equation}
we have that $\rho_{\mc{E}_4,4}(G_{\mbq(t,D)}) \doteq \GL_2(\mbz/4\mbz)_{\chi_4 = \ve}$.  The following lemma summarizes the situation and will be useful in what follows.
\begin{lemma} \label{identifyingthesubfieldslevel4lemma}
For any elliptic curve $E$ over $\mbq$, we have
\begin{equation} \label{arisesasaspecializationlevel4}
\rho_{E,4}(G_\mbq) \, \dot\subseteq \, \GL_2(\mbz/4\mbz)_{\chi_4 = \ve} \; \Longleftrightarrow \; \exists t_0, D_0 \in \mbq \; \text{ with } \; E \simeq_\mbq \mc{E}_4(t_0,D_0),
\end{equation}
where $\mc{E}_4$ is the elliptic curve over $\mbq(t,D)$ defined by \eqref{level4genericellipticcurve}.  Furthermore, when this is the case, we have
\begin{equation} \label{explicitbiquadraticfourthrootofdiscriminant}
\mbq(\mu_4, \gD_E^{1/4}) = \mbq\left( \mu_4,\gD_{\mc{E}_4(t_0,D_0)}^{1/4} \right) = \mbq\left( i, \sqrt{D_0t_0(t_0^2-1728)} \right).
\end{equation}
In particular, when \eqref{arisesasaspecializationlevel4} holds, the subfield $\mbq\left( \mu_4,\gD_E^{1/4} \right) \subseteq \mbq(E[4])$ is either biquadratic or quadratic over $\mbq$.
\end{lemma}
\begin{proof}
The assertion \eqref{arisesasaspecializationlevel4} follows immediately from \eqref{level4containmentintermsofjinvariant}.  The equality \eqref{explicitbiquadraticfourthrootofdiscriminant} follows from
\[
\gD_{\mc{E}_4} = - \left( \frac{2^9 3^6 D^3 (t^2 - 1728) }{t^3} \right)^2,
\]
using the fact that $(-1)^{1/4} = \zeta_8 = \frac{\sqrt{2}}{2} + \frac{\sqrt{2}}{2}i$, and from \eqref{explicitbiquadraticfourthrootofdiscriminant} one sees that $\mbq\left( \mu_4, \gD_E^{1/4} \right)$ is either biquadratic or quadratic over $\mbq$.
\end{proof}

\subsubsection{The level $m = 5$}

The subgroups
\begin{equation} \label{defofG51andG52}
G_{5,1} := \left\{ \begin{pmatrix} \pm 1 & * \\ 0 & * \end{pmatrix} \right\}, \; G_{5,2} := \left\{ \begin{pmatrix} * & * \\ 0 & \pm 1 \end{pmatrix} \right\} \subseteq \left\{ \begin{pmatrix} * & * \\ 0 & * \end{pmatrix} \right\} \subseteq \GL_2(\mbz/5\mbz)
\end{equation}
correspond to two genus zero modular curves $X_{G_{5,1}}$ and $X_{G_{5,2}}$, each of which is a $2$-fold cover of $X_0(5)$.
As discussed in \cite{zywina}, the maps $\mbp^1(t) \rightarrow \mbp^1(j)$ corresponding to the forgetful maps $X_{G_{5,i}} \longrightarrow X(1)$ are given respectively by the rational functions
\[
j_{5,1}(t) := \frac{(t^4 - 12t^3 + 14t^2 + 12t + 1)^3}{t^5(t^2 - 11t - 1)}, \quad\quad
j_{5,2}(t) :=  \frac{(t^4 + 228t^3 + 494t^2 - 228t + 1)^3}{t(t^2 - 11t - 1)^5}.
\]
We define the elliptic curves $\mc{E}_{5,i}$ over $\mbq(t,D)$ by
\begin{equation} \label{level5genericellipticcurve}
\mc{E}_{5,i} : \; y^2 = x^3 + \frac{108D^2 j_{5,i}(t)}{1728 - j_{5,i}(t)} x + \frac{432D^3 j_{5,i}(t)}{1728 - j_{5,i}(t)} \quad\quad\quad \left( i \in \{ 1, 2 \} \right);
\end{equation}
we have
\begin{equation} \label{genericGaloisimageiscontainedinborelmod5}
\rho_{\mc{E}_{5,i},5}(G_{\mbq(t,D)}) \, \doteq \, G_{5,i} \subseteq \GL_2(\mbz/5\mbz) \quad\quad\quad \left( i \in \{ 1, 2 \} \right).
\end{equation}
The next lemma explicitly characterizes the subfields cut out by the characters
\begin{equation*} 
\begin{split}
\psi_{5,1}, \; \psi_{5,2}^2\psi_{5,1} : \left\{ \begin{pmatrix} \pm 1 & * \\ 0 & * \end{pmatrix} \right\} &\longrightarrow \{ \pm 1 \} \subseteq \left( \mbz/5\mbz \right)^\times, \\
\psi_{5,2}, \; \psi_{5,1}^2\psi_{5,2} : \left\{ \begin{pmatrix} * & * \\ 0 & \pm 1 \end{pmatrix} \right\} &\longrightarrow \{ \pm 1 \} \subseteq \left( \mbz/5\mbz \right)^\times.
\end{split}
\end{equation*}
\begin{lemma} \label{subfieldsoflevel5lemma}
For each $i \in \{1, 2\}$, let $\mc{E}_{5,i}$ be the elliptic curve over $\mbq(t,D)$ defined by \eqref{level5genericellipticcurve} and let  
\[
\chi_{5,1}^{(i)}, \; \chi_{5,2}^{(i)} : \gal\left( \mbq(t,D)(\mc{E}_{5,i}[5])/\mbq(t,D) \right) \longrightarrow \{ \pm 1 \} \subseteq (\mbz/5\mbz)^\times
\]
denote the restrictions under \eqref{genericGaloisimageiscontainedinborelmod5} of the characters $\chi_{5,1}^{(i)} := \psi_{5,i}$ and $\chi_{5,2}^{(i)} := \psi_{5,i}\psi_{5,3-i}^2$, where $\psi_{5,i}$ are as in \eqref{generaldefinitionofpsisubp}.  We then have
\begin{equation} \label{subfieldsoflevel5lemmaeqn}
\begin{split}
\mbq(t,D)(\mc{E}_{5,1}[5])^{\ker \chi_{5,1}^{(1)}} &= \mbq(t,D)\left( \sqrt{-\frac{2D(t^4 - 12t^3 + 14t^2 + 12t + 1)}{((t^2 + 1)(t^4 - 18t^3 + 74t^2 + 18t + 1)}} \right), \\
\mbq(t,D)(\mc{E}_{5,1}[5])^{\ker \chi_{5,2}^{(1)}} &= \mbq(t,D)\left( \sqrt{-\frac{10D(t^4 - 12t^3 + 14t^2 + 12t + 1)}{((t^2 + 1)(t^4 - 18t^3 + 74t^2 + 18t + 1)}} \right), \\
\mbq(t,D)(\mc{E}_{5,2}[5])^{\ker \chi_{5,1}^{(2)}} &= \mbq(t,D)\left( \sqrt{-\frac{2D(t^4 + 228t^3 + 494t^2 - 228t + 1)}{(t^2+1)(t^4 - 522t - 10006t^2 + 522t + 1)}} \right), \\
\mbq(t,D)(\mc{E}_{5,2}[5])^{\ker \chi_{5,2}^{(2)}} &= \mbq(t,D)\left( \sqrt{-\frac{10D(t^4 + 228t^3 + 494t^2 - 228t + 1)}{(t^2+1)(t^4 - 522t - 10006t^2 + 522t + 1)}} \right).
\end{split}
\end{equation}
\end{lemma}
\begin{proof}
For any elliptic curve $E$ over $\mbq$ with $\rho_{E,5}(G_\mbq) \subseteq B(5)$, define $\langle P \rangle \subseteq E[5]$ and $\langle P' \rangle \subseteq E_{\langle P \rangle}'[5]$ as in \eqref{interpretationofwhenGaloismapsintoBorel} and \eqref{existenceofPprime}.  By \eqref{psisubellsactingonPandPprime} and \eqref{defofG51andG52}, we have
\begin{equation} \label{conditionforG51andG52}
\begin{split}
\rho_{E,5}(G_\mbq) \, \dot\subseteq \, G_{5,1} \; &\Longleftrightarrow \; \exists \text{ a $G_\mbq$-stable } \langle P \rangle \subseteq E[5] \quad\quad\;\;\; \text{ with } \quad [ \mbq( P ) : \mbq ] \leq 2, \\
\rho_{E,5}(G_\mbq) \, \dot\subseteq \, G_{5,2} \; &\Longleftrightarrow \; \begin{matrix} \exists \text{ a $G_\mbq$-stable } \langle P \rangle \subseteq E[5] \text{ and} \\ \exists \text{ a $G_\mbq$-stable  } \langle P' \rangle \subseteq E_{\langle P \rangle}' [5] \end{matrix} \quad \text{ with } \quad [ \mbq( P' ) : \mbq ] \leq 2,
\end{split}
\end{equation}
and the same statement holds when the base field $\mbq$ is replaced by $\mbq(t,D)$.  Furthermore, we have
\begin{equation} \label{kernelsintermsofPlevel5}
\mbq(t,D)\left( \mc{E}_{5,1}[5] \right)^{\ker \chi_{5,1}^{(1)}} = \mbq(t,D)\left( P \right), \quad\quad \mbq(t,D)\left( \mc{E}_{5,2}[5] \right)^{\ker \chi_{5,1}^{(2)}} = \mbq(t,D)\left( P' \right), 
\end{equation}
where $P \in \mc{E}_{5,1}[5]$ and $P' \in (\mc{E}_{5,2})_{\langle P \rangle}'[5]$ are as in \eqref{conditionforG51andG52}.

Using the linear factor of the $5$th division polynomial of $\mc{E}_{5,1}$, we find the point $P_1 = (x_1,y_1) \in \mc{E}_{5,1}[5]$, where
\[
\begin{split}
x_1 &:= -\frac{6D(t^2 - 6t + 1)(t^4 - 12t^3 + 14t^2 + 12t + 1)}{(t^2 + 1)(t^4 - 18t^3 + 74t^2 + 18t + 1)}, \\
y_1 &:=  \frac{216dt(t^4 - 12t^3 + 14t^2 + 12t + 1)}{(t^2 + 1)(t^4 - 18t^3 + 74t^2 + 18t + 1)} \sqrt{\frac{-2D(t^4 - 12t^3 + 14t^2 + 12t + 1)}{(t^2 + 1)(t^4 - 18t^3 + 74t^2 + 18t + 1)}}.
\end{split}
\]
By \eqref{kernelsintermsofPlevel5}, this proves the first equality in \eqref{subfieldsoflevel5lemmaeqn}. We now consider the character $\displaystyle \psi_{5,2}^2 : G_{5,1} \rightarrow \{ \pm 1 \}$, whose value $\psi_{5,2}(g_1)^2$ agrees with $\displaystyle \left( \frac{5}{\det g_1} \right)$, and thus has corresponding fixed field $\mbq(t,D)\left( \sqrt{5} \right)$.  Since $\chi_{5,2}^{(1)} = \psi_{5,2}^2 \psi_{5,1}$, this observation establishes the second equality in \eqref{subfieldsoflevel5lemmaeqn}.  

For the second pair of equalities in \eqref{subfieldsoflevel5lemmaeqn}, we reason as follows.  The $5$th division polynomial associated to $\mc{E}_{5,2}$ has a quadratic factor that is irreducible over $\mbq(t,D)$, and this leads to a $G_{\mbq(t,D)}$-stable cyclic subgroup $\langle P \rangle \subseteq \mc{E}_{5,2}[5]$.  We find that the isogenous elliptic curve $\left( \mc{E}_{5,2} \right)_{\langle P \rangle}' = \mc{E}_{5,2} / \langle P \rangle$ is given by
\[
\begin{split}
\left( \mc{E}_{5,2} \right)_{\langle P \rangle}' : \; y^2 = x^3 &- \frac{67500D^2(t^4 - 12t^3 + 14t^2 + 12t + 1)(t^4 + 228t^3 + 494t^2 - 228t + 1)^2}{(t^2+1)^2(t^4 - 522t^3 - 10006t^2 + 522t + 1)^2}x \\ 
&- \frac{6750000D^3(t^4 - 18t^3 + 74t^2 + 18t + 1)(t^4 + 228t^3 + 494t^2 - 228t + 1)^3}{(t^2+1)^2(t^4 - 522t^3 - 10006t^2 + 522t + 1)^3}.
\end{split}
\]
The $5$th division polynomial of $\left( \mc{E}_{5,2} \right)_{\langle P \rangle}'$ is seen to have a linear factor, which leads to the point $P_2' = (x_2', y_2') \in \left( \mc{E}_{5,2} \right)_{\langle P \rangle}'[5]$, where
\[
\begin{split}
x_2' &:= -150 \frac{D (t^2 - 6t + 1) (t^4 + 228t^3 + 494t^2 - 228t + 1)}{(t^2 + 1) (t^4 - 522t^3 - 10006t^2 + 522t + 1)}, \\
y_2' &:= 27000 \frac{D t (t^4 + 228t^3 + 494t^2 - 228t + 1)}{(t^2 + 1) (t^4 - 522t^3 - 10006t^2 + 522t + 1)} \sqrt{\frac{-2 D (t^4 + 228t^3 + 494t^2 - 228t + 1)}{(t^2 + 1) (t^4 - 522t^3 - 10006t^2 + 522t + 1)}}.
\end{split}
\]
As before, this, together with the fact that for $g_2 \in G_{5,2}$, the value $\psi_{5,1}(g_2)^2$ agrees with $\displaystyle \left( \frac{5}{\det g_2} \right)$ and that $\chi_{5,2}^{(2)} = \psi_{5,1}^2 \psi_{5,2}$, establishes the second two equalities in \eqref{subfieldsoflevel5lemmaeqn}, proving the lemma.
\end{proof}

\subsubsection{The level $m=7$}

We begin by describing an explicit Weierstrass model $\mc{E}_7$ over $\mbq(t,D)$ that is generic in the sense that its specializations $\mc{E}_7(t_0,D_0)$ give rise to all elliptic curves $E$ over $\mbq$ for which $\rho_{E,7}(G_\mbq) \, \dot\subseteq \, B(7)$, where we recall that
\[
B(7) = \left\{ \begin{pmatrix} * & * \\ 0 & * \end{pmatrix} \right\} \subseteq \GL_2(\mbz/7\mbz)
\]
denotes the Borel subgroup.  We then describe explicitly certain subfields of $\mbq(\mc{E}_7[7])$ that will be useful in the next section.  

Define $j_7(t) \in \mbq(t)$ by
\begin{equation*} 
j_7(t) := \frac{(t^2 + 245t + 2401)^3(t^2 + 13t + 49)}{t^7},
\end{equation*}
and the elliptic curve $\mc{E}_7$ over $\mbq(t,D)$ by
\begin{equation} \label{level7genericellipticcurve}
\begin{split}
&\mc{E}_7 : \; y^2 = x^3 + D^2 a_{4;7}(t) x + D^3 a_{6;7}(t), \\
& a_{4;7}(t) := \frac{108 j_{7}(t)}{1728 - j_{7}(t)}, \quad\quad a_{6;7}(t) := \frac{432 j_{7}(t)}{1728 - j_{7}(t)}.
\end{split}
\end{equation}
As proved in \cite{zywina}, we have
\begin{equation} \label{genericGaloisimageiscontainedinborelmod7}
\rho_{\mc{E}_{7},7}(G_{\mbq(t,D)}) \, \doteq \, B(7) \subseteq \GL_2(\mbz/7\mbz).
\end{equation}
The next lemma explicitly characterizes the subfields cut out by the quadratic characters $\psi_{7,1}^3$ and $\psi_{7,2}^3$.
\begin{lemma} \label{quadraticsubfieldsoflevel7lemma}
Let $\mc{E}_{7}$ be the elliptic curve over $\mbq(t,D)$ defined by \eqref{level7genericellipticcurve} and let  
\[
\psi_{7,1}^3, \; \psi_{7,2}^3 : \gal\left( \mbq(t,D)(\mc{E}_{7}[7])/\mbq(t,D) \right) \rightarrow \{ \pm 1 \}
\]
denote the restrictions under \eqref{genericGaloisimageiscontainedinborelmod7} of the cubes of the characters defined in \eqref{generaldefinitionofpsisubp}.  We then have
\begin{equation} \label{quadsubfieldsoflevel7lemmaeqn}
\begin{split}
\mbq(t,D)(\mc{E}_{7}[7])^{\ker \psi_{7,1}^3} &= \mbq(t,D)\left( \sqrt{\frac{14D(t^2 + 13t + 49)(t^2 + 245t + 2401)}{(t^4 - 490t^3 - 21609t^2 - 235298t - 823543)}} \right), \\
\mbq(t,D)(\mc{E}_{7}[7])^{\ker \psi_{7,2}^3} &= \mbq(t,D)\left( \sqrt{- \frac{2D(t^2 + 13t + 49)(t^2 + 245t + 2401)}{(t^4 - 490t^3 - 21609t^2 - 235298t - 823543)}} \right).
\end{split}
\end{equation}
\end{lemma}
\begin{proof}
A computation reveals that the $7$th division polynomial of $\mc{E}_7$ has the cubic factor
\begin{equation} \label{cubicfactorofpsisub7}
\begin{split}
x^3 &- \frac{126D(t^2 + 13t + 49)(t^2 + 245t + 2401)}{(t^4 - 490t^3 - 21609t^2 - 235298t - 823543)}x^2 \\ &+ \frac{108D^2(t^2 + 13t + 49)(t^2 + 245t + 2401)^2(33t^2 + 637t + 2401)}{(t^4 - 490t^3 - 21609t^2 - 235298t - 823543)^2}x \\
&- \frac{216D^3(t^2 + 13t + 49)(t^2 + 245t + 2401)^3(881t^4 + 38122t^3 + 525819t^2 + 3058874t + 5764801)}{7(t^4 - 490t^3 - 21609t^2 - 235298t - 823543)^3},
\end{split}
\end{equation}
which is irreducible over $\mbq(t,D)$ and whose discriminant is in $\left( \mbq(t,D)^\times \right)^2$.  Writing this polynomial in the form $f(x) = (x - \ga_1)(x - \ga_2)(x - \ga_3)$, we have that $\mbq(t,D)\left( \ga_1 \right)$ is cyclic cubic over $\mbq(t,D)$ and that the point 
\begin{equation} \label{defofPforlevel7}
P := \left( \ga_1, \sqrt{ \ga_1^3 + D^2 a_{4;7}(t) \ga_1 + D^3 a_{6;7}(t) } \right) 
\end{equation}
generates a cyclic submodule $\langle P \rangle \subseteq \mc{E}_7[7]$ on which $G_{\mbq(t,D)}$ acts through the eigenfunction $\psi_{7,1}$ via 
\begin{equation} \label{Pfixedbykernelofpsisub7supone}
\gs : P \mapsto \left[ \psi_{7,1}(\gs) \right] P.
\end{equation}
We have
\begin{equation} \label{whatmbqtdPlookslike}
\begin{split}
\mbq(t,D)\left( P \right) &= \mbq(t,D)\left( \ga_1, \sqrt{ \ga_1^3 + D^2 a_{4;7}(t) \ga_1 + D^3 a_{6;7}(t) } \right) \\
&\supseteq \mbq(t,D)\left( \ga_1, \sqrt{ \prod_{j = 1}^3 (\ga_j^3 + D^2 a_{4;7}(t) \ga_j + D^3 a_{6;7}(t)) } \right);
\end{split}
\end{equation}
a computation using the third equation in \eqref{symmetricpolynomialequalities} shows that 
\[
\mbq(t,D)\left(  \sqrt{ \prod_{j=1}^3 (\ga_j^3 + D^2 a_{4;7}(t) \ga_j + D^3 a_{6;7}(t) ) } \right) =  \mbq(t,D)\left( \sqrt{\frac{14D(t^2 + 13t + 49)(t^2 + 245t + 2401)}{(t^4 - 490t^3 - 21609t^2 - 235298t - 823543)}} \right),
\]
establishing the first equality in \eqref{quadsubfieldsoflevel7lemmaeqn} (and also showing that we have equality in \eqref{whatmbqtdPlookslike}).  The second equality follows from the fact that the product $\psi_{7,1}^3(g) \psi_{7,2}^3(g)$ agrees with $\displaystyle \left( \frac{-7}{\det g} \right)$, whose fixed field is $\mbq(t,D)\left( \sqrt{-7} \right)$.
\end{proof}

Our next lemma explicitly characterizes the subfields cut out by the cubic characters $\psi_{7,1}^2, \psi_{7,2}^2$ and $\psi_{7,1}^2\psi_{7,2}^4$, where $\psi_{7,i}$ is defined as in \eqref{generaldefinitionofpsisubp}.

We define the polynomials
\begin{equation} \label{defoffsubTRandRprime}
\begin{split}
f_{\cyc,7}^+(X) &:= X^3 + X^2 - 2X - 1, \\
f_T(X) &:= X^3 - T_1(t) X^2 + T_2(t) X - T_3(t), \\
f_R(X) &:= X^3 - R_1(t) X^2 + R_2(t) X - R_3(t), \\
f_{R'}(X) &:= X^3 - R_1(t) X^2 + R_2(t) X - R_3'(t),
\end{split}
\end{equation}
where
\begin{equation} \label{defoffsubXcoefficients}
\begin{split}
T_1(t) &:= -21 ( t^2 + 13t + 49 ), \\
T_2(t) &:= 3 (t^2 + 13t + 49) (33 t^2 + 637t + 2401), \\
T_3(t) &:= -\frac{1}{7} (t^2 + 13t + 49) ( 881 t^4 + 38122 t^3 + 525819 t^2 + 3058874 t + 5764801 ), \\
R_1(t) &:= 21 (t^2 + 13t + 49), \\
R_2(t) &:= -21(t^2 + 13t + 49) (9t^2 - 91 t - 343), \\
R_3(t) &:= -7 (t^2 + 13t + 49) ( 223t^4 + 3542t^3 + 3381t^2 - 62426 t - 117649 ), \\
R_3'(t) &:= - (t^2 + 13t + 49) (3289 t^4 + 24794 t^3 + 23667 t^2 - 436982 t - 823543 ).
\end{split}
\end{equation}
\begin{lemma}
Let $\mc{E}_7$ be be the elliptic curve over $\mbq(t,D)$ defined by \eqref{level7genericellipticcurve} and let  
\[
\psi_{7,1}^2, \psi_{7,2}^2 : \gal\left( \mbq(t,D)(\mc{E}_{7}[7])/\mbq(t,D) \right) \longrightarrow \left( (\mbz/7\mbz)^\times \right)^2 \simeq \mu_3
\]
denote the restrictions under \eqref{genericGaloisimageiscontainedinborelmod7} of the squares of the characters $\psi_{7,i}$ defined in \eqref{generaldefinitionofpsisubp}.  Let us denote by $\mbq(t,D)_{f_{\cyc,7}^+}$, $\mbq(t,D)_{f_T}$, $\mbq(t,D)_{f_R}$ and $\mbq(t,D)_{f_{R'}}$ the splitting fields over $\mbq(t,D)$ of the polynomials $f_{\cyc,7}^+(X)$, $f_T(X)$, $f_R(X)$ and $f_{R'}(X)$, respectively, where these polynomials are defined by \eqref{defoffsubTRandRprime} and \eqref{defoffsubXcoefficients}.
We then have
\begin{equation} \label{cubicsubfieldsoflevel7lemmaeqn}
\begin{split}
&\mbq(t,D)\left( \mc{E}_{7}[7] \right)^{\ker \psi_{7,1}^2} = \mbq(t,D)_{f_T}, \quad
\mbq(t,D)\left( \mc{E}_{7}[7] \right)^{\ker \psi_{7,1}^2 \psi_{7,2}^2} = \mbq(t,D)_{f_{\cyc,7}^+}, \\
&\mbq(t,D)\left( \mc{E}_{7}[7] \right)^{\ker \psi_{7,2}^2} = \mbq(t,D)_{f_R}, \quad \mbq(t,D)(\mc{E}_{7}[7])^{\psi_{7,1}^2\psi_{7,2}^4} = \mbq(t,D)_{f_{R'}}.
\end{split}
\end{equation}
\end{lemma}
\begin{proof}
Since $\rho_{\mc{E}_7,7}\left(G_{\mbq(t,D)}\right) = B(7)$, it is straightforward to see that $\mbq(t,D)\left( \mc{E}_7[7] \right)$ has exactly $4$ cyclic cubic subfields.  Furthermore, we have
\[
\mbq(t,D)\left(\mc{E}_7[7] \right)^{\ker \psi_{7,1}^2 \psi_{7,2}^2} = \mbq(t,D)\left( \mu_7 \right)^+ \subseteq \mbq(t,D)\left( \mu_7 \right),
\]
the first equality above following from the fact that $\psi_{7,1}(g)\psi_{7,2}(g) = \det g$, which implies that the fixed field of $\ker \psi_{7,1}^2 \psi_{7,2}^2$ is the maximal real subfield $\mbq(t,D)\left( \mu_7 \right)^+$, i.e. the unique subfield of $\mbq(t,D)\left( \mu_7 \right)$ that is cyclic cubic over $\mbq(t,D)$.  We have $\mbq(t,D)\left( \mu_7 \right)^+ = \mbq(t,D)\left( \zeta_7 + \zeta_7^{-1} \right)$, so $\mbq(t,D)\left( \mu_7 \right)^+$ is the splitting field of
\[
f_{\cyc,7}^+(X) = X^3 + X^2 - 2X - 1,
\]
the minimal polynomial over $\mbq(t,D)$ of the generator $\zeta_7 + \zeta_7^{-1}$.  This establishes the second equality in \eqref{cubicsubfieldsoflevel7lemmaeqn}.

For the equality in \eqref{cubicsubfieldsoflevel7lemmaeqn} involving the fixed field of $\ker \psi_{7,1}^2$, we reason as follows.  By \eqref{Pfixedbykernelofpsisub7supone} and \eqref{defofPforlevel7}, we see that $\mbq(t,D)\left( \mc{E}_7[7] \right)^{\ker \psi_{7,1}} = \mbq(t,D)(P)$, where $\langle P \rangle \subseteq \mc{E}_7[7]$ is a cyclic $G_{\mbq(t,D)}$-stable subgroup, and since this extension is cyclic of degree $6$ over $\mbq(t,D)$, it follows that
\begin{equation} \label{generatedbyxcoordinateofP}
\mbq(t,D)\left( \mc{E}_7[7] \right)^{\ker \psi_{7,1}^2} = \mbq(t,D)\left( \ga_1 \right).
\end{equation}
where $\ga_1$ is the $x$-coordinate of $P$.
Finally, the substitution $x = -\frac{6D(t^2 + 245t + 2401)}{(t^4 - 490t^3 - 21609t^2 - 235298t - 823543)}X$ transforms the cyclic cubic polynomial \eqref{cubicfactorofpsisub7} into $\left( -\frac{6D(t^2 + 245t + 2401)}{(t^4 - 490t^3 - 21609t^2 - 235298t - 823543)} \right)^3f_T(X)$, and the first equality in \eqref{cubicsubfieldsoflevel7lemmaeqn} follows.

The discriminants $\gD_T$ and $\gD_{\cyc,7}^+$ associated to the polynomials $f_T(X)$ and $f_{\cyc,7}^+(X)$ satisfy
\[
\sqrt{\gD_T} = \frac{2^63^3t^4(t^2 + 13t + 49)}{7}, \quad\quad \sqrt{\gD_{\cyc,7}^+} = 7.
\]
Applying Lemma \ref{gettingattheothercubicfieldslemma}, we find that $\mbq(t,D)_{f_R}$ and $\mbq(t,D)_{f_{R'}}$ are the remaining two cyclic cubic subfields of $\mbq(t,D)\left( \mc{E}_7[7] \right)$.  To see which subfield is which, we first note that, just as in \eqref{generatedbyxcoordinateofP}, we have
\[
\mbq(t,D)\left( \mc{E}_7[7] \right)^{\ker \psi_{7,2}^2} = \mbq(t,D)\left( x(P') \right),
\]
where $P' \in (\mc{E}_7)_{\langle P \rangle}'$ is any generator of a cyclic $G_{\mbq(t,D)}$-stable subgroup $\langle P' \rangle \subseteq (\mc{E}_7)_{\langle P \rangle}'[7]$.  A direct computation reveals that the isogenous curve $\left(\mc{E}_7\right)_{\langle P \rangle}'$ has Weierstrass equation
\[
\begin{split}
\left(\mc{E}_7\right)_{\langle P \rangle}' : \; y^2 = x^3 - &\frac{2^23^37^4D^2(t^2 + 5t + 1)(t^2 + 13t + 49)(t^2 + 245t + 2401)^2}{(t^4 - 490t^3 - 21609t^2 - 235298t - 823543)^3}x \\
- &\frac{2^43^37^6D^3(t^2 + 13t + 49)(t^2 + 245t + 2401)^3(t^4 + 14t^3 + 63t^2 + 70t - 7)}{(t^4 - 490t^3 - 21609t^2 - 235298t - 823543)^3},
\end{split}
\]
and that its $7$th division polynomial has the cubic factor
\[
\begin{split}
x^3 + &\frac{2\cdot 3^27^2D(t^2 + 13t + 49)(t^2 + 245t + 2401)}{(t^4 - 490t^3 - 21609t^2 - 235298t - 823543)}x^2 \\
+ &\frac{2^23^37^4D^2(t^2 + 13t + 33)(t^2 + 13t + 49)(t^2 + 245t + 2401)^2}{(t^4 - 490t^3 - 21609t^2 - 235298t - 823543)^2}x \\
+ &\frac{2^33^37^6D^3(t^2 + 13t + 49)(t^2 + 245t + 2401)^3(t^4 + 26t^3 + 219t^2 + 778t + 881)}{(t^4 - 490t^3 - 21609t^2 - 235298t - 823543)^3}.
\end{split}
\]
Finally, applying Lemma \ref{settingcubicfieldsequaltoeachotherlemma} (and some extensive, tedious calculations), we see that the splitting field of this polynomial agrees with the splitting field of $f_R(X)$, and this finishes the proof.
\end{proof}

\section{Developing explicit models for missing trace groups} \label{proofofmaintheoremsection}

In this section, we complete the proof of Theorem \ref{maintheorem}.  Specifically, for each $m$ appearing in the union on the right-hand side of \eqref{genuszerocurvesbylevel}, we will now
\begin{enumerate}
\item list the groups $G \in \mf{G}_{MT}^{\max}(0,m)$, up to conjugation in $\GL_2(\hat{\mbz})$;
\item for each such group $G$, exhibit a rational function $j_{\tilde{G}}(t) \in \mbq(t)$ which defines the forgetful map $j_{\tilde{G}} : X_{\tilde{G}} \longrightarrow X(1)$;
\item in case $G \subsetneq \tilde{G}$, identify each twist parameter $d_G(t) \in \mbq(t)$ for which the elliptic curve $\mc{E}_G$ over $\mbq(t)$ given by
\[
\mc{E}_G : \; d_G(t) y^2 = x^3 + a_{4,\tilde{G}}(t) x + a_{6,\tilde{G}}(t)
\]
satisfies $\rho_{\mc{E}_G,m}(G_{\mbq(t)}) = G$.  (Here the Weierstrass coefficients $a_{4,\tilde{G}}(t)$, $a_{6,\tilde{G}}(t) \in \mbq(t)$ are chosen as usual according to \eqref{defofa4anda6}, so that the $j$-invariant of $\mc{E}_G$ is $j_{\tilde{G}}(t)$.)
\end{enumerate}
Throughout this section, we denote by $\pi_{\GL_2}$ the canonical projection map
\[
\pi_{\GL_2} : \GL_2(\hat{\mbz}) \longrightarrow \GL_2(\mbz/m\mbz),
\]
suppressing the dependence of $\pi_{\GL_2}$ on the level $m$.

\medskip

\subsection{The level $m=2$.}

We have $\mf{G}_{MT}^{\max}(0,2) = \{ G_{2,1} \}$, where $G_{2,1}(2) \subseteq \GL_2(\mbz/2\mbz)$ is given by
\[
G_{2,1}(2) = \left\langle \begin{pmatrix} 1 & 1 \\ 0 & 1 \end{pmatrix} \right\rangle
\]
and $G_{2,1} = \pi_{\GL_2}^{-1}(G_{2,1}(2))$.  Note that $-I \in G_{2,1}$, so $G_{2,1} = \tilde{G}_{2,1}$.  Define the function $j_{2,1}(t) \in \mbq(t)$ by  
\[
j_{2,1}(t) := 256 \frac{(t+1)^3}{t}.
\]
As detailed in \cite{zywina}, for any elliptic curve $E$ over $\mbq$ with $j$-invariant $j_E$, one has
\begin{equation} \label{level2jinvariantstatement}
\rho_{E}(G_\mbq) \, \dot\subseteq \, G_{2,1} \; \Longleftrightarrow \; \exists t_0 \in \mbq \text{ for which } j_E = j_{2,1}(t_0).
\end{equation}
We define the coefficients $a_{4;2,1}(t)$ and $a_{6;2,1}(t)$ by \eqref{defofa4anda6} and the elliptic curve $\mc{E}_{2,1,1}$ over $\mbq(t,D)$ by
\[
\mc{E}_{2,1,1} : \; D y^2 = x^3 + a_{4;2,1}(t) x + a_{6;2,1}(t).
\]
It follows from \eqref{level2jinvariantstatement} that
\[
\rho_{E}(G_\mbq) \, \dot\subseteq \, G_{2,1} \; \Longleftrightarrow \; \exists t_0, D_0 \in \mbq \text{ for which } E \text{ is isomorphic over $\mbq$ to } \mc{E}_{2,1,1}(t_0,D_0).
\]

\medskip

\subsection{The level $m=3$.}

We have $\mf{G}_{MT}^{\max}(0,3) = \{ G_{3,1,1}, G_{3,1,2} \}$, where $G_{3,1,1}(3), G_{3,1,2}(3) \subseteq \GL_2(\mbz/3\mbz)$ are given by
\[
\begin{split}
G_{3,1,1}(3) &= \left\langle \begin{pmatrix} 1 & 1 \\ 0 & 1 \end{pmatrix}, \begin{pmatrix} 1 & 0 \\ 0 & 2 \end{pmatrix} \right\rangle = \left\{ \begin{pmatrix} 1 & * \\ 0 & * \end{pmatrix} \right\}, \\
G_{3,1,2}(3) &= \left\langle \begin{pmatrix} 1 & 1 \\ 0 & 1 \end{pmatrix}, \begin{pmatrix} 2 & 0 \\ 0 & 1 \end{pmatrix} \right\rangle = \left\{ \begin{pmatrix} * & * \\ 0 & 1 \end{pmatrix} \right\}
\end{split}
\]
and $G_{3,1,k} = \pi_{\GL_2}^{-1}(G_{3,1,k}(3))$ for $k \in \{1,2 \}$.  Note that $-I \notin G_{3,1,k}$.  We have 
\[
\tilde{G}_{3,1,1}(3) = \tilde{G}_{3,1,2}(3) = \left\{ \begin{pmatrix} * & * \\ 0 & * \end{pmatrix} \right\};
\]
let us denote this group by $\tilde{G}_{3,1}(3)$, omitting the last subscript.  Define the function $j_{3,1}(t) \in \mbq(t)$ by
\[
j_{3,1}(t) := 27\frac{(t+1)(t+9)^3}{t^3}.
\]
As detailed in \cite{zywina}, for any elliptic curve $E$ over $\mbq$ with $j$-invariant $j_E$, one has
\begin{equation} \label{level3jinvariantstatement}
\rho_{E}(G_\mbq) \, \dot\subseteq \, \tilde{G}_{3,1} \; \Longleftrightarrow \; \exists t_0 \in \mbq \text{ for which } j_E = j_{3,1}(t_0).
\end{equation}
We define the coefficients $a_{4;3,1}(t)$ and $a_{6;3,1}(t)$ by \eqref{defofa4anda6}, the twist parameters $d_{3,1,1}(t), d_{3,1,2}(t) \in \mbq(t)$ by
\[
d_{3,1,1}(t) := \frac{(t+1)(t^2 - 18t - 27)}{6(t+9)}, \quad\quad d_{3,1,2}(t) := -3 d_{3,1,1}(t),
\]
and the elliptic curves $\mc{E}_{3,1,k}$ over $\mbq(t)$ by
\[
\mc{E}_{3,1,k} : \; d_{3,1,k}(t) y^2 = x^3 + a_{4;3,1}(t) x + a_{6;3,1}(t) \quad\quad \left( k \in \{1, 2 \} \right).
\]
As may be found in \cite{zywina}, for any elliptic curve $E$ over $\mbq$ with $j$-invariant $j_E$, one has
\[
\begin{split}
\rho_{E}(G_\mbq) \, \dot\subseteq \, G_{3,1,1} \; &\Longleftrightarrow \; \exists t_0 \in \mbq \text{ for which } E \text{ is isomorphic over $\mbq$ to } \mc{E}_{3,1,1}\left( t_0 \right), \\
\rho_{E}(G_\mbq) \, \dot\subseteq \, G_{3,1,2} \; &\Longleftrightarrow \; \exists t_0 \in \mbq \text{ for which } E \text{ is isomorphic over $\mbq$ to } \mc{E}_{3,1,2}\left( t_0 \right).
\end{split}
\]

\medskip

\subsection{The level $m=4$.}

We have $\mf{G}_{MT}^{\max}(0,4) = \{ G_{4,1,1} \}$, where $G_{4,1,1}(4) \subseteq \GL_2(\mbz/4\mbz)$ is given by
\[
G_{4,1,1}(4) = \left\langle \begin{pmatrix} 1 & 1 \\ 0 & 3 \end{pmatrix}, \begin{pmatrix} 3 & 2 \\ 0 & 3 \end{pmatrix}, \begin{pmatrix} 1 & 1 \\ 1 & 2 \end{pmatrix} \right\rangle
\]
and $G_{4,1,1} = \pi_{\GL_2}^{-1}(G_{4,1,1}(4))$.  Note that $-I \notin G_{4,1,1}$.  We have 
\[
\tilde{G}_{4,1,1}(4) 
= \GL_2(\mbz/4\mbz)_{\chi_4 = \ve},
\]
defined as in \eqref{defofGL2subchi4equalsve}.  Let us set $\tilde{G}_{4,1} := \tilde{G}_{4,1,1}$.  We define the $j$-invariant $j_{4,1}(t) := -t^2 + 1728$ and the elliptic curve $\mc{E}_{4,1}$ over $\mbq(t,D)$ by
\begin{equation*} 
\mc{E}_{4,1} : \; y^2 = x^3 + \frac{108D^2 j_{4.1}(t)}{1728 - j_{4,1}(t)} x + \frac{432D^3 j_{4,1}(t)}{1728 - j_{4,1}(t)};
\end{equation*}
see \eqref{level4genericellipticcurve}.  By Lemma \ref{identifyingthesubfieldslevel4lemma}, for any elliptic curve $E$ over $\mbq$, we have
\[
\rho_{E}(G_\mbq) \, \dot\subseteq \, \tilde{G}_{4,1} \; \Longleftrightarrow \; \exists t_0, D_0 \in \mbq \text{ with } E \simeq_\mbq \mc{E}_{4,1}(t_0,D_0)
\]
and
\begin{equation} \label{specialformofthedummerextension}
\mbq(t,D)\left( i, \gD_{\mc{E}_{4,1}}^{1/4} \right) = \mbq(t,D)\left( i, \sqrt{Dt(t^2-1728)} \right).
\end{equation}

Regarding the index two subgroup $G_{4,1,1}(4) \subseteq \tilde{G}_{4,1}(4)$, a computation reveals that
\[
G_{4,1,1}(4) \cap \SL_2(\mbz/4\mbz)' = G_{4,1,1}(4) \cap \SL_2(\mbz/4\mbz),
\]
and $G_{4,1,1}(4)$ is the unique maximal subgroup (relative to $\dot\subseteq$) of $\tilde{G}_{4,1}(4)$ with this property.  By Lemmas \ref{level4kummersubextensionlemma} and \ref{identifyingthesubfieldslevel4lemma}, together with the Galois correspondence and \eqref{specialformofthedummerextension}, it follows that
\begin{equation*} 
\begin{split}
\rho_{\mc{E}_{4,1}(t_0,D_0)}(G_\mbq) \, \dot\subseteq \, G_{4,1,1} \; &\Longleftrightarrow \; \mbq\left( i,\gD_{\mc{E}_{4,1}(t_0,D_0)}^{1/4} \right) = \mbq(i) \\
&\Longleftrightarrow \; D_0 = \pm t_0(t_0^2 - 1728).
\end{split}
\end{equation*}
Noting that $t \mapsto t(t^2-1728)$ is an odd function of $t$ and $j_{4,1}(t)$ is even, we are led to the single twist parameter
\[
d_{4,1,1}(t) :=  t(t^2-1728),
\]
and we define the elliptic curve $\mc{E}_{4,1,1}$ over $\mbq(t)$ by
\[
\mc{E}_{4,1,1} : \; d_{4,1,1}(t) y^2 = x^3 + a_{4;4,1}(t) x + a_{6;4,1}(t).
\]
For each elliptic curve $E$ over $\mbq$ with $j$-invariant $j_E$, we evidently have
\[
\rho_{E}(G_\mbq) \, \dot\subseteq \, G_{4,1,1} \; \Longleftrightarrow \; \exists t_0 \in \mbq \text{ for which $E$ is isomorphic over $\mbq$ to $\mc{E}_{4,1,1}\left( t_0 \right)$.}
\]

\medskip

\subsection{The level $m=5$.}

We have $\mf{G}_{MT}^{\max}(0,5) = \{ G_{5,1,1}, G_{5,1,2}, G_{5,2,1}, G_{5,2,2} \}$, where the groups $G_{5,i,k}(5) \subseteq \GL_2(\mbz/5\mbz)$ are given by
\[
\begin{split}
G_{5,1,1}(5) &= \left\langle \begin{pmatrix} 1 & 1 \\ 0 & 1 \end{pmatrix}, \begin{pmatrix} 1 & 0 \\ 0 & 2 \end{pmatrix} \right\rangle = \left\{ \begin{pmatrix} 1 & * \\ 0 & * \end{pmatrix} \right\}, \\
G_{5,1,2}(5) &= \left\langle \begin{pmatrix} 1 & 1 \\ 0 & 1 \end{pmatrix}, \begin{pmatrix} 4 & 0 \\ 0 & 2 \end{pmatrix} \right\rangle = \left\{ \begin{pmatrix} a^2 & * \\ 0 & a \end{pmatrix} : a \in (\mbz/5\mbz)^\times \right\}, \\
G_{5,2,1}(5) &= \left\langle \begin{pmatrix} 1 & 1 \\ 0 & 1 \end{pmatrix}, \begin{pmatrix} 2 & 0 \\ 0 & 1 \end{pmatrix} \right\rangle = \left\{ \begin{pmatrix} * & * \\ 0 & 1 \end{pmatrix} \right\}, \\
G_{5,2,2}(5) &= \left\langle \begin{pmatrix} 1 & 1 \\ 0 & 1 \end{pmatrix}, \begin{pmatrix} 2 & 0 \\ 0 & 4 \end{pmatrix} \right\rangle = \left\{ \begin{pmatrix} a & * \\ 0 & a^2 \end{pmatrix} : a \in (\mbz/5\mbz)^\times \right\},
\end{split}
\]
and $G_{5,i,k} = \pi_{\GL_2}^{-1}(G_{5,i,k}(5))$ for $i,k \in \{1,2 \}$.  Note that $-I \notin G_{5,i,k}$, for each $i,k \in \{1, 2 \}$.  We have 
\[
\begin{split}
\tilde{G}_{5,1}(5) &:= \tilde{G}_{5,1,1}(5) = \tilde{G}_{5,1,2}(5) = \left\{ \begin{pmatrix} \pm 1 & * \\ 0 & * \end{pmatrix} \right\}, \\
\tilde{G}_{5,2}(5) &:= \tilde{G}_{5,2,1}(5) = \tilde{G}_{5,2,2}(5) = \left\{ \begin{pmatrix} * & * \\ 0 & \pm 1 \end{pmatrix} \right\}.
\end{split}
\]
Define the functions $j_{5,1}(t), j_{5,2}(t) \in \mbq(t)$ by
\[
\begin{split}
j_{5,1}(t) &:= \frac{(t^4 - 12t^3 + 14t^2 + 12t + 1)^3}{t^5(t^2 - 11t - 1)}, \\
j_{5,2}(t) &:= \frac{(t^4 + 228t^3 + 494t^2 - 228t + 1)^3}{t(t^2 - 11t - 1)^5}.
\end{split}
\]
As detailed in \cite{zywina}, for any elliptic curve $E$ over $\mbq$ with $j$-invariant $j_E$, one has
\begin{equation*} 
\begin{split}
\rho_{E}(G_\mbq) \, \dot\subseteq \, \tilde{G}_{5,1} \; &\Longleftrightarrow \; \exists t_0 \in \mbq \text{ for which } j_E = j_{5,1}(t_0), \\
\rho_{E}(G_\mbq) \, \dot\subseteq \, \tilde{G}_{5,2} \; &\Longleftrightarrow \; \exists t_0 \in \mbq \text{ for which } j_E = j_{5,2}(t_0).
\end{split}
\end{equation*}
We define the coefficients $a_{4;5,i}(t)$, and $a_{6;5,i}(t)$ for each $i \in \{1, 2\}$ by \eqref{defofa4anda6}, the twist parameters $d_{5,i,k}(t) \in \mbq(t)$ for each $i,k \in \{1, 2\}$ by
\[
\begin{split}
d_{5,1,1}(t) &:= - \frac{(t^2+1)(t^4 - 18t^3 + 74t^2 + 18t + 1)}{2(t^4 - 12t^3 + 14t^2 + 12t + 1)}, \quad\quad\quad\quad\;\, d_{5,1,2}(t) := 5 d_{5,1,1}(t), \\
d_{5,2,1}(t) &:= - \frac{(t^2+1)(t^4 - 522t^3 - 10006t^2 + 522t + 1)}{2(t^4 + 228t^3 + 494t^2 - 228t + 1)}, \quad\quad d_{5,2,2}(t) := 5 d_{5,2,1}(t),
\end{split}
\]
and the elliptic curves $\mc{E}_{5,i,k}$ over $\mbq(t)$ by
\[
\mc{E}_{5,i,k} : \; d_{5,i,k}(t) y^2 = x^3 + a_{4;5,i}(t) x + a_{6;5,i}(t) \quad\quad \left( i,k \in \{1, 2 \} \right).
\]
As detailed in \cite{zywina}, for any elliptic curve $E$ over $\mbq$ and for each $i,k \in \{1, 2\}$, we have
\[
\rho_{E}(G_\mbq) \, \dot\subseteq \, G_{5,i,k} \; \Longleftrightarrow \; \exists t_0 \in \mbq \text{ for which } E \text{ is isomorphic over $\mbq$ to } \mc{E}_{5,i,k}\left( t_0 \right). \\
\]

\medskip

\subsection{The level $m = 6$}

We have $\mf{G}_{MT}^{\max}(0,6) = \{ G_{6,1,1}, G_{6,2,1}, G_{6,3,1}, G_{6,3,2} \}$, where $G_{6,i,k}(6) \subseteq \GL_2(\mbz/6\mbz)$ are given by
\begin{equation} \label{descriptionofgroupslevel6}
\begin{split}
G_{6,1,1}(6) &= \left\langle \begin{pmatrix} 1 & 1 \\ 0 & 5 \end{pmatrix}, \begin{pmatrix} 5 & 1 \\ 3 & 2 \end{pmatrix}, \begin{pmatrix} 3 & 2 \\ 4 & 3 \end{pmatrix} \right\rangle \simeq \GL_2(\mbz/2\mbz) \times_{\psi^{(1,1)}} \GL_2(\mbz/3\mbz), \\
G_{6,2,1}(6) &= \left\langle \begin{pmatrix} 1 & 1 \\ 0 & 5 \end{pmatrix}, \begin{pmatrix} 1 & 2 \\ 0 & 1 \end{pmatrix}, \begin{pmatrix} 2 & 3 \\ 3 & 5 \end{pmatrix} \right\rangle \simeq \GL_2(\mbz/2\mbz) \times_{\psi^{(2,1)}} \left\{ \begin{pmatrix} * & * \\ 0 & * \end{pmatrix} \right\}, \\
G_{6,3,1}(6) &= \left\langle \begin{pmatrix} 5 & 0 \\ 0 & 1 \end{pmatrix}, \begin{pmatrix} 5 & 5 \\ 0 & 5 \end{pmatrix}, \begin{pmatrix} 4 & 3 \\ 3 & 1 \end{pmatrix} \right\rangle \simeq \GL_2(\mbz/2\mbz) \times_{\psi^{(3,1)}} \left\{ \begin{pmatrix} * & * \\ 0 & * \end{pmatrix} \right\}, \\
G_{6,3,2}(6) &= \left\langle \begin{pmatrix} 1 & 0 \\ 0 & 5 \end{pmatrix}, \begin{pmatrix} 5 & 5 \\ 0 & 5 \end{pmatrix}, \begin{pmatrix} 4 & 3 \\ 3 & 1 \end{pmatrix} \right\rangle \simeq \GL_2(\mbz/2\mbz) \times_{\psi^{(3,2)}} \left\{ \begin{pmatrix} * & * \\ 0 & * \end{pmatrix} \right\},
\end{split}
\end{equation}
and $G_{6,i,k} = \pi_{\GL_2}^{-1}(G_{6,i,k}(6))$.  In the fibered product involving $\psi^{(1,1)} = (\psi_2^{(1,1)},\psi_3^{(1,1)})$ on the right-hand side of $G_{6,1,1}(6)$ above, the common quotient $\Gamma$ is $D_3$, the dihedral group of order $6$, the map $\psi_2^{(1,1)}$ is any isomorphism $\GL_2(\mbz/2\mbz) \simeq D_3$, and the map $\psi_3^{(1,1)} : \GL_2(\mbz/3\mbz) \longrightarrow D_3$ is a surjective homomorphism, whose kernel is
\[
\ker \psi_3^{(1,1)} = \left\langle \begin{pmatrix} 0 & 2 \\ 1 & 0 \end{pmatrix}, \begin{pmatrix} 2 & 1 \\ 1 & 1 \end{pmatrix} \right\rangle = \SL_2(\mbz/3\mbz)' \subseteq \GL_2(\mbz/3\mbz).
\]
In the fibered products involving $\psi^{(2,1)}$, $\psi^{(3,1)}$ and $\psi^{(3,2)}$, the underlying homomorphisms are as follows:  $\psi_2^{(2,1)} = \psi_2^{(3,1)} = \psi_2^{(3,2)} = \ve$, where $\ve$ is defined by \eqref{defofve}, and $\psi_3^{(2,1)}$, $\psi_3^{(3,1)}$, and $\psi_3^{(3,2)}$ are defined by
\[
\begin{split}
\psi_3^{(2,1)} : &\left\{ \begin{pmatrix} * & * \\ 0 & * \end{pmatrix} \right\} \longrightarrow \{ \pm 1 \}, \quad\quad \psi_3^{(2,1)}\left( \begin{pmatrix} a & b \\ 0 & d \end{pmatrix} \right) := \left( \frac{ad}{3} \right), \\
\psi_3^{(3,1)} : &\left\{ \begin{pmatrix} * & * \\ 0 & * \end{pmatrix} \right\} \longrightarrow \{ \pm 1 \}, \quad\quad \psi_3^{(3,1)}\left( \begin{pmatrix} a & b \\ 0 & d \end{pmatrix} \right) := \left( \frac{d}{3} \right), \\
\psi_3^{(3,2)} : &\left\{ \begin{pmatrix} * & * \\ 0 & * \end{pmatrix} \right\} \longrightarrow \{ \pm 1 \}, \quad\quad \psi_3^{(3,2)}\left( \begin{pmatrix} a & b \\ 0 & d \end{pmatrix} \right) := \left( \frac{a}{3} \right).
\end{split}
\]
Note that $-I \in G_{6,1,1}$ and $-I \in G_{6,2,1}$, but $-I \notin G_{6,3,k}$ for each $k \in \{ 1, 2 \}$.  We have 
\begin{equation} \label{tildeGsub63j}
\tilde{G}_{6,3,k}(6) \simeq  \GL_2(\mbz/2\mbz) \times \left\{ \begin{pmatrix} * & * \\ 0 & * \end{pmatrix} \right\} \quad\quad \left( k \in \{1, 2 \} \right).
\end{equation}
Let us set $\tilde{G}_{6,i} := \tilde{G}_{6,i,k}$ and note that $G_{6,i,1} = \tilde{G}_{6,i}$ for $i \in \{ 1, 2 \}$.  Also note that
$\level_{\GL_2}(\tilde{G}_{6,3}) = 3$.  The group $G_{6,1,1}$ is studied in \cite{braujones} (see also \cite{jonesmcmurdy} and \cite{morrow}); for any elliptic curve $E$ over $\mbq$ we have
\[
\rho_{E}(G_\mbq) \, \dot\subseteq \, G_{6,1,1} \; \Longleftrightarrow \; 
\begin{matrix} 
\left[ \mbq(E[2]) : \mbq \right] = 6 \text{ and } \mbq(E[2]) \subseteq \mbq(E[3]), \text{ or} \\
\mbq(E[2]) = \mbq(\mu_3) \text{ and } \rho_{E,3}(G_\mbq) \, \dot\subseteq \, \mc{N}_{\ns}(3),
\end{matrix} 
\]
where $\mc{N}_{\ns}(3)$ denotes the normalizer in $\GL_2(\mbz/3\mbz)$ of a non-split Cartan subgroup.  Define $j_{6,1}(t) \in \mbq(t)$ and the elliptic curve $\mc{E}_{6,1,1}$ over $\mbq(t,D)$ by
\[
\begin{split}
j_{6,1}(t) :=& 2^{10}3^3t^3(1-4t^3), \\
\mc{E}_{6,1,1} : \; Dy^2 =& x^3 + \frac{108 j_{6,1}(t)}{1728 - j_{6,1}(t)} x + \frac{432 j_{6,1}(t)}{1728 - j_{6,1}(t)}.
\end{split}
\]
As detailed in \cite{braujones}, for any elliptic curve $E$ over $\mbq$, we have
\begin{equation*} 
\rho_{E}(G_\mbq) \, \dot\subseteq \, G_{6,1,1} \; \Longleftrightarrow \; \exists t_0, D_0 \in \mbq \text{ for which } E \text{ is isomorphic over $\mbq$ to } \mc{E}_{6,1,1}(t_0,D_0).
\end{equation*}
Regarding the group $G_{6,2,1} = \tilde{G}_{6,2}$: by \eqref{descriptionofgroupslevel6}, Corollary \ref{keycorollaryforinterpretationofentanglements} and Lemma \ref{level2kummersubextensionlemma}, we have
\begin{equation} \label{from31to62}
\rho_{E}(G_\mbq) \, \dot\subseteq \, G_{6,2,1} \; \Longleftrightarrow \; \mbq\left( \sqrt{\gD_E} \right) = \mbq(\mu_3) \text{ and } \rho_{E,3}(G_\mbq) \, \dot\subseteq \, \left\{ \begin{pmatrix} * & * \\ 0 & * \end{pmatrix} \right\}.
\end{equation}
Recall $j_{3,1}(t) \in \mbq(t)$, defined by $j_{3,1}(t) := 27\frac{(t+1)(t+9)^3}{t^3}$ and the coefficients $a_{4;3,1}(t)$ and $a_{6;3,1}(t)$ defined by \eqref{defofa4anda6}; consider the elliptic curve $\mc{E}_{3,1}$ over $\mbq(t,D)$ defined by 
\[
\mc{E}_{3,1} : \;  y^2 = x^3 + D^2a_{4;3,1}(t) x + D^3a_{6;3,1}(t).
\]
The discriminant $\gD_{\mc{E}_{3,1}}$ of $\mc{E}_{3,1}$ satisfies
\begin{equation} \label{discriminantofmcEsub31}
\gD_{\mc{E}_{3,1}} = 2^{18}3^9\frac{D^6t^3(t+1)^2(t+9)^6}{(t^2-18t-27)^6}.
\end{equation}
In particular, $\mbq\left( \sqrt{\gD_{\mc{E}_{3,1}}} \right) = \mbq(\sqrt{3t})$, so by \eqref{from31to62}, we see that $\rho_{\mc{E}_{3,1}(t_0,D_0)}(G_\mbq) \, \dot\subseteq \, \tilde{G}_{6,2} = G_{6,2,1}$ if and only if $t_0 \in - (\mbq^\times)^2$.  We therefore set 
\[
j_{6,2}(t) := j_{3,1}(-t^2), \quad\quad a_{4;6,2}(t) := a_{4;3,1}(-t^2), \quad a_{6;6,2}(t) := a_{4;3,1}(-t^2)
\]
and define the elliptic curve $\mc{E}_{6,2,1}$ over $\mbq(t,D)$ by
\[
\mc{E}_{6,2,1} : \; D y^2 = x^3 + a_{4;6,2}(t) x + a_{6;6,2}(t).
\]
For any $E$ over $\mbq$, we then have
\[
\rho_{E}(G_\mbq) \, \dot\subseteq \, G_{6,2,1} \; \Longleftrightarrow \; \exists t_0, D_0 \in \mbq \text{ for which $E$ is isomorphic over $\mbq$ to } \mc{E}_{6,2,1}(t_0,D_0).
\]

Finally, we turn to the groups $G_{6,3,1}$ and $G_{6,3,2}$.  By \eqref{tildeGsub63j} and \eqref{level3jinvariantstatement}, for any $E$ over $\mbq$, we have
\[
\rho_{E}(G_\mbq) \, \dot\subseteq \, \tilde{G}_{6,3} \; \Longleftrightarrow \; \exists t_0, D_0 \in \mbq \text{ for which $E$ is isomorphic over $\mbq$ to } \mc{E}_{3,1}(t_0,D_0).
\]
On the other hand, \eqref{descriptionofgroupslevel6}, Corollary \ref{keycorollaryforinterpretationofentanglements} and Lemma \ref{level2kummersubextensionlemma} imply that
\[
\rho_{E}(G_\mbq) \, \dot\subseteq \, G_{6,3,k} \; \Longleftrightarrow \; \rho_{E}(G_\mbq) \, \dot\subseteq \, \tilde{G}_{6,3} \; \text{ and } \; \mbq\left( \sqrt{\gD_E} \right) = \mbq(E[3])^{\ker \psi_3^{(1,k)}} \quad \left( k \in \{ 1, 2 \} \right).
\]
Thus, by Lemma \ref{subfieldsoflevel3lemma} together with \eqref{discriminantofmcEsub31}, we are led to the twist parameters
\[
d_{6,3,1}(t) := \frac{2t(t+1)(t+9)}{t^2 - 18t - 27}, \quad\quad d_{6,3,2}(t) := - \frac{6t(t+1)(t+9)}{t^2 - 18t - 27}.
\]
We furthermore set
\[
a_{4;6,3}(t) := a_{4;3,1}(t), \quad a_{6;6,3}(t) := a_{6;3,1}(t)
\]
and define the elliptic curves $\mc{E}_{6,3,k}$ over $\mbq(t)$ by
\[
\mc{E}_{6,3,k} : \; d_{6,3,k}(t) y^2 = x^3 + a_{4;6,3}(t) x + a_{6;6,3}(t).
\]
Our discussion demonstrates that, for any elliptic curve $E$ over $\mbq$ and for each $k \in \{1, 2 \}$, we have
\[
\rho_{E}(G_\mbq) \, \dot\subseteq \, G_{6,3,k} \; \Longleftrightarrow \; \exists t_0 \in \mbq \text{ for which $E$ is isomorphic over $\mbq$ to } \mc{E}_{6,3,k}\left( t_0 \right).
\]

\medskip

\subsection{The level $m = 7$.}

We have $\mf{G}_{MT}^{\max}(0,7) = \{ G_{7,1,1}, G_{7,1,2}, G_{7,2,1}, G_{7,2,2}, G_{7,3,1}, G_{7,3,2} \}$, where the groups $G_{7,i,k}(7) \subseteq \GL_2(\mbz/7\mbz)$ are given by
\begin{equation*}
\begin{split}
G_{7,1,1}(7) &= \left\langle \begin{pmatrix} 1 & 1 \\ 0 & 1 \end{pmatrix}, \begin{pmatrix} 1 & 0 \\ 0 & 3 \end{pmatrix} \right\rangle = \left\{ \begin{pmatrix} 1 & * \\ 0 & * \end{pmatrix} \right\}, \\
G_{7,1,2}(7) &= \left\langle \begin{pmatrix} 1 & 1 \\ 0 & 1 \end{pmatrix}, \begin{pmatrix} 6 & 0 \\ 0 & 2 \end{pmatrix} \right\rangle = \left\{ \begin{pmatrix} \pm 1 & * \\ 0 & a^2 \end{pmatrix} : a \in (\mbz/7\mbz)^\times \right\}, \\
G_{7,2,1}(7) &= \left\langle \begin{pmatrix} 1 & 1 \\ 0 & 1 \end{pmatrix}, \begin{pmatrix} 3 & 0 \\ 0 & 1 \end{pmatrix} \right\rangle = \left\{ \begin{pmatrix} * & * \\ 0 & 1 \end{pmatrix} \right\}, \\
G_{7,2,2}(7) &= \left\langle \begin{pmatrix} 1 & 1 \\ 0 & 1 \end{pmatrix}, \begin{pmatrix} 2 & 0 \\ 0 & 6 \end{pmatrix} \right\rangle = \left\{ \begin{pmatrix} a^2 & * \\ 0 & \pm 1 \end{pmatrix} : a \in (\mbz/7\mbz)^\times \right\}, \\
G_{7,3,1}(7) &= \left\langle \begin{pmatrix} 1 & 1 \\ 0 & 1 \end{pmatrix}, \begin{pmatrix} 5 & 0 \\ 0 & 2 \end{pmatrix} \right\rangle = \left\{ \begin{pmatrix} \pm a^2 & * \\ 0 & a^2 \end{pmatrix} : a \in (\mbz/7\mbz)^\times \right\}, \\
G_{7,3,2}(7) &= \left\langle \begin{pmatrix} 1 & 1 \\ 0 & 1 \end{pmatrix}, \begin{pmatrix} 2 & 0 \\ 0 & 5 \end{pmatrix} \right\rangle = \left\{ \begin{pmatrix} a^2 & * \\ 0 & \pm a^2 \end{pmatrix} : a \in (\mbz/7\mbz)^\times \right\}
\end{split}
\end{equation*}
and $G_{7,i,k} = \pi_{\GL_2}^{-1}(G_{7,i,k}(7))$ for each $i \in \{ 1, 2, 3 \}$ and $k \in \{1,2 \}$.  Note that $-I \notin G_{7,i,k}$, for each $i,k$.  We have 
\begin{equation} \label{tildeGlevel7groups}
\begin{split}
\tilde{G}_{7,1}(7) &:= \tilde{G}_{7,1,1}(7) = \tilde{G}_{7,1,2}(7) = \left\{ \begin{pmatrix} \pm 1 & * \\ 0 & * \end{pmatrix} \right\}, \\
\tilde{G}_{7,2}(7) &:= \tilde{G}_{7,2,1}(7) = \tilde{G}_{7,2,2}(7) = \left\{ \begin{pmatrix} * & * \\ 0 & \pm 1 \end{pmatrix} \right\}, \\
\tilde{G}_{7,3}(7) &:= \tilde{G}_{7,3,1}(7) = \tilde{G}_{7,3,2}(7) = \left\{ \begin{pmatrix} a & * \\ 0 & \pm a \end{pmatrix} : a \in (\mbz/7\mbz)^\times \right\}.
\end{split}
\end{equation}
Define the functions $j_{7,1}(t), j_{7,2}(t), j_{7,3}(t) \in \mbq(t)$ by
\begin{equation} \label{tildeGlevel7jinvariants}
\begin{split}
j_{7,1}(t) &:= \frac{(t^2 - t + 1)^3(t^6 - 11t^5 + 30t^4 - 15t^3 - 10t^2 + 5t + 1)^3}{t^7(t-1)^7(t^3 - 8t^2 + 5t + 1)}, \\
j_{7,2}(t) &:= \frac{(t^2 - t + 1)^3(t^6 + 229t^5 + 270t^4 - 1695t^3 + 1430t^2 - 235t + 1)^3}{t(t-1)(t^3 - 8t^2 + 5t + 1)^7}, \\
j_{7,3}(t) &:= - \frac{(t^2 - 3t - 3)^3(t^2 - t + 1)^3(3t^2 - 9t + 5)^3(5t^2 - t - 1)^3}{(t^3 - 2t^2 - t + 1)(t^3 - t^2 - 2t + 1)^7}.
\end{split}
\end{equation}
As detailed in \cite{zywina}, for any elliptic curve $E$ over $\mbq$ with $j$-invariant $j_E$, one has
\begin{equation} \label{level7jinvariantstatement}
\begin{split}
\rho_{E}(G_\mbq) \, \dot\subseteq \, \tilde{G}_{7,1} \; &\Longleftrightarrow \; \exists t \in \mbq \text{ for which } j_E = j_{7,1}(t), \\
\rho_{E}(G_\mbq) \, \dot\subseteq \, \tilde{G}_{7,2} \; &\Longleftrightarrow \; \exists t \in \mbq \text{ for which } j_E = j_{7,2}(t), \\
\rho_{E}(G_\mbq) \, \dot\subseteq \, \tilde{G}_{7,3} \; &\Longleftrightarrow \; \exists t \in \mbq \text{ for which } j_E = j_{7,3}(t).
\end{split}
\end{equation}
We define the coefficients $a_{4;7,i}(t)$, and $a_{6;7,i}(t)$ for each $i \in \{1, 2, 3\}$ by \eqref{defofa4anda6}, the twist parameters $d_{7,i,k}(t) \in \mbq(t)$ for each $i \in \{ 1, 2, 3 \}$ and $k \in \{1, 2\}$ by
\[
\begin{split}
d_{7,1,1} &:=  - \frac{t^{12} - 18t^{11} + 117t^{10} - 354t^9 + 570t^8 - 486t^7 + 273t^6 - 222t^5 + 174t^4 - 46t^3 - 15t^2 + 6t + 1}{2(t^2 - t + 1)(t^6 - 11t^5 + 30t^4 - 15t^3 - 10t^2 + 5t + 1)} \\
d_{7,2,1} &:=  - \frac{\begin{pmatrix} t^{12} - 522t^{11} - 8955t^{10} + 37950t^9 - 70998t^8 + 131562t^7 - 253239t^6 + \\  316290t^5 - 218058t^4 + 80090t^3 - 14631t^2 + 510t + 1 \end{pmatrix}}{2(t^2 - t + 1)(t^6 + 229t^5 + 270t^4 - 1695t^3 + 1430t^2 - 235t + 1)} \\
d_{7,3,1} &:=  \frac{7(t^4 - 6t^3 + 17t^2 - 24t + 9)(3t^4 - 4t^3 - 5t^2 - 2t - 1)(9t^4 - 12t^3 - t^2 + 8t - 3)}{2(t^2 - 3t - 3)(t^2 - t + 1)(3t^2 - 9t + 5)(5t^2 - t - 1)}
\end{split}
\]
and $d_{7,i,2} := -7d_{7,i,1}$ for $i \in \{1, 2, 3 \}$; define the elliptic curves $\mc{E}_{7,i,k}$ over $\mbq(t)$ by
\[
\mc{E}_{7,i,k} : \; d_{7,i,k}(t) y^2 = x^3 + a_{4;7,i}(t) x + a_{6;7,i}(t) \quad\quad \left( i \in \{ 1, 2, 3 \}, \, k \in \{1, 2 \} \right).
\]
As may be found in \cite{zywina}, for any elliptic curve $E$ over $\mbq$ with $j$-invariant $j_E$, and for each $i \in \{ 1, 2, 3 \}$ and $k \in \{1, 2\}$, one has
\[
\rho_{E}(G_\mbq) \, \dot\subseteq \, G_{7,i,k} \; \Longleftrightarrow \; \exists t_0 \in \mbq \text{ for which } E \text{ is isomorphic over $\mbq$ to } \mc{E}_{7,i,k}\left( t_0 \right). \\
\]

\medskip

\subsection{The level $m=8$.}

We have $\mf{G}_{MT}^{\max}(0,8) = \{ G_{8,1,1}, G_{8,2,1} \}$, where $G_{8,1,1}(8), G_{8,2,1}(8) \subseteq \GL_2(\mbz/8\mbz)$ are given by
\[
\begin{split}
G_{8,1,1}(8) &= \left\langle \begin{pmatrix} 1 & 1 \\ 0 & 7 \end{pmatrix}, \begin{pmatrix} 3 & 0 \\ 0 & 7 \end{pmatrix}, \begin{pmatrix} 5 & 5 \\ 5 & 2 \end{pmatrix} \right\rangle, \\
G_{8,2,1}(8) &= \left\langle \begin{pmatrix} 5 & 6 \\ 6 & 7 \end{pmatrix}, \begin{pmatrix} 3 & 0 \\ 0 & 7 \end{pmatrix}, \begin{pmatrix} 5 & 5 \\ 5 & 2 \end{pmatrix} \right\rangle,
\end{split}
\]
and $G_{8,i,1} = \pi_{\GL_2}^{-1}(G_{8,i,1}(8))$ for each $i \in \{1, 2 \}$.  Note that $-I \notin G_{8,i,1}$, and we define groups $\tilde{G}_{8,1} := \tilde{G}_{8,1,1}$ and $\tilde{G}_{8,2} := \tilde{G}_{8,2,1}$.

As a consequence of \cite[Lemma 28]{sutherlandzywina} and \cite[Proposition 3.1]{sutherlandzywina}, for any group $G \in \mf{G}(0,8)$, we have
\begin{equation} \label{iszsadmissiblecondition}
X_{\tilde{G}}(\mbq) \neq \emptyset \; \Longleftrightarrow \; \exists g \in \tilde{G} \text{ that is $\GL_2(\mbz/8\mbz)$-conjugate to } \begin{pmatrix} 1 & 0 \\ 0 & -1 \end{pmatrix} \text{ or } \begin{pmatrix} 1 & 1 \\ 0 & -1 \end{pmatrix}.
\end{equation}
A computation shows that $\tilde{G}_{8,2}(8)$ fails the condition on the right-hand side of \eqref{iszsadmissiblecondition}, whereas $\tilde{G}_{8,1}(8)$ satisfies it.  Thus, $\left| X_{\tilde{G}_{8,2}}(\mbq) \right| = 0$ and $\left| X_{\tilde{G}_{8,1}}(\mbq) \right| = \infty$, and we will therefore restrict our consideration to the groups $\tilde{G}_{8,1}$ and $G_{8,1,1}$.  The group $\tilde{G}_{8,1}$ has $\GL_2$-level $4$, and one may verify by direct computation that
\[
\begin{split}
&\tilde{G}_{8,1}(4) \subseteq \GL_2(\mbz/4\mbz)_{\chi_4 = \ve}, \quad \tilde{G}_{8,1}(2) = \GL_2(\mbz/2\mbz) \quad \text{ and } \\
&\ker\left( \GL_2(\mbz/4\mbz) \rightarrow \GL_2(\mbz/2\mbz) \right) \cap \SL_2(\mbz/4\mbz)' \cap \tilde{G}_{8,1}(4) = \{ I \},
\end{split}
\]
where the group $\GL_2(\mbz/4\mbz)_{\chi_4 = \ve}$ is as in \eqref{defofGL2subchi4equalsve}.  Furthermore, $\tilde{G}_{8,1}(4)$ is the unique subgroup (up to $\doteq$) of $\GL_2(\mbz/4\mbz)$ satisfying these three conditions.
By the Galois correspondence and Lemma \ref{level4kummersubextensionlemma}, it follows that, for any elliptic curve $E$ over $\mbq$, we have
\begin{equation} \label{rhocontainedintildeGsub81}
\begin{split}
\rho_{E}(G_\mbq) \, \dot\subseteq \, \tilde{G}_{8,1} \quad &\Longleftrightarrow \quad \rho_{E,4}(G_\mbq) \, \dot\subseteq \, \tilde{G}_{8,1}(4) \\
&\Longleftrightarrow \quad \begin{matrix} \mbq(\sqrt{\gD_E}) = \mbq(i), \; [\mbq(E[2]) : \mbq ] =6, \\ \text{ and } \mbq(E[4]) = \mbq(E[2], \gD_E^{1/4}). \end{matrix}
\end{split}
\end{equation}
Define the rational functions $g_{8,1}(t), f_{8,1}(t)$ and $j_{8,1}(t) \in \mbq(t)$ by
\[
g_{8,1}(t) := - \frac{t^2 + 2t - 2}{t}, \quad\quad f_{8,1}(t) :=  4t^3(8-t), \quad\quad j_{8,1}(t) := f_{8,1}(g_{8,1}(t)).
\]
The group $\tilde{G}_{8,1}$ appears under the label 4$\text{D}^0$-4a in \cite{sutherlandzywina}, wherein it is shown that, for any elliptic curve $E$ over $\mbq$ of $j$-invariant $j_E$, we have
\[
\rho_{E}(G_\mbq) \, \dot\subseteq \, \tilde{G}_{8,1} \; \Longleftrightarrow \; \exists t_0 \in \mbq \text{ for which } j_E = j_{8,1}(t_0).
\]
The group $G_{8,1,1}$ entails an additional vertical entanglement.  Specifically, we have
\[
G_{8,1,1} \cap \pi_{\GL_2}\left( \SL_2(\mbz/4\mbz)' \right) = G_{8,1,1} \cap \pi_{\GL_2}\left( \SL_2(\mbz/8\mbz) \right),
\]
and $G_{8,1,1}$ is the unique maximal subgroup of $\tilde{G}_{8,1}$ (with respect to $\dot\subseteq$) that satisfies this.  It follows that, for any elliptic curve $E$ over $\mbq$,
\begin{equation} \label{conditionforrhotobeinsideG811}
\rho_{E}(G_\mbq) \, \dot\subseteq \, G_{8,1,1} \quad \Longleftrightarrow \quad \begin{matrix} \mbq(\sqrt{\gD_E}) = \mbq(i), \; \mbq(E[4]) = \mbq(E[2],\gD_E^{1/4}), \\ [ \mbq(E[2]) : \mbq] = 6 \; \text{ and } \; \mbq(i, \gD_E^{1/4}) = \mbq(\mu_8). \end{matrix}
\end{equation}
We define the coefficients $a_{4;8,1}(t)$ and $a_{6;8,1}(t)$ by \eqref{defofa4anda6} and consider the elliptic curve $\mc{E}_{8,1}$ over $\mbq(t,D)$ defined by 
\[
\mc{E}_{8,1} : \; y^2 = x^3 + D^2a_{4;8,1}(t) x + D^3a_{6;8,1}(t).
\]
By \eqref{rhocontainedintildeGsub81}, for any 
$t_0, D_0 \in \mbq$ for which $\mc{E}_{8,1}(t_0,D_0)$ is an elliptic curve, we have $\mbq\left( \sqrt{\gD_{\mc{E}_{8,1}(t_0,D_0)}} \right) = \mbq(i)$ and $\mbq(\mc{E}_{8,1}(t_0,D_0)[4]) = \mbq\left( \mc{E}_{8,1}(t_0,D_0)[2],\gD_{\mc{E}_{8,1}(t_0,D_0)}^{1/4} \right)$.
The discriminant $\gD_{\mc{E}_{8,1}}$ satisfies
\[
\gD_{\mc{E}_{8,1}} = - 2^{16} 3^{12} \frac{D^6t^4(t^2 +2t-2)^6(t^2+10t-2)^2}{(t^2+2)^6(t^2+8t-2)^6},
\]
and thus, using $\zeta_8 = \frac{\sqrt{2}}{2} + \frac{\sqrt{2}}{2} i$, we find that
\[
\mbq\left( i, \gD_{\mc{E}_{8,1}(t_0,D_0)}^{1/4} \right) = \mbq\left( i, \sqrt{\frac{2D(t^2 +2t-2)(t^2+10t-2)}{(t^2+2)(t^2+8t-2)}} \right).
\]
Thus, it follows from \eqref{rhocontainedintildeGsub81} and \eqref{conditionforrhotobeinsideG811} that
\[
\rho_{\mc{E}_{8,1}(t_0,D_0)}(G_\mbq) \, \dot\subseteq \, G_{8,1,1} \; \Longleftrightarrow \; D = \pm \frac{(t^2 +2t-2)(t^2+10t-2)}{(t^2+2)(t^2+8t-2)}.
\]
Thus, we are led to the pair of twist parameters
$
d_{8,1,1}^{\pm}(t) := \pm \frac{(t^2 +2t-2)(t^2+10t-2)}{(t^2+2)(t^2+8t-2)}.
$
Finally, noting that $j_{8,1}(-2/t) = j_{8,1}(t)$ and $d_{8,1,1}^{\pm}(-2/t) = d_{8,1,1}^{\mp}(t)$, we are led to the single twist parameter 
\[
d_{8,1,1}(t) := \frac{(t^2 +2t-2)(t^2+10t-2)}{(t^2+2)(t^2+8t-2)},
\]
and, defining the elliptic curve $\mc{E}_{8,1,1}$ over $\mbq(t)$ by
\[
\mc{E}_{8,1,1} : \; d_{8,1,1}(t) y^2 = x^3 + a_{4;8,1}(t)x + a_{6;8,1}(t),
\]
we have that, for each elliptic curve $E$ over $\mbq$ with $j$-invariant $j_E$,
\[
\rho_{E}(G_\mbq) \, \dot\subseteq \, G_{8,1,1} \; \Longleftrightarrow \; \exists t_0 \in \mbq \text{ for which $E$ is isomorphic over $\mbq$ to $\mc{E}_{8,1,1}\left( t_0 \right)$.}
\]

\medskip

\subsection{The level $m=9$.}

We have $\mf{G}_{MT}^{\max}(0,9) = \{ G_{9,1,1}, G_{9,2,1}, G_{9,3,1}, G_{9,4,1}, G_{9,5,1} \}$, where $G_{9,i,1}(9) \subseteq \GL_2(\mbz/9\mbz)$ are given by
\[
\begin{split}
G_{9,1,1}(9) &= \left\langle \begin{pmatrix} 1 & 3 \\ 0 & 1 \end{pmatrix}, \begin{pmatrix} 5 & 0 \\ 3 & 2 \end{pmatrix}, \begin{pmatrix} 4 & 2 \\ 0 & 5 \end{pmatrix} \right\rangle, \\
G_{9,2,1}(9) &= \left\langle \begin{pmatrix} 2 & 1 \\ 0 & 5 \end{pmatrix}, \begin{pmatrix} 4 & 0 \\ 3 & 5 \end{pmatrix} \right\rangle, \\
G_{9,3,1}(9) &= \left\langle \begin{pmatrix} 1 & 3 \\ 0 & 1 \end{pmatrix}, \begin{pmatrix} 5 & 2 \\ 3 & 5 \end{pmatrix}, \begin{pmatrix} 4 & 0 \\ 0 & 5 \end{pmatrix} \right\rangle, \\
G_{9,4,1}(9) &= \left\langle \begin{pmatrix} 0 & 2 \\ 4 & 1 \end{pmatrix}, \begin{pmatrix} 4 & 3 \\ 5 & 4 \end{pmatrix}, \begin{pmatrix} 4 & 5 \\ 0 & 5 \end{pmatrix} \right\rangle, \\
G_{9,5,1}(9) &= \left\langle \begin{pmatrix} 5 & 7 \\ 2 & 8 \end{pmatrix}, \begin{pmatrix} 1 & 0 \\ 0 & 4 \end{pmatrix} \right\rangle
\end{split}
\]
and $G_{9,i,1} = \pi_{\GL_2}^{-1}(G_{9,i,1}(9))$ for each $i \in \{1, 2, 3, 4, 5 \}$.  We have that $-I \in G_{9,i,1}$ for each $i \in \{1, 2, 3, 4, 5 \}$; as usual we define $\tilde{G}_{9,i} := \tilde{G}_{9,i,1}$, which equals $G_{9,i,1}$ in this case.  The group $\tilde{G}_{9,5}(9)$ fails the right-hand condition in \eqref{iszsadmissiblecondition}, whereas, for $i \in \{ 1, 2, 3, 4 \}$, the groups $\tilde{G}_{9,i}(9)$ satisfy it.  Since $X_{\tilde{G}_{9,5}}$ is a thus conic with no rational points, we will restrict our consideration to the first four groups in our list, which appear in \cite{sutherlandzywina} under the labels 9$\text{H}^0$-9c, 9$\text{I}^0$-9b, 9$\text{J}^0$-9c, and 9$\text{F}^0$-9a, respectively.  We define the functions
\[
\begin{array}{lllllllllll}
f_{9,1}(t) &:= & \frac{(t+3)^3(t+27)}{t} & & g_{9,1}(t) &:= & \frac{729}{t^3-27} & & h_{9,1}(t) &:= & \frac{-6(t^3 - 9t)}{t^3 + 9t^2 - 9t - 9} \\
& & & & g_{9,2}(t) &:= & t(t^2 + 9t + 27) & & h_{9,2}(t) &:= & \frac{-3(t^3 + 9t^2 - 9t - 9)}{t^3 + 3t^2 - 9t - 3} \\
& & & &  g_{9,3}(t) &:= & t^3 & & h_{9,3}(t) &:= & \frac{3(t^3 + 3t^2 - 9t - 3)}{t^3 - 3t^2 - 9t + 3} 
\end{array}
\]
and the $j$-invariants
\[
\begin{split}
j_{9,1}(t) &:= f_{9,1}\left(g_{9,1}\left(h_{9,1}(t) \right) \right), \quad j_{9,2}(t) := f_{9,1}\left(g_{9,2}\left(h_{9,2}(t) \right) \right), \quad j_{9,3}(t) := f_{9,1}\left(g_{9,3}\left(h_{9,3}(t) \right) \right), \\
j_{9,4}(t) &:= \frac{3^7(t^2-1)^3(t^6 + 3t^5 + 6t^4 + t^3 - 3t^2 + 12t + 16)^3(2t^3 + 3t^2 - 3t - 5)}{(t^3 - 3t - 1)^9}.
\end{split}
\]
As demonstrated in \cite{sutherlandzywina}, for any elliptic curve $E$ over $\mbq$ with $j$-invariant $j_E$ and for each $i \in \{1, 2, 3, 4 \}$, we have
\[
\rho_{E}(G_\mbq) \, \dot\subseteq \, \tilde{G}_{9,i} \; \Longleftrightarrow \; \exists t_0 \in \mbq \text{ for which } j_E = j_{9,i}(t_0).
\]
We define the coefficients $a_{4;9,i}(t)$ and $a_{6;9,i}(t)$ by \eqref{defofa4anda6} and consider the elliptic curve $\mc{E}_{9,i}$ over $\mbq(t,D)$ defined by 
\[
\mc{E}_{9,i} : \; D y^2 = x^3 + a_{4;9,i}(t) x + a_{6;9,i}(t).
\]
Since $G_{9,i,1} = \tilde{G}_{9,1}$ for each $i$, it follows immediately that, for each elliptic curve $E$ over $\mbq$ with $j$-invariant $j_E$ and for each $i \in \{1, 2, 3, 4 \}$, we have
\[
\rho_{E}(G_\mbq) \, \dot\subseteq \, G_{9,i,1} \; \Longleftrightarrow \; \exists t_0, D_0 \in \mbq \text{ for which $E$ is isomorphic over $\mbq$ to $\mc{E}_{9,i}\left( t_0,D_0 \right)$.}
\]

\medskip

\subsection{The level $m = 10$.}

We have $\mf{G}_{MT}^{\max}(0,10) = \{ G_{10,1,1}, G_{10,1,2}, G_{10,2,1}, G_{10,2,2}, G_{10,3,1} \}$, where $G_{10,i,k}(10) \subseteq \GL_2(\mbz/10\mbz)$ are given by
\begin{equation} \label{descriptionofgroupslevel10}
\begin{split}
G_{10,1,1}(10) &= \left\langle  \begin{pmatrix} 9 & 9 \\ 0 & 9 \end{pmatrix}, \begin{pmatrix} 6 & 5 \\ 5 & 1 \end{pmatrix}, \begin{pmatrix} 1 & 0 \\ 0 & 3 \end{pmatrix} \right\rangle \simeq \GL_2(\mbz/2\mbz) \times_{\psi^{(1,1)}} \left\{ \begin{pmatrix} \pm 1 & * \\ 0 & * \end{pmatrix} \right\}, \\
G_{10,1,2}(10) &= \left\langle  \begin{pmatrix} 9 & 9 \\ 0 & 9 \end{pmatrix}, \begin{pmatrix} 6 & 5 \\ 5 & 1 \end{pmatrix}, \begin{pmatrix} 9 & 0 \\ 0 & 3 \end{pmatrix} \right\rangle \simeq \GL_2(\mbz/2\mbz) \times_{\psi^{(1,2)}} \left\{ \begin{pmatrix} \pm 1 & * \\ 0 & * \end{pmatrix} \right\}, \\
G_{10,2,1}(10) &= \left\langle \begin{pmatrix} 9 & 9 \\ 0 & 9 \end{pmatrix}, \begin{pmatrix} 6 & 5 \\ 5 & 1 \end{pmatrix}, \begin{pmatrix} 3 & 0 \\ 0 & 9 \end{pmatrix} \right\rangle \simeq \GL_2(\mbz/2\mbz) \times_{\psi^{(2,1)}} \left\{ \begin{pmatrix} * & * \\ 0 & \pm 1 \end{pmatrix} \right\}, \\
G_{10,2,2}(10) &= \left\langle  \begin{pmatrix} 9 & 9 \\ 0 & 9 \end{pmatrix}, \begin{pmatrix} 6 & 5 \\ 5 & 1 \end{pmatrix}, \begin{pmatrix} 7 & 0 \\ 0 & 1 \end{pmatrix} \right\rangle \simeq \GL_2(\mbz/2\mbz) \times_{\psi^{(2,2)}} \left\{ \begin{pmatrix} * & * \\ 0 & \pm 1 \end{pmatrix} \right\}, \\
G_{10,3,1}(10) &= \left\langle \begin{pmatrix} 4 & 9 \\ 9 & 6 \end{pmatrix}, \begin{pmatrix} 1 & 3 \\ 9 & 8 \end{pmatrix}, \begin{pmatrix} 9 & 0 \\ 0 & 1 \end{pmatrix} \right\rangle \simeq \GL_2(\mbz/2\mbz) \times_{\psi^{(3,1)}} G_{S_4}(5)
\end{split}
\end{equation}
and $G_{10,i,k} = \pi_{\GL_2}^{-1}(G_{10,i,k}(10))$.  In the fibered product on the right-hand side of $G_{10,3,1}(10)$, the group $G_{S_4}(5)$ denotes the unique (up to conjugation in $\GL_2(\mbz/5\mbz)$) subgroup of $\GL_2(\mbz/5\mbz)$ of index $5$ (its image in $\PGL_2(\mbz/5\mbz)$ is isomorphic to $S_4$, the symmetric group on $4$ symbols), and in that fibered product, the underlying maps $\psi_5^{(3,1)} = (\psi_2^{(3,1)},\psi_5^{(3,1)})$, surject onto a common quotient isomorphic to $D_3$, the dihedral group of order $6$.  The map $\psi_2^{(3,1)}$ is any isomorphism $\GL_2(\mbz/2\mbz) \simeq D_3$, and the map $\psi_5^{(3,1)} : G_{S_4}(5) \longrightarrow D_3$ is the restriction of the projection map $\GL_2(\mbz/5\mbz) \longrightarrow \PGL_2(\mbz/5\mbz)$, followed by any surjection $S_4 \longrightarrow D_3$; its kernel is $\mc{N}_{\s}(5)$, the normalizer in $\GL_2(\mbz/5\mbz)$ of a split Cartan subgroup.  In the fibered products involving $\psi^{(1,1)}$, $\psi^{(1,2)}$, $\psi^{(2,1)}$ and $\psi^{(2,2)}$, the underlying homomorphisms are as follows:  $\psi_2^{(1,1)} = \psi_2^{(1,2)} = \psi_2^{(2,1)} = \psi_2^{(2,2)} = \ve$ as in \eqref{defofve}, whereas $\psi_5^{(1,1)}$, $\psi_5^{(1,2)}$, $\psi_5^{(2,1)}$ and $\psi_5^{(2,2)}$ are defined by
\[
\begin{split}
\psi_5^{(1,1)} : &\left\{ \begin{pmatrix} \pm 1 & * \\ 0 & * \end{pmatrix} \right\} \longrightarrow \{ \pm 1 \}, \quad\quad \psi_5^{(1,1)}\left( \begin{pmatrix} a & b \\ 0 & d \end{pmatrix} \right) := a \in \{ \pm 1 \}, \\
\psi_5^{(1,2)} : &\left\{ \begin{pmatrix} \pm 1 & * \\ 0 & * \end{pmatrix} \right\} \longrightarrow \{ \pm 1 \}, \quad\quad \psi_5^{(1,2)}\left( \begin{pmatrix} a & b \\ 0 & d \end{pmatrix} \right) := \left( \frac{d}{5} \right) a \in \{ \pm 1 \}, \\
\psi_5^{(2,1)} : &\left\{ \begin{pmatrix} * & * \\ 0 & \pm 1 \end{pmatrix} \right\} \longrightarrow \{ \pm 1 \}, \quad\quad \psi_5^{(2,1)}\left( \begin{pmatrix} a & b \\ 0 & d \end{pmatrix} \right) := \left( \frac{a}{5} \right) d \in \{ \pm 1 \}, \\
\psi_5^{(2,2)} : &\left\{ \begin{pmatrix} * & * \\ 0 & \pm 1 \end{pmatrix} \right\} \longrightarrow \{ \pm 1 \}, \quad\quad \psi_5^{(2,2)}\left( \begin{pmatrix} a & b \\ 0 & d \end{pmatrix} \right) := d \in \{ \pm 1 \}.
\end{split}
\]
We note that $-I \in G_{10,3,1}$ and $-I \notin G_{10,i,k}$, for each $i, k \in \{ 1, 2 \}$; we have 
\begin{equation} \label{tildeGsub10ij}
\begin{split}
\tilde{G}_{10,1,k}(10) &\simeq  \GL_2(\mbz/2\mbz) \times \left\{ \begin{pmatrix} \pm 1 & * \\ 0 & * \end{pmatrix} \right\}, \\
\tilde{G}_{10,2,k}(10) &\simeq  \GL_2(\mbz/2\mbz) \times \left\{ \begin{pmatrix} * & * \\ 0 & \pm 1 \end{pmatrix} \right\}\quad\quad \left( k \in \{1, 2 \} \right).
\end{split}
\end{equation}
Let us set $\tilde{G}_{10,i} := \tilde{G}_{10,i,k}$ and note that $G_{10,3,1} = \tilde{G}_{10,3}$.  Also note that
$\level_{\GL_2}(\tilde{G}_{10,i}) = 5$ for $i \in \{1, 2 \}$.  The group $\tilde{G}_{10,3}$ is studied in \cite{jonesmcmurdy}; for any elliptic curve $E$ over $\mbq$ we have
\[
\rho_{E}(G_\mbq) \, \dot\subseteq \, \tilde{G}_{10,3} \; \Longleftrightarrow \; 
\begin{matrix} 
\mbq(\sqrt{5}) \subsetneq \mbq(E[2]) \subseteq \mbq(E[5]) \text{ and }  \rho_{E,5}(G_\mbq) = G_{S_4}(5), \text{ or} \\
\mbq(E[2]) = \mbq(\sqrt{5}) \text{ and } \rho_{E,5}(G_\mbq) \subsetneq G_{S_4}(5).
\end{matrix} 
\]
Define $f_{10,3}(t), g_{10,3}(t) \in \mbq(t)$ by
\[
f_{10,3}(t) :=  t^3(t^2 + 5t + 40), \quad\quad g_{10,3}(t) :=  \frac{3t^6 + 12t^5 + 80t^4 + 50t^3 - 20t^2 - 8t + 8}{(t-1)^2(t^2 + 3t + 1)^2}
\]
and the $j$-invariant $j_{10,3}(t) \in \mbq(t)$ and elliptic curve $\mc{E}_{10,3,1}$ over $\mbq(t,D)$ by
\[
\begin{split}
j_{10,3}(t) :=& f_{10,3}\left( g_{10,3}(t) \right), \\
\mc{E}_{10,3,1} : \; D y^2 =& x^3 + \frac{108j_{10,3}(t)}{1728 - j_{10,3}(t)} x + \frac{432 j_{10,3}(t)}{1728 - j_{10,3}(t)}.
\end{split}
\]
As proved in \cite{jonesmcmurdy}, for any elliptic curve $E$ over $\mbq$, we have
\begin{equation*} 
\rho_{E}(G_\mbq) \, \dot\subseteq \, G_{10,3,1} \; \Longleftrightarrow \; \exists t_0, D_0 \in \mbq \text{ for which } E \text{ is isomorphic over $\mbq$ to } \mc{E}_{10,3,1}(t_0,D_0).
\end{equation*}

Regarding the groups $G_{10,i,k}$ for $i,k \in \{1, 2\}$, we first consider the groups $\tilde{G}_{10,i}$.  Given \eqref{tildeGsub10ij}, we may apply results in \cite{zywina}, which exhibits the $j$-invariants
\[
j_{5,1}(t) := \frac{(t^4 - 12t^3 + 14t^2 + 12t + 1)^3}{t^5(t^2 - 11t - 1)}, \quad\quad
j_{5,2}(t) :=  \frac{(t^4 + 228t^3 + 494t^2 - 228t + 1)^3}{t(t^2 - 11t - 1)^5}
\]
and shows that, for any elliptic curve $E$ over $\mbq$ with $j$-invariant $j_E$, we have
\[
\begin{split}
\rho_{E}(G_\mbq) \, \dot\subseteq \, \tilde{G}_{10,1} \; &\Longleftrightarrow \; \exists t_0 \in \mbq \text{ for which } j_E = j_{5,1}(t_0), \\
\rho_{E}(G_\mbq) \, \dot\subseteq \, \tilde{G}_{10,2} \; &\Longleftrightarrow \; \exists t_0 \in \mbq \text{ for which } j_E = j_{5,2}(t_0).
\end{split}
\]
If $E$ is an elliptic curve satisfying $\rho_{E,5}(G_\mbq) \, \dot\subseteq \, \left\{ \begin{pmatrix} * & * \\ 0 & * \end{pmatrix} \right\}$ then there is a $G_\mbq$-stable cyclic subgroup $\langle P \rangle \subseteq E[5]$; given any such Galois-stable cyclic subgroup $\langle P \rangle$, we let $E_{\langle P \rangle}' := E/\langle P \rangle$ denote the associated isogenous curve (which is necessarily defined over $\mbq$).  By \eqref{tildeGsub10ij}, we have that
\begin{equation*} 
\begin{split}
\rho_{E}(G_\mbq) \, \dot\subseteq \, \tilde{G}_{10,1} \; &\Longleftrightarrow \; \exists \text{ a $G_\mbq$-stable } \langle P \rangle \subseteq E[5] \quad\quad\;\;\; \text{ with } \quad [ \mbq( P ) : \mbq ] \leq 2, \\
\rho_{E}(G_\mbq) \, \dot\subseteq \, \tilde{G}_{10,2} \; &\Longleftrightarrow \; \begin{matrix} \exists \text{ a $G_\mbq$-stable } \langle P \rangle \subseteq E[5] \text{ and} \\ \exists \text{ a $G_\mbq$-stable  } \langle P' \rangle \subseteq E_{\langle P \rangle}' [5] \end{matrix} \quad \text{ with } \quad [ \mbq( P' ) : \mbq ] \leq 2.
\end{split}
\end{equation*}
Furthermore, by \eqref{descriptionofgroupslevel10}, Corollary \ref{keycorollaryforinterpretationofentanglements} and Lemma \ref{level2kummersubextensionlemma}, we have
\begin{equation} \label{conditionsforGsub10ij}
\begin{split}
\rho_{E}(G_\mbq) \, \dot\subseteq \, G_{10,1,1} \; &\Longleftrightarrow \; \exists \text{ a $G_\mbq$-stable } \langle P \rangle \subseteq E[5] \quad\quad\;\;\; \text{ with } \quad\; \mbq(P) = \mbq(\sqrt{\gD_E}), \\
\rho_{E}(G_\mbq) \, \dot\subseteq \, G_{10,1,2} \; &\Longleftrightarrow \; \exists \text{ a $G_\mbq$-stable } \langle P \rangle \subseteq E[5] \quad\quad\;\;\; \text{ with } \quad\; \mbq(P) = \mbq(\sqrt{5\gD_E}), \\
\rho_{E}(G_\mbq) \, \dot\subseteq \, G_{10,2,1} \; &\Longleftrightarrow \; \begin{matrix} \exists \text{ a $G_\mbq$-stable  } \langle P \rangle \subseteq E[5] \text{ and} \\  \exists \text{ a $G_\mbq$-stable } \langle P' \rangle \subseteq E_{\langle P \rangle}'[5] \end{matrix} \quad \text{ with } \quad \mbq(P') = \mbq( \sqrt{5\gD_E}), \\
\rho_{E}(G_\mbq) \, \dot\subseteq \, G_{10,2,2} \; &\Longleftrightarrow \; \begin{matrix} \exists \text{ a $G_\mbq$-stable  } \langle P \rangle \subseteq E[5] \text{ and} \\  \exists \text{ a $G_\mbq$-stable } \langle P' \rangle \subseteq E_{\langle P \rangle}'[5] \end{matrix} \quad \text{ with } \quad \mbq(P') = \mbq( \sqrt{\gD_E}).
\end{split}
\end{equation} 
We define the coefficients $a_{4;10,i}(t)$  $a_{6;10,i}(t)$ defined by \eqref{defofa4anda6}; consider the elliptic curves $\mc{E}_{10,i}$ over $\mbq(t,D)$ defined by 
\[
\begin{split}
\mc{E}_{10,1} &: \;  y^2 = x^3 + D^2a_{4;10,1}(t) x + D^3a_{6;10,1}(t), \\
\mc{E}_{10,2} &: \;  y^2 = x^3 + D^2a_{4;10,2}(t) x + D^3a_{6;10,2}(t).
\end{split}
\]
We have
\[
\begin{split}
\gD_{\mc{E}_{10,1}} = \frac{2^{18}3^{12}D^6 t^5 (t^2 - 11t - 1) (t^4 - 12t^3 + 14t^2 + 12t + 1)^6}{(t^2 + 1)^6 (t^4 - 18t^3 + 74t^2 + 18t + 1)^6}, \\
\gD_{\mc{E}_{10,2}} = \frac{2^{18}3^{12}D^6 t (t^2 - 11t - 1)^5 (t^4 + 228t^3 + 494t^2 - 228t + 1)^6}{(t^2 + 1)^6 (t^4 - 522t^3 - 10006t^2 + 522t + 1)^6},
\end{split}
\]
and thus 
\[
\mbq\left( \sqrt{\gD_{\mc{E}_{10,1}}} \right) = \mbq\left( \sqrt{\gD_{\mc{E}_{10,2}}} \right) = \mbq \left( \sqrt{t(t^2 - 11t - 1)} \right).
\]
By Lemma \ref{subfieldsoflevel5lemma}, we are led to the twist parameters
\[
\begin{split}
d_{10,1,1}(t) &:= \frac{-2t(t^2-11t-1)(t^4 - 12t^3 + 14t^2 + 12t + 1)}{(t^2 + 1)(t^4 - 18t^3 + 74t^2 + 18t + 1)}, \quad\quad\quad\quad d_{10,1,2}(t) := 5 d_{10,1,1}(t), \\
d_{10,2,1}(t) &:= \frac{-10t(t^2-11t-1)(t^4 + 228t^3 + 494t^2 - 228t + 1)}{(t^2 + 1) (t^4 - 522t^3 - 10006t^2 + 522t + 1)}, \quad\quad d_{10,2,2}(t) := 5 d_{10,2,1}(t),
\end{split}
\]
and to the elliptic curves $\mc{E}_{10,i,k}$ over $\mbq(t)$, defined by
\[
\mc{E}_{10,i,k} : \; d_{10,i,k}(t) y^2 = x^3 + a_{4;10,i}(t) x + a_{6; 10,i}(t) \quad\quad \left( i, k \in \{ 1, 2 \} \right).
\]
By \eqref{conditionsforGsub10ij} and Lemma \ref{subfieldsoflevel5lemma}, for each elliptic curve $E$ over $\mbq$ and for each $i, k \in \{1, 2 \}$, we have
\[
\rho_{E}(G_\mbq) \, \dot\subseteq \, G_{10,i,k} \; \Longleftrightarrow \; \exists t_0 \in \mbq \text{ for which $E$ is isomorphic over $\mbq$ to $\mc{E}_{10,i,k}\left( t_0 \right)$.}
\]

\medskip

\subsection{The level $m = 12$.}

We have $\mf{G}_{MT}^{\max}(0,12) = \{ G_{12,1,1}, G_{12,2,1}, G_{12,3,1}, G_{12,4,1}, G_{12,4,2} \}$, where $G_{12,i,k}(12) \subseteq \GL_2(\mbz/12\mbz)$ are given by
\begin{equation} \label{descriptionofgroupslevel12}
\begin{split}
G_{12,1,1}(12) &= \left\langle  \begin{pmatrix} 7 & 7 \\ 0 & 5 \end{pmatrix}, \begin{pmatrix} 5 & 7 \\ 3 & 2 \end{pmatrix}, \begin{pmatrix} 2 & 9 \\ 9 & 8 \end{pmatrix} \right\rangle \simeq \GL_2(\mbz/4\mbz)_{\chi_4 = \ve} \times_{\psi^{(1,1)}} \left\{ \begin{pmatrix} * & * \\ 0 & * \end{pmatrix} \right\}, \\
G_{12,2,1}(12) &= \left\langle  \begin{pmatrix} 5 & 8 \\ 0 & 1 \end{pmatrix}, \begin{pmatrix} 5 & 11 \\ 0 & 11 \end{pmatrix}, \begin{pmatrix} 7 & 6 \\ 3 & 7 \end{pmatrix} \right\rangle \simeq \GL_2(\mbz/4\mbz)_{\chi_4 = \ve} \times_{\psi^{(2,1)}} \left\{ \begin{pmatrix} * & * \\ 0 & * \end{pmatrix} \right\}, \\
G_{12,3,1}(12) &= \left\langle \begin{pmatrix} 5 & 11 \\ 0 & 11 \end{pmatrix}, \begin{pmatrix} 5 & 11 \\ 0 & 7 \end{pmatrix}, \begin{pmatrix} 2 & 1 \\ 9 & 11 \end{pmatrix} \right\rangle\simeq \GL_2(\mbz/4\mbz)_{\chi_4 = \ve} \times_{\psi^{(3,1)}} \left\{ \begin{pmatrix} * & * \\ 0 & * \end{pmatrix} \right\}, \\
G_{12,4,1}(12) &= \left\langle  \begin{pmatrix} 5 & 1 \\ 3 & 2 \end{pmatrix}, \begin{pmatrix} 7 & 6 \\ 0 & 11 \end{pmatrix}, \begin{pmatrix} 7 & 0 \\ 0 & 7 \end{pmatrix} \right\rangle \simeq \pi_{\GL_2}^{-1} \left( \left\langle \begin{pmatrix} 1 & 1 \\ 1 & 0 \end{pmatrix} \right\rangle \right) \times_{\psi^{(4,1)}} \left\{ \begin{pmatrix} * & * \\ 0 & * \end{pmatrix} \right\}, \\
G_{12,4,2}(12) &= \left\langle  \begin{pmatrix} 5 & 1 \\ 3 & 2 \end{pmatrix}, \begin{pmatrix} 11 & 6 \\ 0 & 7 \end{pmatrix}, \begin{pmatrix} 7 & 0 \\ 0 & 7 \end{pmatrix} \right\rangle \simeq \pi_{\GL_2}^{-1} \left( \left\langle \begin{pmatrix} 1 & 1 \\ 1 & 0 \end{pmatrix} \right\rangle \right) \times_{\psi^{(4,2)}} \left\{ \begin{pmatrix} * & * \\ 0 & * \end{pmatrix} \right\}
\end{split}
\end{equation}
and $G_{12,i,k} = \pi_{\GL_2}^{-1}(G_{12,i,k}(12))$; as usual, the representations of the groups on the right-hand are to be understood via the Chinese Remainder Theorem.
In the fibered products $\psi^{(i,k)}$, the underlying homomorphisms are as follows: the maps $\psi_4^{(i,k)}$, are defined by
\begin{equation} \label{definitionoflevel21fiberings}
\begin{split}
&\psi_4^{(i,1)} :  \GL_2(\mbz/4\mbz)_{\chi_4 = \ve} \longrightarrow \{ \pm 1 \}, \quad\quad\, \ker \psi_4^{(i,1)} = \left\langle  \begin{pmatrix} 3 & 2 \\ 0 & 3 \end{pmatrix}, \begin{pmatrix} 3 & 3 \\ 1 & 0 \end{pmatrix}, \begin{pmatrix} 1 & 1 \\ 0 & 3 \end{pmatrix} \right\rangle \quad \left( i \in \{1, 2, 3 \} \right), \\
&\psi_4^{(4,k)} : \pi_{\GL_2}^{-1} \left( \left\langle \begin{pmatrix} 1 & 1 \\ 1 & 0 \end{pmatrix} \right\rangle \right) \rightarrow \{ \pm 1 \}, \quad \psi_4^{(4,k)}(g) = \det g \quad\quad\quad\quad\quad\quad\quad\quad\quad\quad\quad\quad\, \left( k \in \{ 1, 2 \} \right),
\end{split}
\end{equation}
and the maps $\psi_3^{(i,k)}$ are defined by
\[
\begin{split}
\psi_3^{(1,1)} : &\left\{ \begin{pmatrix} * & * \\ 0 & * \end{pmatrix} \right\} \longrightarrow \{ \pm 1 \}, \quad\quad \psi_3^{(1,1)}\left( \begin{pmatrix} a & b \\ 0 & d \end{pmatrix} \right) := \left( \frac{ad}{3} \right), \\
\psi_3^{(2,1)} = \psi_3^{(4,2)} : &\left\{ \begin{pmatrix} * & * \\ 0 & * \end{pmatrix} \right\} \longrightarrow \{ \pm 1 \}, \quad\quad \psi_3^{(2,1)}\left( \begin{pmatrix} a & b \\ 0 & d \end{pmatrix} \right) = \psi_3^{(4,2)}\left( \begin{pmatrix} a & b \\ 0 & d \end{pmatrix} \right) := \left( \frac{d}{3} \right), \\
\psi_3^{(3,1)} = \psi_3^{(4,1)} : &\left\{ \begin{pmatrix} * & * \\ 0 & * \end{pmatrix} \right\} \longrightarrow \{ \pm 1 \}, \quad\quad \psi_3^{(3,1)}\left( \begin{pmatrix} a & b \\ 0 & d \end{pmatrix} \right) = \psi_3^{(4,1)} \left( \begin{pmatrix} a & b \\ 0 & d \end{pmatrix} \right):= \left( \frac{a}{3} \right).
\end{split}
\]
We note that $-I \in G_{12,2,1}, G_{12,3,1}$ and $-I \notin G_{12,1,1}, G_{12,4,k}$, for each $k \in \{ 1, 2 \}$.  We have
\begin{equation} \label{unfiberedlevel12groups}
\begin{split}
\tilde{G}_{12,1}(12) := \tilde{G}_{12,1,1}(12) &\simeq \GL_2(\mbz/4\mbz)_{\chi_4 = \ve} \times \left\{ \begin{pmatrix} * & * \\ 0 & * \end{pmatrix} \right\}, \\
\tilde{G}_{12,4}(12) := \tilde{G}_{12,4,k}(12) &\simeq \pi_{\GL_2}^{-1} \left( \left\langle \begin{pmatrix} 1 & 1 \\ 1 & 0 \end{pmatrix} \right\rangle \right) \times \left\{ \begin{pmatrix} * & * \\ 0 & * \end{pmatrix} \right\} \quad \left( k \in \{1, 2 \} \right).
\end{split}
\end{equation}
As detailed in \cite{sutherlandzywina}, for any elliptic curve $E$ over $\mbq$ with $j$-invariant $j_E$, we have
\[
\begin{split}
\rho_{E,4}(G_\mbq) \, \dot\subseteq \, \GL_2(\mbz/4\mbz)_{\chi_4 = \ve} \; &\Longleftrightarrow \; \exists t_0 \in \mbq \text{ for which } j_E = -t_0^2 + 1728, \\
\rho_{E,4}(G_\mbq) \, \dot\subseteq \, \pi_{\GL_2}^{-1} \left( \left\langle \begin{pmatrix} 1 & 1 \\ 1 & 0 \end{pmatrix} \right\rangle \right) \; &\Longleftrightarrow \; \exists t_0 \in \mbq \text{ for which } j_E = t_0^2 + 1728, \\
\rho_{E,3}(G_\mbq) \, \dot\subseteq \, \left\{ \begin{pmatrix} * & * \\ 0 & * \end{pmatrix} \right\} \; &\Longleftrightarrow \; \exists t_0 \in \mbq \text{ for which } j_E = 27 \frac{(t_0+1)(t_0+9)^3}{t_0^3}.
\end{split}
\]
To obtain models for the modular curves corresponding to the groups in \eqref{unfiberedlevel12groups}, we are led to the equations
\[
-t^2 + 1728 = 27\frac{(s+1)(s+9)^3}{s^3}, \quad\quad t^2 + 1728 = 27\frac{(s+1)(s+9)^3}{s^3},
\]
each of which is a singular model of a conic.  Resolving the singularities in MAGMA, we are led to the substitutions $s = - \frac{27}{u^2}$, $t = \frac{u^4 - 18u^2 - 27}{u}$ for the first equation and $s = \frac{1}{27u^2}$, $t = \frac{1-486u^2 - 19683u^4}{u}$ for the second, and these lead to the $j$-invariants
\[
j_{12,i}(u) := - \frac{(u^2-27)(u^2-3)^3}{u^2} \quad \left( i \in \{1, 2, 3 \} \right), \quad\quad j_{12,4}(u) := \frac{(27u^2+1)(243u^2+1)^3}{u^2}.
\]
(Note that $G_{12,i,1} \subseteq \tilde{G}_{12,1}$ for any $i \in \{1, 2, 3 \}$.). We thus have
\[
\begin{split}
\rho_{E,12}(G_\mbq) \, \dot\subseteq \,  \GL_2(\mbz/4\mbz)_{\chi_4 = \ve} \times \left\{ \begin{pmatrix} * & * \\ 0 & * \end{pmatrix} \right\} \; &\Longleftrightarrow \; \exists u_0 \in \mbq \text{ with } j_E = j_{12,1}(u_0), \\
\rho_{E,12}(G_\mbq) \, \dot\subseteq \, \pi_{\GL_2}^{-1} \left( \left\langle \begin{pmatrix} 1 & 1 \\ 1 & 0 \end{pmatrix} \right\rangle \right) \times \left\{ \begin{pmatrix} * & * \\ 0 & * \end{pmatrix} \right\} \; &\Longleftrightarrow \; \exists u_0 \in \mbq \text{ with } j_E = j_{12,4}(u_0).
\end{split}
\]
We define the coefficients $a_{4;12,i}(u), a_{6;12,i}(u) \in \mbq(u)$ by
\[
\begin{split}
a_{4;12,i}(u) :=& a_{4;3,1}\left( -\frac{27}{u^2} \right), \quad a_{6;12,i} := a_{6;3,1}\left( -\frac{27}{u^2} \right), \quad\quad \left( i \in \{ 1, 2, 3 \} \right), \\
a_{4;12,4}(u) :=& a_{4;3,1}\left( \frac{1}{27u^2} \right), \quad a_{6;12,4} := a_{6;3,1}\left( \frac{1}{27u^2} \right)
\end{split}
\]
and consider the elliptic curves $\mc{E}_{12,i}$ over $\mbq(u,D)$ defined by 
\[
\mc{E}_{12,i} : \; D y^2 = x^3 + a_{4;12,i}(u) x + a_{6;12,i}(u) \quad\quad \left( i \in \{ 1, 2, 3, 4 \} \right).
\]
(Note that $\mc{E}_{12,1} = \mc{E}_{12,2} = \mc{E}_{12,3}$.) 

Applying Lemma \ref{identifyingthesubfieldslevel4lemma} with $\frac{u^4 - 18u^2 - 27}{u}$ substituted for the variable, we find that
\[
\mbq\left( i, \gD_{\mc{E}_{12,i}}^{1/4} \right) = \mbq\left( i, \sqrt{\frac{Du(u^2-27)(u^2-3)}{u^4-18u^2-27}} \right) \quad\quad \left( i \in \{1, 2, 3 \} \right).
\]
By \eqref{descriptionofgroupslevel12} and Corollary \ref{justentanglementequalitycorollary}, we see that, for any specialization $\mc{E}_{12,1}(u_0,D_0)$ that is an elliptic curve,
\[
\begin{split}
\rho_{\mc{E}_{12,1}(u_0,D_0)}(G_\mbq) \, \dot\subseteq \, G_{12,1,1} \; &\Longleftrightarrow \; \mbq\left( \sqrt{\pm \frac{D_0u_0(u_0^2-27)(u_0^2-3)}{u_0^4-18u_0^2-27}} \right) = \mbq\left( \sqrt{-3} \right) \\
&\Longleftrightarrow \; D_0 \in \mp \frac{3u_0(u_0^2-27)(u_0^2-3)}{u_0^4-18u_0^2-27} (\mbq^\times)^2.
\end{split}
\]
By \eqref{descriptionofgroupslevel12} and \eqref{definitionoflevel21fiberings}, and noting that $u \mapsto -\frac{3u(u^2-27)(u^2-3)}{u^4-18u^2-27}$ is an odd function of $u$, we are led to the twist choice
\[
d_{12,1,1}(u) := -\frac{3u(u^2-27)(u^2-3)}{u^4-18u^2-27}
\]
and the model $\mc{E}_{12,1,1}$ over $\mbq(u)$, defined by
\[
\mc{E}_{12,1,1} : \; d_{12,1,1}(u) y^2 = x^3 + a_{4;12,1}(u) x + a_{6;12,1}(u).
\]
Regarding the groups $G_{12,2,1}$ and $G_{12,3,1}$, we apply Lemma \ref{subfieldsoflevel3lemma} with $-\frac{27}{u^2}$ substituted for the variable, obtaining
\[
\begin{split}
\mbq(t,D) \left( \mc{E}_{12,2}[3] \right)^{\ker \psi_3^{(2,1)}} &= \mbq(t,D) \left( \sqrt{\frac{6D(u^2-27)(u^2-3)}{u^4-18u^2-27}} \right), \\
\mbq(t,D) \left( \mc{E}_{12,3}[3] \right)^{\ker \psi_3^{(3,1)}} &= \mbq(t,D) \left( \sqrt{-\frac{2D(u^2-27)(u^2-3)}{u^4-18u^2-27}} \right).
\end{split}
\] 
Noting also that for $i \in \{2, 3 \}$, the Weierstrass coefficients satisfy $a_{4;12,i}(-u) = a_{4;12,i}(u)$ and $a_{6;12,i}(-u) = a_{6;12,i}(u)$, we are thus led to the models $\mc{E}_{12,2,1}$, $\mc{E}_{12,3,1}$ over $\mbq(v,D)$
\[
\begin{split}
\mc{E}_{12,2,1} : \; &D y^2 = x^3 + a_{4;12,2}(6v^2) x + a_{6;12,2}(6v^2), \\
\mc{E}_{12,3,1} : \; &D y^2 = x^3 + a_{4;12,3}(-2v^2) x + a_{6;12,3}(-2v^2).
\end{split}
\]  
For any elliptic curve $E$ over $\mbq$, we have
\[
\begin{split}
\rho_{E}(G_\mbq) \, \dot\subseteq \, G_{12,1,1} \; &\Longleftrightarrow \; \exists u_0 \in \mbq \text{ for which } E \text{ is isomorphic over $\mbq$ to } \mc{E}_{12,1,1}\left( u_0 \right), \\
\rho_{E}(G_\mbq) \, \dot\subseteq \, G_{12,2,1} \; &\Longleftrightarrow \; \exists v_0, D_0 \in \mbq \text{ for which } E \text{ is isomorphic over $\mbq$ to } \mc{E}_{12,2,1}(v_0,D_0), \\
\rho_{E}(G_\mbq) \, \dot\subseteq \, G_{12,3,1} \; &\Longleftrightarrow \; \exists v_0, D_0 \in \mbq \text{ for which } E \text{ is isomorphic over $\mbq$ to } \mc{E}_{12,3,1}(v_0,D_0).
\end{split}
\]
We now find models for the remaining two groups $G_{12,4,1}$, $G_{12,4,2}$.  Applying Lemma \ref{subfieldsoflevel3lemma} with $\frac{1}{27u^2}$ substituted for the variable, we find that
\[
\begin{split}
\mbq(t,D) \left( \mc{E}_{12,4}[3] \right)^{\ker \psi_3^{(4,1)}} &= \mbq(t,D) \left( \sqrt{-\frac{6D(27u^2+1)(243u^2+1)}{19683u^4 + 486u^2 - 1}} \right), \\
\mbq(t,D) \left( \mc{E}_{12,4}[3] \right)^{\ker \psi_3^{(4,2)}} &= \mbq(t,D) \left( \sqrt{\frac{2D(27u^2+1)(243u^2+1)}{19683u^4 + 486u^2 - 1}} \right).
\end{split}
\] 
By \eqref{descriptionofgroupslevel12} and \eqref{definitionoflevel21fiberings}, we obtain the appropriate twist classes by setting each of these fields equal to $\mbq(i)$, which leads to the definitions
\[
d_{12,4,1}(u) := \frac{6(27u^2+1)(243u^2+1)}{19683u^4 + 486u^2 - 1}, \quad\quad d_{12,4,2}(u) := -\frac{2(27u^2+1)(243u^2+1)}{19683u^4 + 486u^2 - 1}.
\]
We define the elliptic curves $\mc{E}_{12,4,k}$ over $\mbq(u)$ by
\[
\mc{E}_{12,4,k} : \; d_{12,4,k}(u) y^2 = x^3 + a_{4;12,4}(u) x + a_{6;12,4}(u) \quad\quad \left( k \in \{ 1, 2 \} \right).
\]
It follows from our discussion that, for any elliptic curve $E$ over $\mbq$,
\[
\begin{split}
\rho_{E}(G_\mbq) \, \dot\subseteq \, G_{12,4,1} \; &\Longleftrightarrow \; \exists u_0 \in \mbq \text{ for which } E \text{ is isomorphic over $\mbq$ to } \mc{E}_{12,4,1}\left( u_0 \right), \\
\rho_{E}(G_\mbq) \, \dot\subseteq \, G_{12,4,2} \; &\Longleftrightarrow \; \exists u_0 \in \mbq \text{ for which } E \text{ is isomorphic over $\mbq$ to } \mc{E}_{12,4,2}\left( u_0 \right).
\end{split}
\]

\medskip

\subsection{The level $m = 14$}

We have 
\[
\mf{G}_{MT}^{\max}(0,14) = \mf{G}_{MT,2}^{\max}(0,14) \sqcup \mf{G}_{MT,3}^{\max}(0,14),
\] 
with
\[
\begin{split}
\mf{G}_{MT,2}^{\max}(0,14) &= \{ G_{14,1,1}, G_{14,2,1}, G_{14,2,1}, G_{14,2,2}, G_{14,3,1}, G_{14,3,2}, G_{14,4,1} \}, \\
\mf{G}_{MT,3}^{\max}(0,14) &= \{ G_{14,5,1}, G_{14,6,1}, G_{14,6,2}, G_{14,7,1}, G_{14,7,2} \}, 
\end{split}
\]
where the groups $G_{14,i,k}(14) \subseteq \GL_2(\mbz/14\mbz)$ for $G_{14,i,k} \in \mf{G}_{MT,2}^{\max}(0,14)$ are given by
\begin{equation} \label{descriptionofgroupslevel14withquadraticfibering}
\begin{split}
G_{14,1,1}(14) &= \left\langle \begin{pmatrix} 9 & 2 \\ 1 & 9 \end{pmatrix}, \begin{pmatrix} 12 & 5 \\ 11 & 6 \end{pmatrix} \right\rangle\simeq \GL_2(\mbz/2\mbz) \times_{\psi^{(1,1)}} \left\{ \begin{pmatrix} \pm 1 & * \\ 0 & * \end{pmatrix} \right\}, \\
G_{12,1,2}(14) &= \left\langle  \begin{pmatrix} 13 & 0 \\ 3 & 1 \end{pmatrix}, \begin{pmatrix} 6 & 1 \\ 9 & 7 \end{pmatrix} \right\rangle \simeq \GL_2(\mbz/2\mbz) \times_{\psi^{(1,2)}} \left\{ \begin{pmatrix} \pm 1 & * \\ 0 & * \end{pmatrix} \right\}, \\
G_{14,2,1}(14) &= \left\langle  \begin{pmatrix} 1 & 11 \\ 4 & 7 \end{pmatrix}, \begin{pmatrix} 9 & 4 \\ 13 & 7 \end{pmatrix} \right\rangle \simeq \GL_2(\mbz/2\mbz) \times_{\psi^{(2,1)}} \left\{ \begin{pmatrix} * & * \\ 0 & \pm 1 \end{pmatrix} \right\}, \\
G_{14,2,2}(14) &= \left\langle \begin{pmatrix} 0 & 9 \\ 9 & 3 \end{pmatrix}, \begin{pmatrix} 1 & 6 \\ 7 & 13 \end{pmatrix} \right\rangle \simeq \GL_2(\mbz/2\mbz) \times_{\psi^{(2,2)}} \left\{ \begin{pmatrix} * & * \\ 0 & \pm 1 \end{pmatrix} \right\}, \\
G_{14,3,1}(14) &= \left\langle  \begin{pmatrix} 9 & 4 \\ 3 & 5 \end{pmatrix}, \begin{pmatrix} 1 & 7 \\ 11 & 6 \end{pmatrix} \right\rangle \simeq \GL_2(\mbz/2\mbz) \times_{\psi^{(3,1)}} \left\{ \begin{pmatrix} a & * \\ 0 & \pm a \end{pmatrix} : \; a \in (\mbz/7\mbz)^\times \right\}, \\
G_{12,3,2}(14) &= \left\langle \begin{pmatrix} 7 & 13 \\ 3 & 0 \end{pmatrix}, \begin{pmatrix} 1 & 12 \\ 3 & 13 \end{pmatrix} \right\rangle \simeq \GL_2(\mbz/2\mbz) \times_{\psi^{(3,2)}} \left\{ \begin{pmatrix} a & * \\ 0 & \pm a \end{pmatrix} : \; a \in (\mbz/7\mbz)^\times \right\}, \\
G_{14,4,1}(14) &= \left\langle  \begin{pmatrix} 9 & 3 \\ 13 & 6 \end{pmatrix}, \begin{pmatrix} 1 & 6 \\ 7 & 3 \end{pmatrix} \right\rangle \simeq \GL_2(\mbz/2\mbz) \times_{\psi^{(4,1)}} \left\{ \begin{pmatrix} * & * \\ 0 & * \end{pmatrix} \right\},
\end{split}
\end{equation}
the groups $G_{14,i,k}(14) \subseteq \GL_2(\mbz/14\mbz)$ for $G_{14,i,k} \in \mf{G}_{MT,3}^{\max}(0,14)$ are given by
\begin{equation} \label{descriptionofgroupslevel14withcubicfibering}
\begin{split}
G_{14,5,1}(14) &= \left\langle  \begin{pmatrix} 3 & 7 \\ 9 & 2 \end{pmatrix}, \begin{pmatrix} 6 & 13 \\ 7 & 11 \end{pmatrix} \right\rangle \simeq \left\langle \begin{pmatrix} 1 & 1 \\ 1 & 0 \end{pmatrix} \right\rangle \times_{\phi^{(5)}} \left\{ \begin{pmatrix} * & * \\ 0 & * \end{pmatrix} \right\}, \\
G_{14,6,1}(14) &= \left\langle  \begin{pmatrix} 5 & 1 \\ 7 & 4 \end{pmatrix}, \begin{pmatrix} 7 & 13 \\ 9 & 10 \end{pmatrix} \right\rangle \simeq \left\langle \begin{pmatrix} 1 & 1 \\ 1 & 0 \end{pmatrix} \right\rangle \times_{\phi^{(6)}} \left\{ \begin{pmatrix} a^2 & * \\ 0 & * \end{pmatrix} : \; a \in (\mbz/7\mbz)^\times \right\}, \\
G_{14,6,2}(14) &= \left\langle \begin{pmatrix} 9 & 11 \\ 7 & 12 \end{pmatrix}, \begin{pmatrix} 7 & 11 \\ 9 & 12 \end{pmatrix} \right\rangle \simeq  \left\langle \begin{pmatrix} 1 & 1 \\ 1 & 0 \end{pmatrix} \right\rangle \times_{\phi^{(6)}} \left\{ \begin{pmatrix} * & * \\ 0 & d^2 \end{pmatrix} : \; d \in (\mbz/7\mbz)^\times \right\}, \\
G_{14,7,1}(14) &= \left\langle \begin{pmatrix} 3 & 1 \\ 1 & 10 \end{pmatrix}, \begin{pmatrix} 9 & 7 \\ 11 & 6 \end{pmatrix} \right\rangle \simeq \left\langle \begin{pmatrix} 1 & 1 \\ 1 & 0 \end{pmatrix} \right\rangle \times_{\phi^{(7)}} \left\{ \begin{pmatrix} a^2 & * \\ 0 & * \end{pmatrix} : \; a \in (\mbz/7\mbz)^\times \right\}, \\
G_{14,7,2}(14) &= \left\langle  \begin{pmatrix} 0 & 3 \\ 9 & 13 \end{pmatrix}, \begin{pmatrix} 9 & 11 \\ 13 & 4 \end{pmatrix} \right\rangle \simeq  \left\langle \begin{pmatrix} 1 & 1 \\ 1 & 0 \end{pmatrix} \right\rangle \times_{\phi^{(7)}} \left\{ \begin{pmatrix} * & * \\ 0 & d^2 \end{pmatrix} : \; d \in (\mbz/7\mbz)^\times \right\},
\end{split}
\end{equation}
and each $G_{14,i,k} = \pi_{\GL_2}^{-1}(G_{14,i,k}(14))$.  In all cases, the representations of the groups on the right-hand are to be understood via the Chinese Remainder Theorem as subgroups of $\GL_2(\mbz/2\mbz) \times \GL_2(\mbz/7\mbz)$.
For each group $G \in \mf{G}_{MT,2}^{\max}(0,14)$, the associated common quotient $\psi_2^{(i,k)}\left( \GL_2(\mbz/2\mbz) \right)$ is a cyclic group of order $2$, whereas for each group $G \in \mf{G}_{MT,3}^{\max}(0,14)$ the associated common quotients $\phi_2^{(i)}\left( \left\langle \begin{pmatrix} 1 & 1 \\ 1 & 0 \end{pmatrix} \right\rangle \right)$ is a cyclic group of order $3$ (i.e. $\phi_2^{(i)}$ is a group isomorphism).  These homomorphisms are defined as follows: the maps $\psi_2^{(i,k)} : \GL_2(\mbz/2\mbz) \longrightarrow \{ \pm 1 \}$ are all equal to the map $\ve$ as in \eqref{defofve}, whereas the maps $\psi_7^{(i,k)}$ are given by
\begin{equation} \label{psifiberingmapsmod7}
\begin{split}
\psi_7^{(1,k)} : \left\{ \begin{pmatrix} \pm 1 & * \\ 0 & * \end{pmatrix} \right\} &\longrightarrow \{ \pm 1 \}; \quad \psi_7^{(1,1)}\left( \begin{pmatrix} a & b \\ 0 & d \end{pmatrix} \right) := \left( \frac{d}{7} \right), \quad \psi_7^{(1,2)}\left( \begin{pmatrix} a & b \\ 0 & d \end{pmatrix} \right) := \left( \frac{a}{7} \right), \\
\psi_7^{(2,k)} : \left\{ \begin{pmatrix} * & * \\ 0 & \pm 1 \end{pmatrix} \right\} &\longrightarrow \{ \pm 1 \}, \quad \psi_2^{(2,1)}\left( \begin{pmatrix} a & b \\ 0 & d \end{pmatrix} \right) := \left( \frac{d}{7} \right), \quad \psi_2^{(2,2)}\left( \begin{pmatrix} a & b \\ 0 & d \end{pmatrix} \right) := \left( \frac{a}{7} \right), \\
\psi_7^{(3,k)} : \left\{ \begin{pmatrix} a & * \\ 0 & \pm a \end{pmatrix} \right\} &\longrightarrow \{ \pm 1 \}, \quad \psi_2^{(3,1)}\left( \begin{pmatrix} a & b \\ 0 & d \end{pmatrix} \right) := \left( \frac{d}{7} \right), \quad \psi_2^{(3,2)}\left( \begin{pmatrix} a & b \\ 0 & d \end{pmatrix} \right) := \left( \frac{a}{7} \right), \\ 
\psi_7^{(4,1)} : \left\{ \begin{pmatrix} * & * \\ 0 & * \end{pmatrix} \right\} &\longrightarrow \{ \pm 1 \}, \quad\quad\quad\quad\;\;\, \psi_2^{(4,1)}\left( g \right) := \left( \frac{\det g}{7} \right).
\end{split}
\end{equation}
For $G \in \mf{G}_{MT,3}^{\max}(0,14)$, the common quotient will be $\left( (\mbz/7\mbz)^\times \right)^2$, which is cyclic of order $3$.  The map $\phi_2^{(i)} :  \left\langle \begin{pmatrix} 1 & 1 \\ 1 & 0 \end{pmatrix} \right\rangle \longrightarrow \left( (\mbz/7\mbz)^\times \right)^2$ is any isomorphism, whereas the maps $\phi_7^{(i)}$ are defined by
\begin{equation} \label{phifiberingmapsmod7}
\begin{split}
\phi_7^{(5)} : &\left\{ \begin{pmatrix} * & * \\ 0 & * \end{pmatrix} \right\} \longrightarrow \left( (\mbz/7\mbz)^\times \right)^2, \quad\quad \;\;\phi_7^{(5)}\left( \begin{pmatrix} a & b \\ 0 & d \end{pmatrix} \right) := (a/d)^2, \\
\phi_7^{(6)} : &\left\{ \begin{pmatrix} * & * \\ 0 & * \end{pmatrix} \right\} \longrightarrow \left( (\mbz/7\mbz)^\times \right)^2, \quad\quad \phi_7^{(6)}\left( \begin{pmatrix} a & b \\ 0 & d \end{pmatrix} \right) := d^2, \\
\phi_7^{(7)} : &\left\{ \begin{pmatrix} * & * \\ 0 & * \end{pmatrix} \right\} \longrightarrow \left( (\mbz/7\mbz)^\times \right)^2, \quad\quad \phi_7^{(7)}\left( \begin{pmatrix} a & b \\ 0 & d \end{pmatrix} \right) := a^2.
\end{split}
\end{equation}
We note that $-I \in G_{14,4,1}, G_{14,5,1}$, whereas $-I \notin G_{14,i,k}$ for $i \in \{1, 2, 3, 6, 7 \}$ and $k \in \{1, 2 \}$.  For $i \in \{ 1, 2, 3 \}$, $G_{14,i,k} \in \mf{G}_{MT,2}^{\max}(0,14)$, and in this case we have
\begin{equation} \label{tildeGlevel14groupsincase2}
\begin{split}
\tilde{G}_{14,1}(14) := \tilde{G}_{14,1,1}(14) = \tilde{G}_{14,1,2}(14) &\simeq \GL_2(\mbz/2\mbz) \times \left\{ \begin{pmatrix} \pm 1 & * \\ 0 & * \end{pmatrix} \right\}, \\
\tilde{G}_{14,2}(14) := \tilde{G}_{14,2,1}(14) = \tilde{G}_{14,2,2}(14) &\simeq \GL_2(\mbz/2\mbz) \times \left\{ \begin{pmatrix} * & * \\ 0 & \pm 1 \end{pmatrix} \right\}, \\
\tilde{G}_{14,3}(14) := \tilde{G}_{14,3,1}(14) = \tilde{G}_{14,3,2}(14) &\simeq \GL_2(\mbz/2\mbz) \times \left\{ \begin{pmatrix} a & * \\ 0 & \pm a \end{pmatrix} : \; a \in (\mbz/7\mbz)^\times \right\}.
\end{split}
\end{equation}
By contrast, for $i \in \{6, 7\}$ and $k \in \{1, 2 \}$, the group $G_{14,i,k} \in \mf{G}_{MT,3}^{\max}(0,14)$, and in this case the fibering does not disappear under $G \mapsto \tilde{G}$.  Indeed, we have
\begin{equation} \label{tildeGlevel14groupsincase3}
\begin{split}
\tilde{G}_{14,6}(14) = \tilde{G}_{14,6,k}(14) &\simeq \left\langle \begin{pmatrix} 1 & 1 \\ 1 & 0 \end{pmatrix} \right\rangle \times_{\phi^{(6)}} \left\{ \begin{pmatrix} * & * \\ 0 & * \end{pmatrix} \right\}, \\
\tilde{G}_{14,7}(14) = \tilde{G}_{14,7,k}(14) &\simeq \left\langle \begin{pmatrix} 1 & 1 \\ 1 & 0 \end{pmatrix} \right\rangle \times_{\phi^{(7)}} \left\{ \begin{pmatrix} * & * \\ 0 & * \end{pmatrix} \right\}.
\end{split}
\end{equation}
Since the arguments are quite different, we will handle separately the levels of $G \in \mf{G}_{MT,2}^{\max}(0,14)$ and $G \in \mf{G}_{MT,3}^{\max}(0,14)$.

\subsubsection{The case $G \in \mf{G}_{MT,2}^{\max}(0,14)$:  quadratic entanglements.}

Let $E$ be an elliptic curve defined over $\mbq$.  By virtue of \eqref{tildeGlevel14groupsincase2}, we have
\begin{equation} \label{level14andlevel7tildeequivalence}
\begin{split}
\rho_{E}(G_\mbq) \, \dot\subseteq \, \tilde{G}_{14,1} \; &\Longleftrightarrow \; \rho_{E,7}(G_\mbq) \, \dot\subseteq \, \left\{  \begin{pmatrix} \pm 1 & * \\ 0 & * \end{pmatrix} \right\}, \\
\rho_{E}(G_\mbq) \, \dot\subseteq \,  \tilde{G}_{14,2} \; &\Longleftrightarrow \; \rho_{E,7}(G_\mbq) \, \dot\subseteq \, \left\{  \begin{pmatrix} * & * \\ 0 & \pm 1 \end{pmatrix} \right\}, \\
\rho_{E}(G_\mbq) \, \dot\subseteq \,  \tilde{G}_{14,3} \; &\Longleftrightarrow \; \rho_{E,7}(G_\mbq) \, \dot\subseteq \, \left\{  \begin{pmatrix} a & * \\ 0 & \pm a \end{pmatrix} : \; a \in (\mbz/7\mbz)^\times \right\}, \\
\rho_{E}(G_\mbq) \, \dot\subseteq \,  \tilde{G}_{14,4} \; &\Longrightarrow \; \rho_{E,7}(G_\mbq) \, \dot\subseteq \, \left\{  \begin{pmatrix} * & * \\ 0 & * \end{pmatrix} \right\}. \\
\end{split}
\end{equation}
Comparing \eqref{level14andlevel7tildeequivalence} with \eqref{tildeGlevel7groups} and considering \eqref{level7jinvariantstatement}, we are thus led to define the $j$-invariants $j_{14,i}(t) \in \mbq(t)$ by
\begin{equation} \label{level7jinvariantn1234}
\begin{split}
j_{14,i}(t) &:= j_{7,i}(t) \quad\quad \left( i \in \{1, 2, 3 \} \right), \\
j_{7,4}(t) &:= \frac{(t^2 + 245t + 2401)^3(t^2 + 13t + 49)}{t^7},
\end{split}
\end{equation}
where $j_{7,i}(t) \in \mbq(t)$ are as in \eqref{tildeGlevel7jinvariants}.  Combining results in \cite{zywina} with \eqref{level14andlevel7tildeequivalence}, for each $E$ over $\mbq$ with $j$-invariant $j_E$, we have
\begin{equation} \label{level14andlevel7modulistatement}
\begin{split}
\rho_{E}(G_\mbq) \, \dot\subseteq \, \tilde{G}_{14,i} \; &\Longleftrightarrow \; \exists t_0 \in \mbq \text{ for which } j_E = j_{14,i}(t_0) \quad \left(i \in \{ 1, 2, 3 \} \right), \\
\rho_{E}(G_\mbq) \, \dot\subseteq \, \tilde{G}_{14,4} \; &\Longrightarrow \rho_{E,7}(G_\mbq) \, \dot\subseteq \, \left\{ \begin{pmatrix} * & * \\ 0 & * \end{pmatrix} \right\}, \\
\rho_{E,7}(G_\mbq) \, \dot\subseteq \, \left\{ \begin{pmatrix} * & * \\ 0 & * \end{pmatrix} \right\} \; &\Longleftrightarrow \; \exists t_0 \in \mbq \text{ for which } j_E = j_{7,4}(t_0).
\end{split}
\end{equation}
We further define the Weierstrass coefficients $a_{4;14,i}(t)$, $a_{6;14,i}(t)$ for $i \in \{1, 2, 3\}$ (resp. for $i=4$ the coefficients $a_{4;7,4}(t)$ and $a_{6;7,4}(t)$) by \eqref{level7jinvariantn1234} and \eqref{defofa4anda6} and the elliptic curves $\mc{E}_{14,i}$ and $\mc{E}_{7,4}$ over $\mbq(t,D)$ by
\begin{equation*} 
\begin{split}
\mc{E}_{14,i} : \; &Dy^2 = x^3 + a_{4;14,i}(t)x + a_{6;14,i}(t) \quad\quad \left( i \in \{1, 2, 3 \} \right), \\
\mc{E}_{7,4} : \; &Dy^2 = x^3 + a_{4;7,4}(t)x + a_{6;7,4}(t),
\end{split}
\end{equation*}
By computing directly the discriminants $\gD_{\mc{E}_{14,i}}$ and $\gD_{\mc{E}_{7,4}}$, we find that
\begin{equation} \label{level14quadraticfieldson2side}
\begin{split}
\mbq\left( \sqrt{\gD_{\mc{E}_{14,i}}} \right) &= \mbq\left( \sqrt{t(t-1)(t^3 - 8t^2 + 5t + 1)} \right) \quad\quad\quad \left( i \in \{ 1, 2 \} \right), \\
\mbq\left( \sqrt{\gD_{\mc{E}_{14,3}}} \right) &= \mbq\left( \sqrt{-7(t^3 - 2t^2 - t + 1)(t^3 - t^2 - 2t + 1)} \right), \\
\mbq\left( \sqrt{\gD_{\mc{E}_{7,4}}} \right) &= \mbq\left( \sqrt{t} \right).
\end{split}
\end{equation}
We now note that the $j$-invariant functions $j_{14,i}(t) \in \mbq(t)$ satisfy
\[
j_{14,i}(t) = j_{7,4}\left( u_i(t) \right) \quad\quad \left( i \in \{1, 2, 3 \} \right),
\]
where
\[
u_1(t) := \frac{49t(t-1)}{t^3 - 8t^2 + 5t + 1}, \quad\quad
u_2(t) := \frac{t^3 - 8t^2 + 5t + 1}{t(t-1)}, \quad\quad
u_3(t) := -\frac{7(t^3 - t^2 - 2t + 1)}{t^3 - 2t^2 - t + 1}.
\]
Applying Lemma \ref{quadraticsubfieldsoflevel7lemma} with $u_i(t)$ in place of $t$, we find that
\[
\begin{split}
&\mbq(t,D)(\mc{E}_{14,1}[7])^{\ker \psi_7^{(1,1)}} = \mbq(t,D)\left( \sqrt{\frac{14D(t^2 - t + 1)(t^6 - 11t^5 + 30t^4 - 15t^3 - 10t^2 + 5t + 1)}{\begin{pmatrix} t^{12} - 18t^{11} + 117t^{10} - 354t^9 + 570t^8 - 486t^7 + \\ 273t^6 -  222t^5 + 174t^4 - 46t^3 - 15t^2 + 6t + 1 \end{pmatrix}}} \right), \\
&\mbq(t,D)(\mc{E}_{14,2}[7])^{\ker \psi_7^{(2,1)}} = \mbq(t,D)\left( \sqrt{\frac{-2D (t^2 - t + 1) (t^6 + 229t^5 + 270t^4 - 1695t^3 + 1430t^2 - 235t + 1)}{\begin{pmatrix} t^{12} - 522t^{11} - 8955t^{10} + 37950t^9 - 70998t^8 + 131562t^7 - \\ 253239t^6 + 316290t^5 - 218058t^4 + 80090t^3 - 14631t^2 + 510t + 1 \end{pmatrix} }} \right) \\
&\mbq(t,D)(\mc{E}_{14,3}[7])^{\ker \psi_7^{(3,1)}} = \mbq(t,D)\left( \sqrt{\frac{14D (t^2 - 3t - 3)  (t^2 - t + 1) (3t^2 - 9t + 5) (3t^4 - 4t^3 - 5t^2 - 2t - 1)  }{ (5t^2 - t - 1) (t^4 - 6t^3 + 17t^2 - 24t + 9) (9t^4 - 12t^3 - t^2 + 8t - 3)}} \right).
\end{split}
\]
Thus, by \eqref{level14quadraticfieldson2side}, \eqref{descriptionofgroupslevel14withquadraticfibering} and \eqref{psifiberingmapsmod7}, we are led to the twist parameters
\[
\begin{split}
d_{14,1,1}(t) &:= \frac{14 \begin{pmatrix} t^{12} - 18t^{11} + 117t^{10} - 354t^9 + 570t^8 - 486t^7 + \\ 273t^6 -  222t^5 + 174t^4 - 46t^3 - 15t^2 + 6t + 1 \end{pmatrix} (t^2 - t + 1)}{t(t-1)(t^3 - 8t^2 + 5t + 1)(t^6 - 11t^5 + 30t^4 - 15t^3 - 10t^2 + 5t + 1)}, \\
d_{14,2,1}(t) &:= \frac{\begin{pmatrix} t^{12} - 522t^{11} - 8955t^{10} + 37950t^9 - 70998t^8 + 131562t^7 - \\ 253239t^6 +  316290t^5 - 218058t^4 + 80090t^3 - 14631t^2 + 510t + 1 \end{pmatrix} (t^2 - t + 1)}{-2t(t-1)(t^3 - 8t^2 + 5t + 1)(t^6 + 229t^5 + 270t^4 - 1695t^3 + 1430t^2 - 235t + 1)}, \\
d_{14,3,1}(t) &:= \frac{-2 (t^2 - 3t - 3)  (t^2 - t + 1) (3t^2 - 9t + 5) (3t^4 - 4t^3 - 5t^2 - 2t - 1) (t^3 - 2t^2 - t + 1) }{ (5t^2 - t - 1) (t^4 - 6t^3 + 17t^2 - 24t + 9) (9t^4 - 12t^3 - t^2 + 8t - 3) (t^3 - t^2 - 2t + 1)},
\end{split}
\] 
and
\[
d_{14,i,2}(t) := -7d_{14,i,1}(t) \quad\quad\quad \left( i \in \{ 1, 2, 3 \} \right). 
\]
Defining the elliptic curves $\mc{E}_{14,i,k}$ over $\mbq(t)$ by
\[
\mc{E}_{14,i,k} : \; d_{14,i,k}(t) y^2 = x^3 + a_{4;14,i}(t) x + a_{6;14,i}(t) \quad\quad \left( i \in \{1, 2, 3 \}, \, k \in \{1, 2 \} \right),
\]
we see that, for each $E$ over $\mbq$, we have
\[
\rho_{E}(G_\mbq) \, \dot\subseteq \, G_{14,i,k} \; \Longleftrightarrow \; \exists t_0 \in \mbq \text{ for which } E \simeq_{\mbq} \mc{E}_{14,i,k}\left( t_0 \right) \quad\quad \begin{pmatrix} i \in \{1, 2, 3 \} \\ k \in \{1, 2 \} \end{pmatrix}.
\]
Finally, we see from \eqref{level14quadraticfieldson2side}, \eqref{descriptionofgroupslevel14withquadraticfibering} and \eqref{psifiberingmapsmod7}, that, defining the elliptic curve $\mc{E}_{14,4,1}$ over $\mbq(u,D)$ by
\[
\mc{E}_{14,4,1} : \; D y^2 = x^3 + a_{4;7,4}(-7u^2) x + a_{6;7,4}(-7u^2),
\]
we have, for any elliptic curve $E$ over $\mbq$,
\[
\rho_{E}(G_\mbq) \, \dot\subseteq \, G_{14,4,1} \; \Longleftrightarrow \; \exists u_0,D_0 \in \mbq \text{ for which } E \simeq_{\mbq} \mc{E}_{14,4,1}\left( u_0,D_0 \right).
\]

\subsubsection{The case $G \in \mf{G}_{MT,3}^{\max}(0,14)$:  cubic entanglements.}

By \eqref{tildeGlevel14groupsincase3}, we have
\[
\tilde{G}_{14,i}(14) \subseteq \left\langle \begin{pmatrix} 1 & 1 \\ 1 & 0 \end{pmatrix} \right\rangle \times \left\{ \begin{pmatrix} * & * \\ 0 & * \end{pmatrix} \right\} \subseteq \GL_2(\mbz/2\mbz) \times \GL_2(\mbz/7\mbz) \quad\quad \left( i \in \{5, 6, 7 \} \right).
\]
As outlined in \cite{zywina}, we have
\begin{equation}
\begin{split} \label{levels2and7models}
\rho_{E,2}(G_\mbq) \, \dot\subseteq \, \left\langle \begin{pmatrix} 1 & 1 \\ 1 & 0 \end{pmatrix} \right\rangle \; &\Longleftrightarrow \; \exists t_0 \in \mbq \text{ for which } j_E = t_0^2 + 1728, \\
\rho_{E,7}(G_\mbq) \, \dot\subseteq \, \left\{ \begin{pmatrix} * & * \\ 0 & * \end{pmatrix} \right\} \; &\Longleftrightarrow \; \exists t_0 \in \mbq \text{ for which } j_E = j_{7,4}(t_0).
\end{split}
\end{equation}
where $j_{7,4}(t)$ as in \eqref{level7jinvariantn1234}.  We further define the Weierstrass coefficients $a_{4;7,4}(t)$, $a_{6;7,4}(t) \in \mbq(t)$ as usual by \eqref{defofa4anda6}, the twist parameters $d_{7,1}'(t), d_{7,2}'(t) \in \mbq(t)$ by
\begin{equation} \label{defofdsub76anddsub77}
\begin{split}
d_{7,1}'(t) &:= \frac{(t^4 - 490t^3 - 21609t^2 - 235298t - 823543)(t^2 + 13t + 49)}{14(t^2 + 245t + 2401)}, \\
d_{7,2}'(t) &:= \frac{(t^4 - 490t^3 - 21609t^2 - 235298t - 823543)(t^2 + 13t + 49)}{- 2(t^2 + 245t + 2401)}
\end{split}
\end{equation}
and the elliptic curves $\mc{E}_{7,6}', \mc{E}_{7,7}'$ over $\mbq(t)$ by
\[
\mc{E}_{7,i}' : \; d_{7,i}'(t) y^2 = x^3 + a_{4;7,4}(t) x + a_{6;7,4}(t) \quad\quad \left( i \in \{ 1, 2 \} \right),
\]
As demonstrated in \cite{zywina}, for any elliptic curve $E$ over $\mbq$ we have
\begin{equation} \label{relevantlevel7twists}
\begin{split}
\rho_{E,7}(G_\mbq) \, \dot\subseteq \, \left\{ \begin{pmatrix} a^2 & * \\ 0 & * \end{pmatrix} : \; a \in (\mbz/7\mbz)^\times \right\} \; &\Longleftrightarrow \; \exists t_0 \in \mbq \text{ for which } E \simeq_\mbq \mc{E}_{7,1}'(t_0), \\
\rho_{E,7}(G_\mbq) \, \dot\subseteq \, \left\{ \begin{pmatrix} * & * \\ 0 & d^2 \end{pmatrix} : \; d \in (\mbz/7\mbz)^\times \right\} \; &\Longleftrightarrow \; \exists t_0 \in \mbq \text{ for which } E \simeq_\mbq \mc{E}_{7,2}'(t_0).
\end{split}
\end{equation}
To first obtain a model for the modular curve corresponding to the level $14$ group 
\[
\left\langle \begin{pmatrix} 1 & 1 \\ 1 & 0 \end{pmatrix} \right\rangle \times \left\{ \begin{pmatrix} * & * \\ 0 & * \end{pmatrix} \right\}, 
\]
we consider the equation
\begin{equation} \label{sexpressionequaltotexpression}
s^2 + 1728 = \frac{(t^2 + 245t + 2401)^3(t^2 + 13t + 49)}{t^7},
\end{equation}
which is a singular conic.  Resolving the singularities in MAGMA, we are led to the substitutions 
\begin{equation} \label{sandtintermsofu}
t = \frac{1}{u^2}, \quad\quad s = \frac{823543u^8 + 235298u^6 + 21609u^4 + 490u^2 - 1}{u}; 
\end{equation}
this gives rise to the $j$-invariant $j_{14}'(u) \in \mbq(u)$, Weierstrass coefficients $a_{4;14}'(u)$, $a_{6,14}'(u) \in \mbq(u)$ and twist parameters $d_{14,1}'(u)$, $d_{14,2}'(u) \in \mbq(u)$, given by
\begin{equation} \label{jinvariantsandWeierstrasscoefficientslevel14}
\begin{split}
j_{14}'(u) &= j_{7,4}(1/u^2) = \frac{(49u^4 + 13u^2 + 1)(2401u^4 + 245u^2 + 1)^3}{u^2}, \\
a_{4;14}'(u) &= a_{4;7,4}(1/u^2), \quad\quad a_{6;14}'(u) = a_{6;7,4}(1/u^2), \\
d_{14,1}'(u) &= d_{7,1}'(1/u^2), \quad\quad\;\; d_{14,2}'(u) = d_{7,2}'(1/u^2),
\end{split}
\end{equation}
(where $d_{7,i}'(t)$ are as in \eqref{defofdsub76anddsub77}) and to the elliptic curves $\mc{E}_{14,5}'$ over $\mbq(u,D)$ and $\mc{E}_{14,6}'$ and $\mc{E}_{14,7}'$ over $\mbq(u)$, defined by
\begin{equation} \label{definitionofmcEsub145primeandmcEsub14iprime}
\begin{split}
&\mc{E}_{14,5}' : \; D y^2 = x^3 + a_{4;14}'(u) x + a_{6;14}'(u), \\
&\mc{E}_{14,*,i}' : \; d_{14,i}'(u) y^2 = x^3 + a_{4;14}'(u) x + a_{6;14}'(u) \quad\quad \left( i \in \{ 1, 2 \} \right).
\end{split}
\end{equation}
By \eqref{levels2and7models} and \eqref{relevantlevel7twists}, for any elliptic curve $E$ over $\mbq$ we have
\begin{equation} \label{twistedmodelslevel14forlateruse}
\begin{split}
\rho_{E,14}(G_\mbq) \, \dot\subseteq \, \left\langle \begin{pmatrix} 1 & 1 \\ 1 & 0 \end{pmatrix} \right\rangle \times \left\{ \begin{pmatrix} * & * \\ 0 & * \end{pmatrix} \right\} \; &\Longleftrightarrow \; \exists u_0, D_0 \in \mbq \text{ for which } E \simeq_\mbq \mc{E}_{14,5}'(u_0,D_0), \\
\rho_{E,14}(G_\mbq) \, \dot\subseteq \,  \left\langle \begin{pmatrix} 1 & 1 \\ 1 & 0 \end{pmatrix} \right\rangle \times  \left\{ \begin{pmatrix} a^2 & * \\ 0 & * \end{pmatrix} \right\} \; &\Longleftrightarrow \; \exists u_0 \in \mbq \text{ for which } E \simeq_\mbq \mc{E}_{14,*,1}'(u_0), \\
\rho_{E,14}(G_\mbq) \, \dot\subseteq \,  \left\langle \begin{pmatrix} 1 & 1 \\ 1 & 0 \end{pmatrix} \right\rangle \times  \left\{ \begin{pmatrix} * & * \\ 0 & d^2 \end{pmatrix} \right\} \; &\Longleftrightarrow \; \exists u_0 \in \mbq \text{ for which } E \simeq_\mbq \mc{E}_{14,*,2}'(u_0).
\end{split}
\end{equation}
Fixing an elliptic curve $E$ over $\mbq$ satisfying $\rho_{E,14}(G_\mbq) \, \dot\subseteq \, \left\langle \begin{pmatrix} 1 & 1 \\ 1 & 0 \end{pmatrix} \right\rangle \times \left\{ \begin{pmatrix} * & * \\ 0 & * \end{pmatrix} \right\}$, we will next identify conditions under which
\[
\mbq(E[2]) \subseteq \mbq(E[7]).
\]
By \eqref{twistedmodelslevel14forlateruse}, such an elliptic curve $E$ must satisfy $E \simeq_\mbq \mc{E}_{14,5}'(u_0,D_0)$ for some $u_0, D_0 \in \mbq$.
We define the homomorphisms $\eta_i : B(7) \longrightarrow \left( (\mbz/7\mbz)^\times \right)^2$ for $i \in \{ 1, 2, 3, 4 \}$ by
\begin{equation} \label{defofetasubis}
\eta_1(g) := \det(g)^2, \quad \eta_2(g) := \phi_7^{(7)}(g), \quad \eta_3(g) := \phi_7^{(6)}(g), \quad \eta_4(g) := \phi_7^{(5)}(g),
\end{equation}
where $\phi_7^{(i)}$ are as in \eqref{phifiberingmapsmod7}.
The following lemma specializes Lemma \ref{gettingattheothercubicfieldslemma} to the present case, allowing us to exhibit explicit polynomials for generators of each of the four cyclic cubic subfields $\mbq(u,D)\left(\mc{E}_{14,5}'[7] \right)^{\ker \eta_i} \subseteq \mbq(u,D)\left( \mc{E}_{14,5}'[7] \right)$.  
Define the polynomials $f_i(x) \in \mbq(u)[x]$ by
\begin{equation} \label{defoffgandh}
\begin{split}
f_1(x) &= x^3 +x^2 - 2x - 1, \\
f_2(x) &= x^3 - T_1(u) x^2 + R_2(u) x - S_3(u), \\
f_3(x) &= x^3 - T_1(u) x^2 + T_2(u) x - R_3(u), \\
f_4(x) &= x^3 - T_1(u) x^2 + T_2(u) x - T_3(u),
\end{split}
\end{equation}
where
\begin{equation} \label{cycliccubiccoefficients}
\begin{split}
T_1(u) &:= 3\left( u^4 + \frac{13}{49}u^2 + \frac{1}{49} \right), \\
R_2(u) &:= 3\left( u^4 + \frac{13}{49}u^2 + \frac{1}{49} \right) \left( u^4 + \frac{13}{49}u^2 + \frac{33}{2401} \right), \\
S_3(u) &:= \left( u^4 + \frac{13}{49}u^2 + \frac{1}{49} \right) \left( u^8 + \frac{26}{49}u^6 + \frac{219}{2401}u^4 + \frac{778}{117649}u^2 + \frac{881}{5764801} \right), \\
T_2(u) &:= 3\left( u^4 + \frac{13}{49}u^2 + \frac{1}{49} \right) \left( u^4 + \frac{13}{49}u^2 - \frac{9}{343} \right), \\
R_3(u) &:= \left( u^4 + \frac{13}{49}u^2 + \frac{1}{49} \right) \left( u^8 + \frac{26}{49}u^6 - \frac{69}{2401}u^4 - \frac{506}{16807}u^2 - \frac{3289}{823543} \right), \\
T_3(u) &:= \left( u^4 + \frac{13}{49}u^2 + \frac{1}{49} \right) \left( u^8 + \frac{26}{49}u^6 - \frac{69}{2401}u^4 - \frac{506}{16807}u^2 - \frac{223}{117649} \right).
\end{split}
\end{equation}
\begin{lemma} \label{explicitcubicsubfieldsatlevel7lemma}
Let $\mc{E}_{14,5}'$ be the elliptic curve over $\mbq(u,D)$ defined by \eqref{definitionofmcEsub145primeandmcEsub14iprime}.  
The four cyclic cubic subfields of $\mbq(u,D)\left( \mc{E}_{14,5}'[7] \right)$ are as follows.  For each $i \in \{1, 2, 3, 4 \}$, the field
\[
\mbq(u,D)\left( \mc{E}_{14,5}'[7] \right)^{\ker \eta_i},
\]
where $\eta_i$ is as in \eqref{defofetasubis}, is equal to the splitting field of $f_i(x)$, where $f_i(x)$ is defined by \eqref{defoffgandh} and \eqref{cycliccubiccoefficients}.  
\end{lemma}
\begin{proof}
Setting $t := 1/u^2$ in Lemma \ref{gettingattheothercubicfieldslemma} and performing variable substitutions of the form $x \mapsto g(u)x$ proves the lemma.
\end{proof}
Turning back to our elliptic curve 
\[
E :  y^2 = x^3 + D_0^2 a_{4;14,5}(u_0)x + D_0^3 a_{6;14,5}(u_0) \quad\quad \left( u_0, D_0 \in \mbq \right)
\]
over $\mbq$ satisfying $\rho_{E,14}(G_\mbq) \, \dot\subseteq \, \left\langle \begin{pmatrix} 1 & 1 \\ 1 & 0 \end{pmatrix} \right\rangle \times \left\{ \begin{pmatrix} * & * \\ 0 & * \end{pmatrix} \right\}$, we have the polynomials
\[
\psi_{E,2}(x) = x^3 + D_0^2 a_{4;14,5}(u_0)x + D_0^3a_{6;14,5}, \quad f_i(x) \quad\quad \left( i \in \{1, 2, 3, 4 \} \right),
\]
(where $f_i(x)$ is as in \eqref{defoffgandh}); we would like to determine conditions under which their splitting fields agree.  
To illustrate how we use the above results to compute explicit models of the modular curves $X_G$ associated to groups $G \in \mf{G}_{MT,3}^{\max}(0,14)$, we will go through the details for the first of the groups in \eqref{descriptionofgroupslevel14withcubicfibering}; the other computations are done similarly.  We wish to find a rational function $g(v) \in \mbq(v)$ so that, defining the elliptic curve $\mc{E}_{14,5,1}$ over $\mbq(v,D)$ by
\begin{equation} \label{mcE1451writtenout}
\mc{E}_{14,5,1} : \; Dy^2 = x^3 + a_{4;14}'(g(v)) x + a_{6;14}'(g(v)),
\end{equation}
we have, for each elliptic curve $E$ over $\mbq$,
\[
\rho_{E}(G_\mbq) \, \dot\subseteq \, G_{14,5,1} \; \Longleftrightarrow \; \exists v_0, D_0 \in \mbq \text{ for which } E \simeq_\mbq \mc{E}_{14,5,1}(v_0,D_0).
\]
In other words, we need 
\[
\begin{split}
\mbq\left( \mc{E}_{14,5,1}[2] \right) &= \mbq\left(\mc{E}_{14,5,1}[7] \right)^{\ker \phi_7^{(5)}} \\
&= \mbq\left(\mc{E}_{14,5,1}[7] \right)^{\ker \eta_4}.
\end{split}
\]
By Lemma \ref{explicitcubicsubfieldsatlevel7lemma} we are led to apply Lemma \ref{settingcubicfieldsequaltoeachotherlemma} to the polynomials
\[
\begin{split}
f_S(x) &= x^3 + a_{4;14}'(u) x + a_{6;14}'(u), \\
f_T(x) &= x^3 - T_1(u) x^2 + T_2(u) x - T_3(u),
\end{split}
\]
where the coefficient functions $T_i(u)$ are as in \eqref{cycliccubiccoefficients} (the twist parameter $D$ occurring in \eqref{mcE1451writtenout} has been absorbed into the variable).  Setting $S_1(u) := 0$, $S_2(u) := a_{4;14}'(u)$ and $S_3(u) := -a_{6;14}'(u)$, the condition \eqref{abcequations} now reads
\begin{equation} \label{newabcequations}
\begin{split}
T_1(u) = &-2a S_2(u) + 3c, \\
T_2(u) = &a^2 S_2(u)^2 - 3abS_3(u) - 4acS_2(u) + b^2 S_2(u) + 3c^2, \\
T_3(u) = &a^3 S_3(u)^2 + a^2b S_2(u) S_3(u) + a^2c S_2(u)^2 - 3abc S_3(u) \\
& - 2ac^2 S_2(u) + b^3 S_3(u) + b^2 c S_2(u) + c^3.
\end{split}
\end{equation}
Setting $c := (T_1(u) + 2aS_2(u))/3$, the first equation above is satisfied.  Inserting this into the second equation, we obtain a quadratic equation of the form
\begin{equation} \label{secondequationconic}
A(u) b^2 + B(u) ab + C(u)a^2 + c(u) = 0.
\end{equation}
Viewing the left-hand side as a quadratic polynomial in $b$ with coefficients in $\mbq(a,u)$, its discriminant 
\[
\gD(a,u) := \left( B(u) a \right)^2 - 4A(u) \left( C(u)a^2 + c(u) \right)
\]
is equal to
\begin{equation} \label{originaldiscriminant}
-\frac{2^{14}3^{11} u^2\left( u^4 + \frac{13}{49}u^2 + \frac{1}{49} \right)^2\left(u^4 + \frac{5}{49}u^2 + \frac{1}{2401}\right)^6}{7^{14} \left(u^8 + \frac{2}{7}u^6 + \frac{9}{343}u^4 + \frac{10}{16807}u^2 - \frac{1}{823543} \right)^6} a^2 + \frac{2^{8}3^{4} \left( u^4 + \frac{13}{49}u^2 + \frac{1}{49} \right)^2 \left( u^4 + \frac{5}{49}u^2 + \frac{1}{2401} \right)^3}{7^{3} \left(u^8 + \frac{2}{7}u^6 + \frac{9}{343}u^4 + \frac{10}{16807}u^2 - \frac{1}{823543} \right)^2};
\end{equation}
we would like this to be a perfect square.  Under the substitution 
\begin{equation} \label{atildefroma}
\tilde{a} := \frac{2^3 3^3 u \left(u^4 + \frac{5}{49}u^2 + \frac{1}{2401}\right)^2 a}{7^3 \left(u^8 + \frac{2}{7}u^6 + \frac{9}{343}u^4 + \frac{10}{16807}u^2 - \frac{1}{823543} \right)^2}, 
\end{equation}
we see that $\gD(a,u)$ is equal modulo $\left( \mbq(a,u)^\times \right)^2$ to 
\[
- 3 \tilde{a}^2 + 7^5 \left( u^4 + \frac{5}{49}u^2 + \frac{1}{2401} \right).
\]
Further substituting $u = \tilde{u}/7$ and setting this expression equal to a perfect square, we arrive at the equation
\[
X^2 + 3 \tilde{a}^2 = 7 \left( \tilde{u}^4 + 5\tilde{u}^2 + 1 \right),
\]
which we view as a conic over $\mbq(u)$.  Since $7 = 2^2 + 3 \cdot 1^2$ and $\tilde{u}^4 + 5\tilde{u}^2 + 1 = (\tilde{u}^2+1)^2 + 3\tilde{u}^2$ are each represented by the left-hand norm form, we discover the $\mbq(u)$-rational point
\[
(X,\tilde{a}) = \left( 2\tilde{u}^2 - 3\tilde{u} + 2, (\tilde{u} + 1)^2 \right)
\]
on this conic.  Projecting from this point, we arrive at the $\mbq(u,v)$-rational point
\begin{equation} \label{defofXandtildea}
\begin{split}
X &:= 2 \frac{\tilde{u}^6 - \frac{3}{2}\tilde{u}^5 - \tilde{u}^4v + 6\tilde{u}^4 - \frac{15}{2}\tilde{u}^3 + \frac{1}{7}\tilde{u}^2v^2 - 5\tilde{u}^2v + 6\tilde{u}^2 - \frac{3}{14}\tilde{u}v^2 - \frac{3}{2}\tilde{u} + \frac{1}{7}v^2 - v + 1}{\tilde{u}^4 - \frac{4}{7}\tilde{u}^2v + 5\tilde{u}^2 + \frac{6}{7}\tilde{u}v + \frac{1}{7}v^2 - \frac{4}{7}v + 1}, \\
\tilde{a} &:= \frac{(\tilde{u}+1)^2\left( \tilde{u}^4 + 5\tilde{u}^2 - \frac{1}{7}v^2 + 1 \right)}{\tilde{u}^4 - \frac{4}{7}\tilde{u}^2v + 5\tilde{u}^2 + \frac{6}{7}\tilde{u}v + \frac{1}{7}v^2 - \frac{4}{7}v + 1} \\
&= \frac{(7u+1)^2\left( 16807u^4 + 1715u^2 - v^2 + 7 \right)}{16807u^4 - 196u^2v + 1715u^2 + 42uv + v^2 - 4v + 7}.
\end{split}
\end{equation}
Inserting this into \eqref{atildefroma}, we find that
\[
a = a(u,v) = \frac{7^3 (7u+1)^2\left( 16807u^4 + 1715u^2 - v^2 + 7 \right) \left(u^8 + \frac{2}{7}u^6 + \frac{9}{343}u^4 + \frac{10}{16807}u^2 - \frac{1}{823543} \right)^2}{2^3 3^3 u \left(u^4 + \frac{5}{49}u^2 + \frac{1}{2401}\right)^2\left( 16807u^4 - 196u^2v + 1715u^2 + 42uv + v^2 - 4v + 7 \right)};
\] 
inserting this into \eqref{originaldiscriminant} (or alternatively working from the expression for $X$ in \eqref{defofXandtildea}), we find that the discriminant $\gD(u,a(u,v)) \in \mbq(u,v)$ of the original quadratic \eqref{secondequationconic} is now equal to
\[
\left( \frac{2^5 3^2\left( u^4 + \frac{13}{49}u^2 + \frac{1}{49} \right) \left( u^4 + \frac{5}{49}u^2 + \frac{1}{2401} \right)\begin{pmatrix} u^6 - \frac{3}{14}u^5 - \frac{1}{49}u^4v + \frac{6}{49}u^4 - \frac{15}{686}u^3 + \frac{1}{16807}u^2v^2 - \frac{5}{2401}u^2v \\ + \frac{6}{2401}u^2 - \frac{3}{235298}uv^2 - \frac{3}{33614}u + \frac{1}{823543}v^2 - \frac{1}{117649}v + \frac{1}{117649} \end{pmatrix}}{7^2 \left( u^4 - \frac{4}{343}u^2v + \frac{5}{49}u^2 + \frac{6}{2401}uv + \frac{1}{16807}v^2 - \frac{4}{16807}v + \frac{1}{2401} \right) \left(u^8 + \frac{2}{7}u^6 + \frac{9}{343}u^4 + \frac{10}{16807}u^2 - \frac{1}{823543} \right)} \right)^2.
\]
In particular, we have $\gD(u,v) = \gD(u,v)^2 \in (\mbq(u,v)^\times)^2$, and applying the quadratic formula to \eqref{secondequationconic}, we obtain functions
\[
b_{\pm}(u,v) := \frac{-B(u) \pm \gD(u,v)}{2A(u)} \in \mbq(u,v).
\]
By construction, the first two equations of \eqref{newabcequations} are satisfied when $a = a(u,v)$, $b = b_{\pm}(u,v)$ and $c := (T_1(u) + 2a(u,v)S_2(u))/3$.  We now insert these rational functions into the third equation in \eqref{newabcequations}.  For instance, choosing to insert $b_{+}(u,v)$, gathering all terms to one side and factoring into irreducible polynomials leads to an equation of the form
\begin{equation} \label{expressioninu}
\left( u^4 + \frac{13}{49}u^2 + \frac{1}{49} \right) f_1(u,v) f_2(u,v) = 0,
\end{equation}
where
\[
\begin{split}
f_1(u,v) := &u^{13} - \frac{1}{7}u^{12} - \frac{1}{7^2}u^{11}v + \frac{15}{7^2}u^{11} - \frac{15}{7^3}u^{10} + \frac{1}{7^5}u^9v^2 - \frac{16}{7^4}u^9v + \frac{78}{7^4}u^9 - \frac{3}{7^6}u^8v^2 
- \frac{1}{7^5}u^8v  \\
&- \frac{78}{7^5}u^8 + \frac{1}{7^8}u^7v^3 + \frac{18}{7^7}u^7v^2 - \frac{87}{7^6}u^7v +
 \frac{155}{7^6}u^7 + \frac{6}{7^9}u^6v^3 - \frac{32}{7^8}u^6v^2 - \frac{10}{7^7}u^6v - \frac{155}{7^7}u^6 \\
 &+ \frac{1}{7^{10}}u^5v^3 + 
 \frac{81}{7^9}u^5v^2 - \frac{172}{7^8}u^5v + \frac{78}{7^8}u^5 +  \frac{55}{7^{11}}u^4v^3 - \frac{81}{7^{10}}u^4v^2 - 
 \frac{27}{7^9}u^4v - \frac{78}{7^9}u^4 + \frac{1}{7^{12}}u^3v^3 \\
 & + \frac{88}{7^{11}}u^3v^2 - \frac{61}{7^{10}}u^3v + \frac{15}{7^{10}}u^3 + \frac{55}{7^{13}}u^2v^3 + \frac{18}{7^{12}}u^2v^2 - \frac{10}{7^{11}}u^2v - \frac{15}{7^{11}}u^2 +  \frac{8}{7^{14}}uv^3 + \frac{15}{7^{13}}uv^2 \\
 & - \frac{6}{7^{12}}uv + \frac{1}{7^{12}}u - \frac{1}{7^{14}}v^3 + \frac{1}{7^{13}}v^2 - 
  \frac{1}{7^{13}}v - \frac{1}{7^{13}}
\end{split}
\]
and $f_2(u,v) \in \mbq[u,v]$ is another polynomial of degree $13$.  The polynomial equation $f_1(u,v) = 0$ defines a singular conic $\mf{S}$ that is found (by a computation in MAGMA) to be birational to the smooth conic
\[
C : \; r^2 - \frac{9}{49}s^2 + 600250r + 32242s + 90392079680 = 0;
\]
we denote by $\tau : \mf{S} \longrightarrow C$ be the birational map produced by our MAGMA calculation.  Projecting from the rational point $(r_0,s_0) = (-300125,184877)$ gives rise to a isomorphism $\mb{P}^1(w) \rightarrow C$ with coordinate functions
\[
\begin{split}
r &= r(w) =  - \frac{300125(9w^2 - 47w - 3920)}{(3w - 79)(3w + 46)}, \\
s &= s(w) =  - \frac{84035(w^2 - 11w + 8624)}{(3w - 79)(3w + 46)}.
\end{split}
\]
Furthermore, composing this isomorphism with $\tau^{-1} : C \longrightarrow \mf{S}$, we obtain
\begin{equation} \label{defofuoft}
u = u_5(w) := - \frac{w^3 + 546w^2 - 10003w - 205807}{13w^3 - 777w^2 - 43414w + 504259}.
\end{equation}
We define the elliptic curve $\mc{E}_{14,5,1}$ over $\mbq(w,D)$ by
\begin{equation} \label{defofmcE145}
\mc{E}_{14,5,1} : \; Dy^2 = x^3 + a_{4,14}'\left(  u_5(w) \right) x + a_{6,14}'\left( u_5(w) \right),
\end{equation}
where the Weierstrass coefficients $a_{4,14}'(u), a_{6,14}'(u) \in \mbq(u)$ are as in \eqref{jinvariantsandWeierstrasscoefficientslevel14} and $u_5(w) \in \mbq(w)$ is as in \eqref{defofuoft}.
It follows from our discussion that, for each elliptic curve $E$ over $\mbq$, we have
\begin{equation} \label{conditiononmcE145}
\exists w_0, D_0 \in \mbq \text{ for which } E \simeq_\mbq \mc{E}_{14,5,1}(w_0,D_0) \; \Longrightarrow \; \rho_{E}(G_\mbq) \, \dot\subseteq \, G_{14,5,1}.
\end{equation}
Now suppose we instead consider the singular conic $\mf{S}_2$ defined by $f_2(u,v) = 0$ (where $f_2(u,v)$ is as in \eqref{expressioninu}), and obtain a degree three function $u_5^{(2)}(w) \in \mbq(w)$ similar to \eqref{defofuoft}, and thus to an elliptic curve $\mc{E}_{14,5,1}^{(2)}$ over $\mbq(w,D)$ as in \eqref{defofmcE145} with $u_5(w)$ replaced by $u_5^{(2)}(w)$.  The elliptic curve $\mc{E}_{14,5,1}^{(2)}$ then satisfies property \eqref{conditiononmcE145}, and by considering the degrees of the associated $j$-invariants, it follows that $u_5^{(2)}(w) = u_5(\mu(w))$, where $\mu(w)$ is a linear fractional transformation, i.e. an automorphism of $\mb{P}^1$.  The same is true if we instead use the function $b_{-}(u,v)$ in place of $b_{+}(u,v)$ and consider any irreducible factor resulting from the third equation of \eqref{newabcequations}.  We have thus established that, for any elliptic curve $E$ over $\mbq$,
\begin{equation*} 
\rho_{E}(G_\mbq) \, \dot\subseteq \, G_{14,5,1} \; \Longleftrightarrow \; \exists w_0, D_0 \in \mbq \text{ for which } E \simeq_\mbq \mc{E}_{14,5,1}(w_0,D_0).
\end{equation*}

The arguments and computations that lead to explicit models associated to the groups $G_{14,6,k}$ and $G_{14,7,k}$ are similar, and we skip most of the details, only summarizing the results.  An analogous computation for the group $\tilde{G}_{14,6}(14)$ involves applying Lemma \ref{settingcubicfieldsequaltoeachotherlemma} to the polynomials
\[
\begin{split}
f_S(x) &:= x^3 + a_{4;14}'(u)x + a_{6;14}'(u) \\
f_T(x) &:= x^3 - T_1(u) x^2 + T_2(u) x - R_3(u),
\end{split}
\]
where $T_1(u)$, $T_2(u)$ and $R_3(u)$ are as in \eqref{cycliccubiccoefficients}.  Continuing as above, we are led to the rational function
\begin{equation} \label{defofusub6}
u_6(w) := - \frac{4(w+2)(w+25)(5w+33)}{71w^3 + 357w^2 - 5243w - 23513},
\end{equation}
and we define the elliptic curve $\mc{E}_{14,6}$ over $\mbq(w,D)$ by
\begin{equation} \label{defofmcEsub146}
\mc{E}_{14,6} : \; D y^2 = x^3 + a_{4;14}'\left( u_6(w) \right) x + a_{6;14}'\left( u_6(w) \right).
\end{equation}
For each elliptic curve $E$ over $\mbq$, we have
\begin{equation} \label{containedinGtilde146}
\rho_{E}(G_\mbq) \, \dot\subseteq \, \tilde{G}_{14,6} \; \Longleftrightarrow \; \exists w_0, D_0 \in \mbq \text{ for which } E \simeq_\mbq \mc{E}_{14,6}(w_0,D_0).
\end{equation}
Considering \eqref{relevantlevel7twists} and \eqref{descriptionofgroupslevel14withcubicfibering}, we are led to define the twist families 
$\mc{E}_{14,6,1}$ and $\mc{E}_{14,6,2}$ over $\mbq(w)$ by
\begin{equation*} 
\mc{E}_{14,6,k} : \; d_{14,k}'(u_6(w)) y^2 = x^3 + a_{4;14}'\left( u_6(w) \right) x + a_{6;14}'\left( u_6(w) \right) \quad\quad \left( k \in \{ 1, 2 \} \right),
\end{equation*}
where $d_{14,k}'(u) := d_{7,k}'(1/u^2)$ and $d_{7,k}'(t)$ is defined by \eqref{defofdsub76anddsub77}.  By \eqref{containedinGtilde146} and \eqref{twistedmodelslevel14forlateruse}, for any elliptic curve $E$ over $\mbq$ we have
\[
\rho_{E}(G_\mbq) \, \dot\subseteq \, G_{14,6,k} \; \Longleftrightarrow \; \exists w_0 \in \mbq \text{ for which } E \simeq_\mbq \mc{E}_{14,6,k}(w_0) \quad\quad \left( k \in \{ 1, 2 \} \right).
\]
Similarly, the group $\tilde{G}_{14,7}(14)$ leads us to apply Lemma \ref{settingcubicfieldsequaltoeachotherlemma} to the polynomials
\[
\begin{split}
f_S(x) &:= x^3 + a_{4;14}'(u)x + a_{6;14}'(u) \\
f_T(x) &:= x^3 - T_1(u) x^2 + R_2(u) x - S_3(u),
\end{split}
\]
where $T_1(u)$, $R_2(u)$ and $S_3(u)$ are as in \eqref{cycliccubiccoefficients}.  Continuing as above, we are led to the rational function
\begin{equation} \label{defofusub7}
u_7(w) :=  \frac{91w^3 - 42w^2 - 28w + 8}{28w(w-2)(5w-2)},
\end{equation}
and we define the elliptic curve $\mc{E}_{14,7}$ over $\mbq(w,D)$ by
\begin{equation} \label{defofmcEsub147}
\mc{E}_{14,7} : \; D y^2 = x^3 + a_{4;14}'\left( u_7(w) \right) x + a_{6;14}'\left( u_7(w) \right).
\end{equation}
For each elliptic curve $E$ over $\mbq$, we have
\begin{equation} \label{containedinGtilde147}
\rho_{E}(G_\mbq) \, \dot\subseteq \, \tilde{G}_{14,7} \; \Longleftrightarrow \; \exists w_0, D_0 \in \mbq \text{ for which } E \simeq_\mbq \mc{E}_{14,7}(w_0,D_0).
\end{equation}
We define the twist families 
$\mc{E}_{14,7,1}$ and $\mc{E}_{14,7,2}$ over $\mbq(w)$ by
\begin{equation*} 
\mc{E}_{14,7,k} : \; d_{14,k}'(u_7(w)) y^2 = x^3 + a_{4;14}'\left( u_7(w) \right) x + a_{6;14}'\left( u_7(w) \right) \quad\quad \left( k \in \{ 1, 2 \} \right),
\end{equation*}
where $d_{14,k}'(u)$ is as before.  By \eqref{containedinGtilde146} and \eqref{twistedmodelslevel14forlateruse}, for any elliptic curve $E$ over $\mbq$ we have
\[
\rho_{E}(G_\mbq) \, \dot\subseteq \, G_{14,7,k} \; \Longleftrightarrow \; \exists w_0 \in \mbq \text{ for which } E \simeq_\mbq \mc{E}_{14,7,k}(w_0) \quad\quad \left( k \in \{ 1, 2 \} \right).
\]

\medskip

\subsection{The level $m = 28$}

We have $\mf{G}_{MT}^{\max}(0,28) = \{ G_{28,1,1}, G_{28,2,1}, G_{28,2,2}, G_{28,3,1}, G_{28,3,2} \}$, where $G_{28,i,k}(28) \subseteq \GL_2(\mbz/28\mbz)$ are given by
\begin{equation*} 
\begin{split}
G_{28,1,1}(28) &= \left\langle  \begin{pmatrix} 5 & 19 \\ 21 & 8 \end{pmatrix}, \begin{pmatrix} 9 & 3 \\ 14 & 1 \end{pmatrix}, \begin{pmatrix} 1 & 10 \\ 0 & 17 \end{pmatrix} \right\rangle \simeq \GL_2(\mbz/4\mbz)_{\chi_4 = \ve} \times_{\psi^{(1,1)}} \left\{ \begin{pmatrix} * & * \\ 0 & * \end{pmatrix} \right\}, \\
G_{28,2,1}(28) &= \left\langle \begin{pmatrix} 5 & 9 \\ 19 & 10 \end{pmatrix}, \begin{pmatrix} 2 & 23 \\ 13 & 5 \end{pmatrix}, \begin{pmatrix} 27 & 10 \\ 14 & 11 \end{pmatrix} \right\rangle\simeq \pi_{\GL_2}^{-1} \left( \left\langle \begin{pmatrix} 1 & 1 \\ 1 & 0 \end{pmatrix} \right\rangle \right) \times_{\psi^{(2,1)}} \left\{ \begin{pmatrix} * & * \\ 0 & * \end{pmatrix} \right\}, \\
G_{28,2,2}(28) &= \left\langle  \begin{pmatrix} 4 & 19 \\ 11 & 9 \end{pmatrix}, \begin{pmatrix} 3 & 12 \\ 2 & 5 \end{pmatrix}, \begin{pmatrix} 22 & 9 \\ 9 & 19 \end{pmatrix} \right\rangle \simeq \pi_{\GL_2}^{-1} \left( \left\langle \begin{pmatrix} 1 & 1 \\ 1 & 0 \end{pmatrix} \right\rangle \right) \times_{\psi^{(2,2)}} \left\{ \begin{pmatrix} * & * \\ 0 & * \end{pmatrix} \right\}, \\
G_{28,3,1}(28) &= \left\langle  \begin{pmatrix} 20 & 1 \\ 5 & 7 \end{pmatrix}, \begin{pmatrix} 15 & 14 \\  2 & 11 \end{pmatrix}, \begin{pmatrix} 7 & 12 \\ 10 & 21 \end{pmatrix} \right\rangle \simeq \pi_{\GL_2}^{-1} \left( \left\langle \begin{pmatrix} 1 & 1 \\ 1 & 0 \end{pmatrix} \right\rangle \right) \times_{\psi^{(3,1)}} \left\{ \begin{pmatrix} * & * \\ 0 & * \end{pmatrix} \right\}, \\
G_{28,3,2}(28) &= \left\langle  \begin{pmatrix} 17 & 5 \\ 3 & 24 \end{pmatrix}, \begin{pmatrix} 26 & 19 \\ 1 & 23 \end{pmatrix}, \begin{pmatrix} 0 & 13 \\ 27 & 5 \end{pmatrix} \right\rangle \simeq \pi_{\GL_2}^{-1} \left( \left\langle \begin{pmatrix} 1 & 1 \\ 1 & 0 \end{pmatrix} \right\rangle \right) \times_{\psi^{(3,2)}} \left\{ \begin{pmatrix} * & * \\ 0 & * \end{pmatrix} \right\},
\end{split}
\end{equation*}
and $G_{28,i,k} = \pi_{\GL_2}^{-1}(G_{28,i,k}(28))$.  In all cases, the representations of the groups on the right-hand are to be understood via the Chinese Remainder Theorem as subgroups of $\GL_2(\mbz/4\mbz) \times \GL_2(\mbz/7\mbz)$, and as before, we are making the usual use of the abbreviation
$
\GL_2(\mbz/4\mbz)_{\chi_4 = \ve} := \left\{ g \in \GL_2(\mbz/4\mbz) : \chi_4(\det g) = \ve(g \mod 2 ) \right\}.
$
In the fibered products $\psi^{(i,k)}$, the underlying homomorphisms are as follows: the maps $\psi_4^{(i,k)}$ are defined by
\begin{equation} \label{definitionoflevel28fiberings}
\begin{split}
\psi_4^{(1,1)} :  \GL_2(\mbz/4\mbz)_{\chi_4 = \ve} \longrightarrow \{ \pm 1 \}, \quad\quad &\ker \psi_4^{(1,1)} = \left\langle  \begin{pmatrix} 3 & 2 \\ 0 & 3 \end{pmatrix}, \begin{pmatrix} 3 & 3 \\ 1 & 0 \end{pmatrix}, \begin{pmatrix} 1 & 1 \\ 0 & 3 \end{pmatrix} \right\rangle, \\
\psi_4^{(i,k)} : \pi_{\GL_2}^{-1} \left( \left\langle \begin{pmatrix} 1 & 1 \\ 1 & 0 \end{pmatrix} \right\rangle \right)  \longrightarrow \left( \mbz/7\mbz \right)^\times, \quad\quad &\ker \psi_4^{(i,k)} = \left\langle  \begin{pmatrix} 1 & 2 \\ 2 & 1 \end{pmatrix}, \begin{pmatrix} 3 & 2 \\ 0 & 3 \end{pmatrix}, \begin{pmatrix} 3 & 0 \\ 0 & 3 \end{pmatrix} \right\rangle \quad \begin{pmatrix} i \in \{2, 3 \} \\ k \in \{1, 2 \} \end{pmatrix},
\end{split}
\end{equation}
Note that, by Corollary \ref{keycorollaryforinterpretationofentanglements}, we need only specify the kernels of these automorphisms, since if we post-compose (say) 
$\psi_4^{(i,k)}$ by an automorphism of $(\mbz/7\mbz)^\times$, the resulting fibered product group would be $\GL_2(\mbz/28\mbz)$-conjugate to the original group.
On the ``$7$ side,'' the maps $\psi_7^{(1,1)}$, $\psi_7^{(2,1)}$, $\psi_7^{(2,2)}$, $\psi_7^{(3,1)}$ and $\psi_7^{(3,2)}$ are defined by
\begin{equation} \label{fiberingsinlevel28case}
\begin{split}
\psi_7^{(1,1)} : &\left\{ \begin{pmatrix} * & * \\ 0 & * \end{pmatrix} \right\} \longrightarrow \{ \pm 1 \}, \quad\quad\quad\;\; \psi_7^{(1,1)}\left( \begin{pmatrix} a & b \\ 0 & d \end{pmatrix} \right) := \left( \frac{ad}{7} \right), \\
\psi_7^{(2,1)} : &\left\{ \begin{pmatrix} * & * \\ 0 & * \end{pmatrix} \right\} \longrightarrow \left( \mbz/7\mbz \right)^\times, \quad\quad \psi_7^{(2,1)}\left( \begin{pmatrix} a & b \\ 0 & d \end{pmatrix} \right) := a^3 d^2, \\
\psi_7^{(2,2)} : &\left\{ \begin{pmatrix} * & * \\ 0 & * \end{pmatrix} \right\} \longrightarrow \left( \mbz/7\mbz \right)^\times, \quad\quad \psi_7^{(2,2)}\left( \begin{pmatrix} a & b \\ 0 & d \end{pmatrix} \right) := d, \\
\psi_7^{(3,1)} : &\left\{ \begin{pmatrix} * & * \\ 0 & * \end{pmatrix} \right\} \longrightarrow \left( \mbz/7\mbz \right)^\times, \quad\quad \psi_7^{(3,1)}\left( \begin{pmatrix} a & b \\ 0 & d \end{pmatrix} \right) := a, \\
\psi_7^{(3,2)} : &\left\{ \begin{pmatrix} * & * \\ 0 & * \end{pmatrix} \right\} \longrightarrow \left( \mbz/7\mbz \right)^\times, \quad\quad \psi_7^{(3,2)}\left( \begin{pmatrix} a & b \\ 0 & d \end{pmatrix} \right) := a^2 d^3.
\end{split}
\end{equation}
We note that $-I \notin G_{28,i,k}$ for each $i \in \{1, 2, 3 \}$ and $k \in \{ 1, 2 \}$, and we have $\tilde{G}_{28,2,1} = \tilde{G}_{28,2,2} = \tilde{G}_{14,6}$, and $\tilde{G}_{28,3,1} = \tilde{G}_{28,3,2} = \tilde{G}_{14,7}$.  In particular, denoting by $\tilde{G}_{28,2}$ the common value of the two groups $\tilde{G}_{28,2,k}$ and by $\tilde{G}_{28,3}$ the common value of the two groups $\tilde{G}_{28,3,k}$, we see that each of the groups $\tilde{G}_{28,2}$ and $\tilde{G}_{28,3}$ have $\GL_2$-level $14$, whereas $\tilde{G}_{28,1} := \tilde{G}_{28,1,1}$ has $\GL_2$-level $28$.  Precisely, we have
\begin{equation} \label{unfiberedlevel28groups}
\begin{split}
\tilde{G}_{28,1}(28) = \tilde{G}_{28,1,1}(28) &\simeq \GL_2(\mbz/4\mbz)_{\chi_4 = \ve} \times \left\{ \begin{pmatrix} * & * \\ 0 & * \end{pmatrix} \right\}, \\
\tilde{G}_{28,2}(14) = \tilde{G}_{28,2,k}(14) &\simeq \left\langle \begin{pmatrix} 1 & 1 \\ 1 & 0 \end{pmatrix} \right\rangle \times_{\phi^{(6)}} \left\{ \begin{pmatrix} * & * \\ 0 & * \end{pmatrix} \right\}, \\
\tilde{G}_{28,3}(14) = \tilde{G}_{28,3,k}(14) &\simeq \left\langle \begin{pmatrix} 1 & 1 \\ 1 & 0 \end{pmatrix} \right\rangle \times_{\phi^{(7)}} \left\{ \begin{pmatrix} * & * \\ 0 & * \end{pmatrix} \right\},
\end{split}
\end{equation}
where the fibering maps $\phi_2^{(6)}$ and $\phi_2^{(7)}$ are isomorphisms $\left\langle \begin{pmatrix} 1 & 1 \\ 1 & 0 \end{pmatrix} \right\rangle \longrightarrow \left( (\mbz/7\mbz)^\times \right)^2$ and $\phi_7^{(6)}$, $\phi_7^{(7)}$ are as in \eqref{phifiberingmapsmod7}.  By \eqref{tildeGlevel14groupsincase3}, it is then natural to define the elliptic curves $\mc{E}_{28,2}$ and $\mc{E}_{28,3}$ over $\mbq(w,D)$ by
\[
\mc{E}_{28,2} := \mc{E}_{14,6}, \quad\quad \mc{E}_{28,3} := \mc{E}_{14,7},
\]
where $\mc{E}_{14,6}$ is as in \eqref{defofmcEsub146} and  $\mc{E}_{14,7}$ is as in \eqref{defofmcEsub147}.  By
\eqref{containedinGtilde146} and \eqref{containedinGtilde147}, for each elliptic curve $E$ over $\mbq$ we have
\[
\begin{split}
\rho_{E}(G_\mbq) \, \dot\subseteq \, \tilde{G}_{28,2} \; &\Longleftrightarrow \; \exists w_0, D_0 \in \mbq \text{ for which } E \simeq_\mbq \mc{E}_{28,2}(w_0,D_0), \\
\rho_{E}(G_\mbq) \, \dot\subseteq \, \tilde{G}_{28,3} \; &\Longleftrightarrow \; \exists w_0, D_0 \in \mbq \text{ for which } E \simeq_\mbq \mc{E}_{28,3}(w_0,D_0).
\end{split}
\]
We note by \eqref{definitionoflevel28fiberings} that, for each $i \in \{2, 3\}$ and $k \in \{1, 2\}$, $\ker \psi_4^{(i,k)} \subseteq \SL_2(\mbz/4\mbz)$, and it follows that
\[
\mbq(u,D)\left( \mc{E}_{28,i}[4] \right)^{\ker \psi_4^{(i,k)}} = \mbq(u,D)\left( \mc{E}_{28,i}[2], i \right) \quad\quad \begin{pmatrix} i \in \{ 2, 3 \} \\ k \in \{ 1, 2 \} \end{pmatrix}.
\]
On the other hand, a computation shows that
\[
\mbq(u,D)\left( \mc{E}_{28,i}[7] \right)^{\ker \psi_7^{(i,k)}} = \mbq(u,D)\left( \mc{E}_{28,i}[7] \right)^{\ker \phi_7^{(4+i)}} \cdot \mbq(u,D) \left( \sqrt{D f_k(u)} \right) \quad\quad \begin{pmatrix} i \in \{ 2, 3 \} \\ k \in \{ 1, 2 \} \end{pmatrix},
\]
where
\[
\begin{split}
f_1(u) &:= \frac{-14(49u^4 + 13u^2 + 1)(2401u^4 + 245u^2 + 1)}{823543u^8 + 235298u^6 + 21609u^4 + 490u^2 - 1}, \\
f_2(u) &:= \frac{2(49u^4 + 13u^2 + 1)(2401u^4 + 245u^2 + 1)}{823543u^8 + 235298u^6 + 21609u^4 + 490u^2 - 1}.
\end{split}
\]
This leads us to define the twist parameters $d_k(u) \in \mbq(u)$ by
\[
\begin{split}
d_1(u) &:= -f_1(u) = \frac{14(49u^4 + 13u^2 + 1)(2401u^4 + 245u^2 + 1)}{823543u^8 + 235298u^6 + 21609u^4 + 490u^2 - 1}, \\
d_2(u) &:= -f_2(u) =  \frac{-2(49u^4 + 13u^2 + 1)(2401u^4 + 245u^2 + 1)}{823543u^8 + 235298u^6 + 21609u^4 + 490u^2 - 1},
\end{split}
\]
and the elliptic curves $\mc{E}_{28,i,k}$ over $\mbq(w)$ by
\[
\begin{split}
&\mc{E}_{28,2,k} : d_k(u_6(w)) y^2 = x^3 + a_{4;14,4}'(u_6(w)) x + a_{6;14,4}'(u_6(w)), \\
&\mc{E}_{28,3,k} : d_k(u_7(w)) y^2 = x^3 + a_{4;14,4}'(u_7(w)) x + a_{6;14,4}'(u_7(w)),
\end{split}
\]
where $a_{4;14,4}'(u)$, $a_{6;14,4}'(u)$ are as in \eqref{jinvariantsandWeierstrasscoefficientslevel14}, $u_6(w)$ is as in \eqref{defofusub6} and $u_7(w)$ is as in \eqref{defofusub7}.  By the above discussion taken together with Corollary \ref{keycorollaryforinterpretationofentanglements}, for each elliptic curve $E$ over $\mbq$, we have
\[
\rho_E(G_\mbq) \, \dot\subseteq \, G_{28,i,k} \; \Longleftrightarrow \; \exists w_0 \in \mbq \text{ for which } E \simeq_\mbq \mc{E}_{28,i,k}(w_0) \quad\quad \begin{pmatrix} i \in \{ 2, 3 \} \\ k \in \{ 1, 2 \} \end{pmatrix}.
\]

To handle the group $G_{28,1,1}$, we first recall that, as detailed in \cite{sutherlandzywina}, one has
\begin{equation} \label{chi4equalsepsilonjinvariant}
\rho_{E,4}(G_\mbq) \, \dot\subseteq \, \GL_2(\mbz/4\mbz)_{\chi_4 = \ve} \; \Longleftrightarrow \; \exists t_0 \in \mbq \text{ for which } j_E = -t_0^2 + 1728.
\end{equation}
This, together with \eqref{unfiberedlevel28groups} and \eqref{level14andlevel7modulistatement}, leads us to the equation
\begin{equation} \label{level28singularjequation}
-s^2 + 1728 = \frac{(t^2 + 245t + 2401)^3(t^2 + 13t + 49)}{t^7} =: j_{7,4}(t)
\end{equation}
which is quite close to \eqref{sexpressionequaltotexpression}.  The replacement $u \mapsto iu$ in \eqref{sandtintermsofu} leads us to the substitutions
\begin{equation} \label{tandsforchi4eqepsilon}
t = - \frac{1}{u^2}, \quad\quad s = \frac{823543u^8 - 235298u^6 + 21609u^4 - 490u^2 - 1}{u},
\end{equation}
which satisfy the equation \eqref{level28singularjequation}.  We set 
\[
j_{28,1}(u) := j_{7,4}(-1/u^2),
\]
where $j_{7,4}(t)$ is as in \eqref{level28singularjequation}, we define $a_{4;28,1}(u), a_{6;28,1}(u) \in \mbq(u)$ as usual by \eqref{defofa4anda6} and finally the elliptic curve $\mc{E}_{28,1}$ over $\mbq(u,D)$ by
\[
\mc{E}_{28,1} : \; Dy^2 = x^3 + a_{4;28,1}(u) x + a_{6;28,1}(u).
\]
By \eqref{unfiberedlevel28groups}, \eqref{chi4equalsepsilonjinvariant} and \eqref{level14andlevel7modulistatement}, we see that, for any elliptic curve $E$ over $\mbq$, we have
\[
\rho_{E}(G_\mbq) \, \dot\subseteq \, \tilde{G}_{28,1} \; \Longleftrightarrow \; \exists u_0, D_0 \in \mbq \text{ for which } E \simeq_\mbq \mc{E}_{28,1}(u_0,D_0).
\]
Applying the substitution \eqref{tandsforchi4eqepsilon} to Lemma \ref{identifyingthesubfieldslevel4lemma} and using \eqref{fiberingsinlevel28case}, we find that
\[
\rho_{E}(G_\mbq) \, \dot\subseteq \, G_{28,1,1} \; \Longleftrightarrow \; 
\begin{matrix} 
\exists u_0, D_0 \in \mbq \text{ for which } E \simeq_\mbq \mc{E}_{28,1}(u_0,D_0) \text{ and } \\
\mbq\left( \sqrt{\frac{D_0u_0(49u_0^4 - 13u_0^2 + 1)(2401u_0^4 - 245u_0^2 + 1)}{823543u_0^8 - 235298u_0^6 + 21609u_0^4 - 490u_0^2 - 1}} \right) = \mbq\left( \sqrt{-7} \right),
\end{matrix}
\]
and this leads us to the condition
\[
D = d_{28,1,1}(u) := \frac{-7u(49u^4 - 13u^2 + 1)(2401u^4 - 245u^2 + 1)}{823543u^8 - 235298u^6 + 21609u^4 - 490u^2 - 1}.
\]
We define the elliptic curve $\mc{E}_{28,1,1}$ over $\mbq(u)$ by 
\[
\mc{E}_{28,1,1} : d_{28,1,1}(u) y^2 = x^3 + a_{4;28,1}(u) x + a_{6;28,1}(u).
\]
We have verified that, for any elliptic curve $E$ over $\mbq$,
\[
\rho_{E}(G_\mbq) \, \dot\subseteq \, G_{28,1,1} \; \Longleftrightarrow \; \exists u_0 \in \mbq \text{ for which } E \simeq_\mbq \mc{E}_{28,1,1}(u_0).
\]

\section{Tables of $j$-invariants and twist parameters associated to $G \in \mf{G}_{MT}^{\max}(0)$} \label{tablesection}

We first define the auxiliary rational functions $f_{m,i}(t)$, $g_{m,i}(t)$ and $h_{m,i}(t)$ in Table \ref{tableofauxiliaryrationalfunctions}.  These functions allow us to express some of the $j$-invariants in the subsequent tables in a reasonably compact way (these functions would otherwise not have fit on the page).  Next, we list the j-invariants $j_{m,i}(t) \in \mbq(t)$ in Table \ref{masterlistofjinvariants}.  Finally, in Table \ref{masterlistoftwistparameters} we define the relevant twist parameters $d_{m,i,k} \in \mbq(t,D)$.

\begin{table}[H]
\[
\begin{array}{|c||c|c|c|} \hline
 & & & \\ [-.75 em]
(m,i) & 
f_{m,i}(t) & 
g_{m,i}(t) & 
h_{m,i}(t) \\
 & & & \\ [-.75 em]
\hline \hline 
& & & \\ [-.75 em]
(9,1) & \frac{(t+3)^3(t+27)}{t} & \frac{729}{t^3-27} &  \frac{-6(t^3 - 9t)}{t^3 + 9t^2 - 9t - 9} \\ 
& & & \\ [-.75 em]
\hline 
& & & \\ [-.75 em]
(9,2) &  & t(t^2 + 9t + 27) & \frac{-3(t^3 + 9t^2 - 9t - 9)}{t^3 + 3t^2 - 9t - 3}  \\
& & & \\ [-.75 em]
\hline 
& & & \\ [-.75 em]
(9,3) &  & t^3 & \frac{3(t^3 + 3t^2 - 9t - 3)}{t^3 - 3t^2 - 9t + 3} \\
& & & \\ [-.75 em]
\hline 
& & & \\ [-.75 em]
(10,3) &  t^3(t^2 + 5t + 40) & \frac{3t^6 + 12t^5 + 80t^4 + 50t^3 - 20t^2 - 8t + 8}{(t-1)^2(t^2 + 3t + 1)^2} & \\
& & & \\ [-.75 em]
\hline 
& & & \\ [-.75 em]
(14,1) &  \frac{(49t^4 + 13t^2 + 1)(2401t^4 + 245t^2 + 1)^3}{t^2} & & \\
& & & \\ [-.75 em]
\hline 
& & & \\ [-.75 em]
(14,2) & \frac{(49t^4 + 13t^2 + 1)(823543t^8 + 235298t^6 + 21609t^4 + 490t^2 - 1)}{-14t^8(2401t^4 + 245t^2 + 1)} &  &  \\
& & & \\ [-.75 em]
\hline
& & & \\ [-.75 em]
(14,5) & & - \frac{t^3 + 546t^2 - 10003t - 205807}{13t^3 - 777t^2 - 43414t + 504259} & \\
& & & \\ [-.75 em]
\hline
& & & \\ [-.75 em]
(14,6) & & - \frac{4(t+2)(5t+33)(t+25)}{71t^3 + 357t^2 - 5243t - 23513} & \\
& & & \\ [-.75 em]
\hline
& & & \\ [-.75 em]
(14,7) & & \frac{91t^3 - 42t^2 - 28t + 8}{28t(t-2)(5t-2)} & \\
& & & \\ 
\hline
\end{array}
\]
\vspace{.1in}
\caption{ Some auxiliary rational functions }
\label{tableofauxiliaryrationalfunctions}
\end{table}

\begin{table}[H]
\[
\begin{array}{|c|c|} 
\hline
& \\ [-.75em]
j_{2,1}(t) := 256\frac{(t+1)^3}{t} & j_{3,1}(t) := 27\frac{(t+1)(t+9)^3}{t^3} \\
& \\ [-.75em]
\hline 
& \\ [-.75em]
j_{4,1}(t) := -t^2 + 1728 & j_{5,1}(t) := \frac{(t^4 - 12t^3 + 14t^2 + 12t + 1)^3}{t^5(t^2 - 11t - 1)} \\
& \\ [-.75em]
\hline
& \\ [-.75em]
j_{5,2}(t) :=  \frac{(t^4 + 228t^3 + 494t^2 - 228t + 1)^3}{t(t^2 - 11t - 1)^5} &  j_{6,1}(t) :=  2^{10}3^3t^3(1-4t^3) \\
& \\ [-.75em]
\hline
& \\ [-.75em]
j_{6,2}(t) := \frac{-27(t^2-9)^3(t^2-1)}{t^6} & j_{6,3}(t) :=  27\frac{(t+1)(t+9)^3}{t^3} \\
& \\ [-.75em]
\hline
& \\ [-.75em]
j_{7,1}(t) := \frac{(t^2 - t + 1)^3(t^6 - 11t^5 + 30t^4 - 15t^3 - 10t^2 + 5t + 1)^3}{t^7(t-1)^7(t^3 - 8t^2 + 5t + 1)}  & \\
& \\ [-.75em]
\hline
& \\ [-.75em]
 j_{7,2}(t) :=  \frac{(t^2 - t + 1)^3(t^6 + 229t^5 + 270t^4 - 1695t^3 + 1430t^2 - 235t + 1)^3}{t(t-1)(t^3 - 8t^2 + 5t + 1)^7} & \\
& \\ [-.75em]
\hline
& \\ [-.75em]
j_{7,3}(t) :=  - \frac{(t^2 - 3t - 3)^3(t^2 - t + 1)^3(3t^2 - 9t + 5)^3(5t^2 - t - 1)^3}{(t^3 - 2t^2 - t + 1)(t^3 - t^2 - 2t + 1)^7} & j_{8,1}(t) := \frac{-4(t^2+2t-2)^3(t^2+10t-2)}{t^4} \\
& \\ [-.75em]
\hline
& \\ [-.75em]
j_{9,1}(t) :=  f_{9,1} \left( g_{9,1} \left( h_{9,1}(t) \right) \right) & j_{9,2}(t) :=  f_{9,1} \left( g_{9,2} \left( h_{9,2}(t) \right) \right) \\
& \\ [-.75em]
\hline
& \\ [-.75em]
j_{9,3}(t) :=  f_{9,1} \left( g_{9,3} \left( h_{9,3}(t) \right) \right) & \\
& \\ [-.75em]
\hline
& \\ [-.75em]
 j_{9,4}(t) :=  \frac{3^7(t^2-1)^3(t^6 + 3t^5 + 6t^4 + t^3 - 3t^2 + 12t + 16)^3(2t^3 + 3t^2 - 3t - 5)}{(t^3 - 3t - 1)^9} & \\
& \\ [-.75em]
\hline
& \\ [-.75em]
j_{10,1}(t) := \frac{(t^4 - 12t^3 + 14t^2 + 12t + 1)^3}{t^5(t^2 - 11t - 1)} & j_{10,2}(t) := \frac{(t^4 + 228t^3 + 494t^2 - 228t + 1)^3}{t(t^2 - 11t - 1)^5} \\
& \\ [-.75em]
\hline
& \\ [-.75em]
j_{10,3}(t) :=  f_{10,3}\left( g_{10,3}(t) \right) & j_{12,1}(t) := - \frac{(t^2 - 27)(t^2 - 3)^3}{t^2} \\
& \\ [-.75em]
\hline
& \\ [-.75em]
j_{12,2}(t) := - \frac{(36t^2 - 27)(36t^2 - 3)^3}{36t^2} & j_{12,3}(t) := - \frac{(4t^2 - 27)(4t^2 - 3)^3}{4t^2} \\
& \\ [-.75em]
\hline
& \\ [-.75em]
j_{12,4}(t) := \frac{(27t^2 + 1)(243t^2+1)^3}{t^2} &  \\
& \\ [-.75em]
\hline
& \\ [-.75em]
j_{14,1}(t) := j_{7,1}(t) & j_{14,2}(t) := j_{7,2}(t) \\
& \\ [-.75em]
\hline
& \\ [-.75em]
j_{14,3}(t) := j_{7,3}(t) & \\
& \\ [-.75em]
\hline
& \\ [-.75em]
j_{14,4}(t) :=  - \frac{(49t^4 - 1715t^2 + 2401)^3(49t^4 - 91t^2 + 49)}{823543t^{14}} & j_{14,5}(t) :=  f_{14,1}(g_{14,5}(t)) \\
& \\ [-.75em]
\hline
& \\ [-.75em]
j_{14,6}(t) :=  f_{14,1}(g_{14,6}(t)) & j_{14,7}(t) :=  f_{14,1}(g_{14,7}(t)) \\
& \\ [-.75em]
\hline
& \\ [-.75em]
j_{28,1}(t) := - \frac{(49t^4 - 13t^2 + 1)(2401t^4 - 245t^2 + 1)^3}{t^2} & j_{28,2}(t) := f_{14,1}(g_{14,6}(t)) \\
& \\ [-.75em]
\hline
& \\ [-.75em]
j_{28,3}(t) :=  f_{14,1}(g_{14,7}(t)) & \\ 
& \\
\hline
\end{array}
\]
\vspace{.1in}
\caption{ $j$-invariants associated to maximal genus zero missing trace groups}
\label{masterlistofjinvariants}
\begin{center}
(see Table \ref{tableofauxiliaryrationalfunctions} for the definitions of $f_{m,i}(t)$, $g_{m,i}(t)$ and $h_{m,i}(t)$)
\end{center}
\end{table}
 
\newpage

\begin{table}[H]
\[
\begin{array}{|c|c|c|} 
\hline
& & \\ [-.75em]
d_{2,1,1} :=  D & d_{3,1,1} :=  \frac{6(t+1)(t+9)}{t^2 - 18t - 27} & d_{3,1,2} :=  - 3 d_{3,1,1} \\
& & \\ [-.75em]
\hline 
& & \\ [-.75em]
d_{4,1,1} := t(t^2-1728) & d_{5,1,1} := - \frac{(t^2+1)(t^4 - 18t^3 + 74t^2 + 18t + 1)}{2(t^4 - 12t^3 + 14t^2 + 12t + 1)} & d_{5,1,2} := 5 d_{5,1,1} \\
& & \\ [-.75em]
\hline
& & \\ [-.75em]
& d_{5,2,1} :=  - \frac{(t^2+1)(t^4 - 522t^3 - 10006t^2 + 522t + 1)}{2(t^4 + 228t^3 + 494t^2 - 228t + 1)} & d_{5,2,2} := 5 d_{5,2,1} \\
& & \\ [-.75em]
\hline
& & \\ [-.75em]
d_{6,1,1} := D & d_{6,2,1} := D & d_{6,3,1} := -t d_{3,1,1} \\
& & \\ [-.75em]
\hline
& & \\ [-.75em]
d_{6,3,2} := 3t d_{3,1,1} & & \\
& & \\
\hline
\end{array}
\]
\[
\begin{array}{|c|}
\hline
\\ [-.75em]
d_{7,1,1} :=  - \frac{t^{12} - 18t^{11} + 117t^{10} - 354t^9 + 570t^8 - 486t^7 + 273t^6 - 222t^5 + 174t^4 - 46t^3 - 15t^2 + 6t + 1}{2(t^2 - t + 1)(t^6 - 11t^5 + 30t^4 - 15t^3 - 10t^2 + 5t + 1)} \\
\\ [-.75em]
\hline
\\ [-.75em]
d_{7,2,1} :=  - \frac{t^{12} - 522t^{11} - 8955t^{10} + 37950t^9 - 70998t^8 + 131562t^7 - 253239t^6 + 316290t^5 - 218058t^4 + 80090t^3 - 14631t^2 + 510t + 1}{2(t^2 - t + 1)(t^6 + 229t^5 + 270t^4 - 1695t^3 + 1430t^2 - 235t + 1)} \\
\\ [-.75em]
\hline
\\ [-.75em]
d_{7,3,1} := \frac{7(t^4 - 6t^3 + 17t^2 - 24t + 9)(3t^4 - 4t^3 - 5t^2 - 2t - 1)(9t^4 - 12t^3 - t^2 + 8t - 3)}{2(t^2 - 3t - 3)(t^2 - t + 1)(3t^2 - 9t + 5)(5t^2 - t - 1)} \\
\\
\hline
\end{array}
\]
\[
\begin{array}{|c|c|c|}
\hline
& & \\ [-.75em]
d_{7,1,2} := -7 d_{7,1,1} & d_{7,2,2} := -7 d_{7,2,1} & d_{7,3,2} :=  -7d_{7,3,1} \\
& & \\ [-.75em]
\hline
\hline
& & \\ [-.75em]
d_{8,1,1} := \frac{(t^2 + 2t - 2)(t^2 + 10t - 2)}{(t^2 + 2)(t^2 + 8t - 2)} & d_{9,1,1} := D & d_{9,2,1} := D \\
& & \\ [-.75em]
\hline
& & \\ [-.75em]
d_{9,3,1} := D & d_{9,4,1} := D & \\
& & \\ [-.75em]
\hline
& & \\ [-.75em]
d_{10,1,1} := \frac{-2t(t^2 - 11t - 1)(t^4 - 12t^3 + 14t^2 + 12t + 1)}{(t^2+1)(t^4 - 18t^3 + 74t^2 + 18t + 1)} & d_{10,1,2} := 5 d_{10,1,1} &  \\
& & \\ [-.75em]
\hline
& & \\ [-.75em]
d_{10,2,1} := \frac{-2t(t^2-11t-1)(t^4 + 228t^3 + 494t^2 - 228t + 1)}{(t^2 + 1) (t^4 - 522t^3 - 10006t^2 + 522t + 1)} & d_{10,2,2} := 5d_{10,2,1} & d_{10,3,1} := D \\
& & \\ [-.75em]
\hline
& & \\ [-.75em]
d_{12,1,1} := -\frac{3t(t^2-27)(t^2-3)}{t^4-18t^2-27} & d_{12,2,1} := D &  d_{12,3,1} :=  D \\
& & \\ [-.75em]
\hline
& & \\ [-.75em]
d_{12,4,1} :=  \frac{6(27t^2+1)(243t^2+1)}{19683t^4 + 486t^2 - 1}  & d_{12,4,2} := -3d_{12,4,1} & \\
& & \\
\hline
\end{array}
\]
\[
\begin{array}{|c|}
\hline
\\ [-.75em]
d_{14,1,1} := \frac{t^{12} - 18t^{11} + 117t^{10} - 354t^9 + 570t^8 - 486t^7 + 273t^6 - 222t^5 + 174t^4 - 46t^3 - 15t^2 + 6t + 1}{14t(t-1)(t^2-t+1)(t^3-8t^2+5t+1)(t^6 - 11t^5 + 30t^4 - 15t^3 - 10t^2 + 5t + 1)} \\
\\ [-.75em]
\hline
\\ [-.75em]
d_{14,2,1} := \frac{t^{12} - 522t^{11} - 8955t^{10} + 37950t^9 - 70998t^8 + 131562t^7 - 253239t^6 + 316290t^5 - 218058t^4 + 80090t^3 - 14631t^2 + 510t + 1}{-2t(t-1)(t^2-t+1)(t^3-8t^2+5t+1)(t^6 + 229t^5 + 270t^4 - 1695t^3 + 1430t^2 - 235t + 1)} \\
\\ [-.75em]
\hline
\\ [-.75em]
d_{14,3,1} :=  \frac{-2(t^4 - 6t^3 + 17t^2 - 24t + 9)(3t^4 - 4t^3 - 5t^2 - 2 - 1)(9t^4 - 12t^3 - t^2 + 8t - 3)}{(t^2 - 3t-3)(t^2 - t + 1)(3t^2 - 9t + 5)(5t^2 - t - 1)(t^3 - 2t^2 - t + 1)(t^3 - t^2 - 2t + 1)} \\
\\
\hline
\end{array}
\]
\[
\begin{array}{|c|c|c|}
\hline
& & \\ [-.75em]
d_{14,1,2} := -7d_{14,1,1} & d_{14,2,2} := -7d_{14,2,1} & d_{14,3,2} := -7d_{14,3,1} \\
& & \\ [-.75em]
\hline
\hline
& & \\ [-.75em]
d_{14,4,1} := D & d_{14,5,1} := D & d_{14,6,1} := f_{14,2}(g_{14,6}(t)) \\
& & \\ [-.75em]
\hline
& & \\ [-.75em]
d_{14,6,2} := -7 d_{14,6,1} & d_{14,7,1} :=  f_{14,2}(g_{14,7}(t)) & d_{14,7,2} := -7 d_{14,7,1} \\
& & \\ [-.75em]
\hline
& & \\ [-.75em]
d_{28,1,1} := \frac{-7t(49t^4 - 13t^2 + 1)(2401t^4 - 245t^2 + 1)}{823543t^8 - 235298t^6 + 21609t^4 - 490t^2 - 1} & d_{28,2,1} := - d_{14,6,1} & d_{28,2,2} := -7 d_{28,2,1} \\
& & \\ [-.75em]
\hline
& & \\ [-.75em]
 d_{28,3,1} :=  - d_{14,7,1}  & d_{28,3,2}  :=  -7 d_{28,3,1} & \\
& & \\ 
\hline
\end{array}
\]

\vspace{.1in}
\caption{ Twist parameters associated to maximal genus zero missing trace groups }
\label{masterlistoftwistparameters}
\end{table}

\section{Concluding remarks} \label{concludingremarkssection}

Serre's open image theorem may be stated as follows: for each number field $K$ and each elliptic curve $E$ defined over $K$, there is a constant $C_{E,K} > 0$ such that, for each prime $\ell \geq C_{E,K}$, we have $\rho_{E,\ell}(G_K) = \GL_2(\mbz/\ell\mbz)$.  Serre's uniformity question asks whether the constant $C_{E,K}$ above may be chosen independent of $E$, i.e. it asks whether, for each number field $K$ there exists a constant $C_K$ such that, for each elliptic curve $E$ defined over $K$ and for each prime $\ell \geq C_K$, we have $\rho_{E,\ell}(G_K) = \GL_2(\mbz/\ell\mbz)$.  In case $K = \mbq$, it is generally believed that Serre's uniformity question has an affirmative answer with $C_\mbq = 41$.  Furthermore, assuming that it does, then as a corollary of \cite[Lemma 4.10]{jonesconductorbound}, we may see that, for any non-CM elliptic curve $E$ over $\mbq$,
$
\ds \level_{\SL_2}\left( \rho_E(G_\mbq) \right)$ divides $ \prod_{\ell \leq C_\mbq} \ell^\infty.
$
This, taken together with Proposition \ref{twotothekproposition} implies that that
\begin{equation} \label{usefulbound}
G \in \mf{G}_{MT}^{\max} \; \Longrightarrow \; \level_{\GL_2}\left( G \right) \text{ divides } \prod_{\ell \leq C_\mbq} \ell^\infty,
\end{equation}
Recent work establishes vertical bounds on the largest possible $n$ for which $\ell^n$ divides $m_G$ when $G = \rho_{E}(G_\mbq)$ when $\ell = 2$ (see \cite{rousedzb}) and it is of general interest to establish such vertical bounds for all $\ell$; clearly such bounds are relevant to the problem of determining the set $\mf{G}_{MT}^{\max}$, even conditionally on an affirmative answer to Serre's uniformity question.  In any event, together with the tables of Cummins-Pauli, \eqref{usefulbound} reduces the determination of $\mf{G}_{MT}^{\max}(g)$ to a targeted computer search, for any fixed genus $g$.
It would be interesting to extend the results in the present paper to $g = 1$, in the spirit of Sutherland-Zywina, to find all modular curves (or Weierstrass models in case $-I \notin G$) that have infinitely many rational points.  We leave this to future work.

\end{document}